\documentclass[opre,nonblindrev]{informs3opt}

\usepackage{endnotes,amsmath,mathabx}
\usepackage{subfigure}
\let\footnote=\endnote

\usepackage{natbib}
\bibpunct[, ]{(}{)}{,}{a}{}{,}%
\def\bibfont{\small}%
\TheoremsNumberedThrough     % Preferred (Theorem 1, Lemma 1, Theorem 2)
\ECRepeatTheorems

\EquationsNumberedThrough    % Default: (1), (2), ...
\usepackage{graphicx}
\usepackage{subfigure}
\usepackage{multirow}

\usepackage[hidelinks,bookmarksnumbered=true]{hyperref}
\usepackage{tikz}
\usepackage{url}
\usepackage{amssymb}
\usepackage{psfrag}
\usepackage{amsmath}
\usepackage{setspace,todonotes,soul}
\newcommand{\exclude}[1]{}
\usepackage{bm}
\usepackage[flushleft]{threeparttable}
\usepackage{algorithm}
\usepackage[noend]{algpseudocode}

\algdef{SE}[DOWHILE]{Do}{doWhile}{\algorithmicdo}[1]{\algorithmicwhile\ #1}%
\algnewcommand{\Or}{\textbf{or}}
\algnewcommand{\And}{\textbf{and}}

\usepackage{thmtools} 
\usepackage{thm-restate}

\usepackage{cleveref}
\usepackage{cases}

\usepackage{pgfplots}
\pgfplotsset{compat=1.18}
\usetikzlibrary{patterns}
\usepackage{titlesec}
\usepackage{appendix}

\def\Re{\mathbb{R}}

\def\hat{\widehat}

\def \Ze{{\mathbb{Z}}}

\def\Re{{\mathbb R}}

\def\Q{Q}

\newcommand{\y}{\bm{y}}

\def\hatx{\hat{\bm x}}
\def\pixhat{{\Pi_{\hat{\bm x}}^s}}
\def\dxhat{{\mathcal{D}_{\hat{\bm x}}^s}}
\def\barx{\widebar{X}}
\def\ix{I_{\hat{\bm x}}}

\DeclareMathOperator{\cone}{cone}
\DeclareMathOperator{\reccone}{rc}
\DeclareMathOperator{\linspace}{lin}

\DeclareMathOperator{\diag}{diag}

\DeclareMathOperator{\conv}{conv}
\DeclareMathOperator{\epi}{epi}

\DeclareMathOperator{\ext}{ext}
\DeclareMathOperator{\co}{co}

\renewcommand*{\qed}{\hfill\ensuremath{\square}}
\newcommand*{\qedA}{\hfill\ensuremath{\diamond}}

\usepackage{enumerate}

\allowdisplaybreaks

\begin{document}

\RUNAUTHOR{Haoyun Deng and Weijun Xie}

\RUNTITLE{On the ReLU Lagrangian Cuts for Stochastic Mixed Integer Programming}

 \TITLE{On the ReLU Lagrangian Cuts for Stochastic Mixed Integer Programming}

\ARTICLEAUTHORS{%
\AUTHOR{Haoyun Deng}
\AFF{H. Milton Stewart School of Industrial and Systems Engineering, Georgia Institute of Technology, Atlanta, GA, USA, \EMAIL{hdeng81@gatech.edu}} %
\AUTHOR{Weijun Xie}
\AFF{H. Milton Stewart School of Industrial and Systems Engineering, Georgia Institute of Technology,  Atlanta, GA, USA, \EMAIL{wxie@gatech.edu}}
} %

\ABSTRACT{We study stochastic mixed integer programs with both first-stage and recourse decisions involving mixed integer variables. A new family of Lagrangian cuts, termed ``ReLU Lagrangian cuts," is introduced by reformulating the nonanticipativity constraints using ReLU functions. These cuts can be integrated into scenario decomposition methods. We show that including ReLU Lagrangian cuts is sufficient to achieve optimality in the original stochastic mixed integer programs. Without solving the Lagrangian dual problems, we derive closed-form expressions for these cuts. Furthermore, to speed up the cut-generating procedures, we introduce linear programming-based methods to enhance the cut coefficients. Numerical studies demonstrate the effectiveness of the proposed cuts compared to existing cut families.
}%


\maketitle

 \section{Introduction}
Consider a two-stage Stochastic Mixed Integer Program (SMIP) with finite support of the form:
\begin{subequations}
\begin{equation}\label{master_problem}
\min_{\bm x\in\Ze^{n_1}\times\Re^{n_2}}\left\{\bm c^\top \bm x + \sum_{s\in [N]} p_s\Q_s(\bm x):\bm A\bm x\geq \bm b\right\},
\end{equation}
where the local recourse function is defined as
\begin{equation}\label{loc_rec}
\Q_s(\bm x) = \min_{\bm y\in\Ze^{m_1}\times\Re^{m_2}}\left\{(\bm q^s)^\top \y: \bm W^s\bm y \geq \bm h^s-\bm T^s\bm x\right\}.
\end{equation}\end{subequations}
Here, we let $\bm x$ and $\bm y$ represent the decisions in the first and second stages, respectively. In the objective function of the first-stage problem \eqref{master_problem}, $\bm c^\top \bm x$ denotes the first-stage cost, and $\mathcal{Q}(\bm x):= \sum_{s\in[N]}p_s\Q_s(\bm x)$ is the expected recourse function that takes the expectation of the second-stage cost (i.e., local recourse function) over finite support. Given a first-stage decision, the local recourse function in scenario $s$, denoted by $\Q_s(\bm x)$, is determined by a mixed integer program \eqref{loc_rec}. In this second-stage problem, $\bm q^s,\bm W^s, \bm T^s, \bm h^s$ are the realizations of random parameters in scenario $s$, which are assumed to be rational \citep{louveaux2003stochastic}. For notational convenience, we let $n=n_1+n_2$ and $m=m_1+m_2$.  
In addition, we make the following assumptions throughout the paper.
\begin{assumption}\label{assumption:rel_comp_rec}
The SMIP \eqref{master_problem} has a relatively complete recourse, that is,
for every feasible first-stage decision 
$\bm x$ in \eqref{master_problem}, the second-stage problem \eqref{loc_rec} is feasible.
\end{assumption}

\begin{assumption}\label{assumption:lb}
There exists a universal lower bound $L$ for local recourse functions.
\end{assumption}

\begin{assumption}\label{assumption:xbounded}
The feasible region of the first-stage problem is nonempty and bounded.
\end{assumption}
Assumptions \ref{assumption:rel_comp_rec} - \ref{assumption:xbounded} are standard in an SMIP setting \citep{louveaux2003stochastic}, which together imply that the SMIP \eqref{master_problem} is always feasible, and its optimal value is bounded from below.
Assumption \ref{assumption:rel_comp_rec} ensures that the conditions of the fundamental theorem of integer programming \citep{schrijver1998theory} hold. Assumption \ref{assumption:lb} is required for the integer L-shaped cuts \citep{laporte1993integer}, which can be used as initial ReLU Lagrangian cuts.
According to Assumption \ref{assumption:xbounded}, we can shift the feasible region such that all variables are nonnegative, i.e., $\barx \subseteq \Ze^{n_1}\times\Re^{n_2}\bigcap ([0, B_1] \times \cdots \times [0, B_n])$.

\subsection{Related literature}
An SMIP is a widely-used modeling paradigm for decision-making under uncertainty, where decisions are made as uncertain parameters are revealed over time. However, solving the SMIP is usually computationally challenging. A common approach to addressing this challenge involves decomposing an SMIP into smaller scenario-based local recourse problems, which can be optimized independently and subsequently provide valid inequalities to the master problem. To accelerate this decomposition method, valid inequalities for the epigraph of the expected recourse function are iteratively added to enhance the solution quality of the master problem. These inequalities are commonly called ``cuts'' in the literature.

Benders decomposition \citep{benders1962partitioning, van1969shaped} is a classic method that solves the linear programming (LP) relaxation of local recourse problems to generate cuts. This approach is only guaranteed to solve SMIP models with continuous second-stage decisions to optimality. Therefore, Benders cuts can only recover the epigraph of LP relaxation of a given local recourse function, since there exists an integrality gap between a local recourse problem and its LP relaxation in each scenario. To resolve this issue and improve the performance of Benders decomposition, one earlier effort is to obtain integer L-shaped cuts \citep{laporte1993integer} for purely binary first-stage decisions. \cite{angulo2016improving} improve this method by introducing an alternative cut-generating strategy and deriving new L-shaped cuts based on the explored recourse function values. In addition, disjunctive programming \citep{sen2005c, sen2006decomposition}, Fenchel cuts \citep{ntaimo2013fenchel,gade2014decomposition}, and Gomory cuts \citep{gade2014decomposition, zhang2014finitely} have also been studied to solve an SMIP with purely binary or integer first-stage variables. However, these methods struggle with problems of general mixed integer first-stage variables, even when the underlying distribution has only one scenario.

Building on Benders decomposition, \cite{zou2019stochastic} propose Lagrangian cuts and strengthened Benders cuts, which can be viewed as special cases of the Benders dual decomposition (BDD) framework later introduced by \cite{rahmaniani2020benders}. Unlike traditional Benders decomposition, which takes the duals of LP relaxations of the local recourse problems, the BDD method introduces a copy of the first-stage variables in each local recourse problem and enforces nonanticipativity constraints to ensure that these copies are equal, similar to the dual decomposition method \citep{caroe1999dual}. The key difference between Benders decomposition and BDD is that BDD retains the integrality of second-stage variables when taking the Lagrangian dual with respect to the nonanticipativity constraints. Strengthened Benders cuts are derived by optimizing Lagrangian functions, with Lagrangian dual multipliers equal to the Bender cuts coefficients. These cuts are parallel to the original Benders cuts but are shifted upward to have higher vertical intercepts by enforcing the integrality constraints in the inner minimization problems of the Lagrangian duals. For the Lagrangian cuts, the cut coefficients are optimal dual multipliers. This family of cuts is exact for binary first-stage variables, which guarantees convergence of the cutting plane method to solve the SMIP problem to optimality. It has also been shown to be sufficient to recover the convex hulls of epigraphs of local recourse functions \citep{chen2022generating}. However, generating such cuts can be computationally demanding. Existing methods rely on first-order approaches and require solving multiple MIPs that correspond to the inner minimization problems. Moreover, the Lagrangian dual problem may have multiple optimal solutions, though not all are effective Lagrangian cuts. For instance, \cite{bansal2024computational} demonstrate that the integer L-shaped cut-- known for its weak global approximation and resulting slow convergence-- is, in fact, a Lagrangian cut. The existing literature improves the implementation of Lagrangian cuts primarily in three areas: (i) Accelerating the solution procedure of Lagrangian dual problems, (ii) selecting cuts with specific desirable properties, and (iii) improving the decomposition and cutting plane framework by incorporating additional cut families.

\cite{rahmaniani2020benders} propose a three-phase implementation strategy in which Lagrangian cuts are generated at the final stage by heuristically solving Lagrangian duals using an inner approximation. They also suggest partially relaxing integrality constraints or fixing certain variables in each local recourse problem and solving the Lagrangian dual to $\epsilon$-optimality. \cite{chen2022generating} provide a new formulation that can be used to derive both optimality and feasibility cuts, with the cut coefficients restricted to the span of previous Benders cuts' coefficients under certain normalization. Recently, in the context of multistage SIMPs, \cite{bansal2024computational} conduct computational studies on the alternating cut procedure for Lagrangian cuts and other valid cut families to improve overall performance on multistage SMIP problems. \cite{fullner2024lipschitz} generate tight Lagrangian cuts with bounded coefficients using regularized local recourse problems, and \cite{fullnernew} extend existing concepts in Benders decomposition literature to derive facet-defining, Pareto-optimal, or deep cuts using proper normalization techniques.

However, linear cuts, such as Lagrangian cuts, can at most recover the convex envelopes of local recourse functions. In addition, the expectation of the convex envelopes of the local recourse functions is no larger than and can be strictly less than the convex envelope of the expected recourse function in certain first-stage decisions \citep{van2023converging}. To address this issue, \cite{zou2019stochastic} propose to approximate SMIP problems by stochastic integer programs with purely binary first-stage decisions and prove that under some assumptions, such as Lipschitz continuity, the number of binary variables required can be bounded based on the desired precision of an optimal solution. Other methods directly solve the SMIP problems. For example, \cite{ahmed2022stochastic} introduce reverse norm cuts, which leverage the Lipschitz continuity of local recourse functions, and augmented Lagrangian cuts. The tightness of the augmented Lagrangian cuts is ensured by the strong duality of the augmented Lagrangian duals, as shown in \cite{feizollahi2017exact}.
\cite{zhang2022stochastic} derive generalized conjugacy cuts based on regularized value functions.  \cite{van2023converging} derive scaled cuts for the expected recourse function and recently extend them on multistage SMIP problems \citep{romeijnders2024benders}.

\subsection{Summary of contributions}
The main contributions of this work are summarized below.
\begin{enumerate}[1)]
\item We introduce a new family of nonlinear cuts, referred to as ``ReLU Lagrangian cuts." These cuts are effective for solving a two-stage SMIP that involves general mixed integer decisions in the first stage. Through establishing a strong duality theory, we show that the ReLU Lagrangian cuts are tight, enabling the recovery of the epigraphs for both the local recourse functions and the expected recourse function.

\item We compare ReLU Lagrangian cuts with existing cut families from two perspectives: (i) ReLU Lagrangian cuts provide outer approximations of local recourse epigraphs and their convex hulls that are at least as strong as those from existing methods, requiring fewer iterations before the cutting plane method terminates; and (ii) since existing cuts are special ReLU Lagrangian cuts, they can serve as starting points for generating the strong ReLU Lagrangian cuts.

\item We show the equivalence between traditional Lagrangian cuts and ReLU Lagrangian cuts for purely binary first-stage decisions. We propose a cut generation scheme that begins with integer L-shaped cuts and strengthens them by solving LPs. This approach overcomes a limitation in the existing literature, where obtaining ideal coefficients for Lagrangian cuts often requires solving multiple MILPs within an iterative procedure.
\item For pure integer first-stage decisions, we theoretically compare two alternative approaches for generating initial ReLU Lagrangian cuts: (i) focusing on the original space to generalize integer L-shaped cuts and (ii) using the binary expansion technique from \cite{zou2019stochastic}. For mixed integer first-stage decisions, we propose both binary search and closed-form methods to derive the initial ReLU Lagrangian cuts. In both cases, the cuts are further strengthened using an LP-based approach.
\end{enumerate}

\paragraph{Organization.} In Section \ref{sec_ReLU_cuts}, we review the ordinary Lagrangian cuts and introduce the ReLU Lagrangian cuts. In Section \ref{sec_binary} and Section \ref{sec_general_mixed}, we study the properties of ReLU Lagrangian cuts for SMIPs with purely binary and general mixed integer first-stage decisions, respectively. Section \ref{sec_numerical} shows the numerical evidence of the effectiveness of the proposed ReLU Lagrangian cuts, and Section \ref{sec_con} concludes the paper.\looseness=-1

\paragraph{Notation.}
For a given function $f$ defined with domain $S$, let $\epi_S(f):= \{(\bm x, \theta)\in S\times \Re: \theta \geq f(\bm x)\}$ denote its epigraph and $\conv(\epi_S(f))$ denote the convex hull of its epigraph. For a convex set $S$, let $\mathrm{ext}(S)$ denote the set of its extreme points. Let us define the sets $\barx = \{\bm{x} \in \Ze^{n_1} \times \Re^{n_2} : \bm{A}\bm{x} \geq \bm{b}\}$, $X = \{\bm{x} \in \{0,1\}^n : \bm{A}\bm{x} \geq \bm{b}\}$, and $X^{LP} = \{\bm{x} \in [0,1]^n : \bm{A}\bm{x} \geq \bm{b}\}$. Given a positive integer $\tau$ and a nonnegative integer $\ell\leq \tau$, we let $[\tau]=\{1,\ldots,\tau\}$ and $[\ell,\tau]=\{\ell,\ell+1,\ldots,\tau\}$. A variable is bold when it is a vector.

\section{ReLU Lagrangian Cuts}\label{sec_ReLU_cuts}
In this section, we introduce a new family of cuts, termed ReLU Lagrangian cuts, which generalize ordinary Lagrangian cuts \citep{zou2019stochastic}. Throughout this paper, unless otherwise specified, each cut is a local cut; that is, it is derived based on a given scenario $s \in [N]$.

\subsection{Preliminary: ordinary Lagrangian cuts}

Lagrangian cuts, first introduced by \cite{zou2019stochastic}, are derived by dualizing the nonanticipativity constraints. Namely, for a given feasible first-stage decision $\hat{\bm{x}}\in\barx$ and a scenario $s\in[N]$, we have
\begin{align}
\Q_s(\hat{\bm{x}}) &= \inf_{\bm{x}} \left\{\Q_s(\bm{x}) : \bm{x} = \hat{\bm{x}}, \ \bm{x} \in \barx \right\}= \inf_{\bm{x}} \left\{\theta : (\bm{x}, \theta) \in \epi_{\barx}(\Q_s), \ \bm{x} = \hat{\bm{x}} \right\},\label{lag_primal} 
\end{align}
where the epigraphical set
\begin{equation*}
\begin{aligned}
\epi_{\barx}(\Q_s)=&\left\{(\bm x, \theta): \theta \geq \min_{\bm y}\left\{(\bm q^s)^\top \y: \bm T^s\bm x+\bm W^s\bm y \geq \bm h^s, \bm y\in\Ze^{m_1}\times\Re^{m_2}\right\}, \bm x\in \barx\right\}\\
=&\left\{(\bm x, \theta):\exists \bm y\in\Ze^{m_1}\times\Re^{m_2}, \theta \geq (\bm q^s)^\top \y,  \bm T^s\bm x+\bm W^s\bm y \geq \bm h^s,  \bm x\in \barx\right\}.
\end{aligned}
\end{equation*}
We observe that set $\epi_{\barx}(\Q_s)$ can be described by linear inequalities with integrality constraints. 
Therefore, its convex hull $\conv(\epi_{\barx}(\Q_s))$ is a polyhedron. Following theorem 1 from \cite{geoffrion1974lagrangean},  We further obtain that
\begin{subequations}\label{lag_cut_eqs2}
\begin{align}
\Q_s(\hat{\bm{x}}) 
&\geq \min_{\bm{x}} \left\{\theta : (\bm{x}, \theta) \in \conv(\epi_{\barx}(\Q_s)), \bm{x} = \hat{\bm{x}} \right\} \label{primal0} \\
&
= \max_{\bm{\pi}} \left\{\min_{\bm{x}} \left\{\Q_s(\bm{x}) + \bm{\pi}^\top (\hat{\bm{x}} - \bm{x}) : \bm{x} \in \barx \right\}\right\},
\label{lag_prob}
\end{align}
\end{subequations}
where 
the equality follows from the strong duality of linear programming and
the fact that minimizing a linear function over a set is equivalent to minimizing over its convex hull.
Here, problem \eqref{lag_prob} is a Lagrangian dual of the equivalent formulation \eqref{lag_primal} for a given first-stage decision. The LP \eqref{primal0} is the primal characterization of the Lagrangian dual. This LP is feasible and bounded since $\hatx\in\barx$ and $\Q_s(\bm x)\geq L$ for all $\bm x \in \barx$. Thus, there always exists an optimal dual multiplier $\bm\pi$ to \eqref{lag_prob}. \cite{zou2019stochastic} further prove that the inequality of \eqref{primal0} holds at equality when the first-stage variables are binary and introduce the following Lagrangian cuts.
\begin{definition}[Lagrangian Cuts]
    Given a feasible first-stage decision $\hat{\bm{x}}\in \barx$, let $\hat{\bm{\pi}}$ be optimal to the outer maximization problem of \eqref{lag_prob}. A Lagrangian cut takes the form
    \begin{equation}\label{lag_cut}
\theta \geq \mathcal{L}_s(\hat{\bm{\pi}}; \hat{\bm{x}}) + \hat{\bm{\pi}}^\top (\bm{x} - \hat{\bm{x}}),
\end{equation}
where $\mathcal{L}_s(\bm{\pi}; \hat{\bm{x}}) := \min_{\bm{x}} \left\{\Q_s(\bm{x}) + \bm{\pi}^\top (\hat{\bm{x}} - \bm{x}) : \bm{x} \in \barx \right\}$.
\end{definition}

By preserving the integrality constraints in each local recourse problem, the Lagrangian dual \eqref{lag_prob} yields a stronger lower bound for the local recourse function value compared to the bound obtained from the LP relaxation used to derive a Benders cut. Under certain conditions, this lower bound coincides with the local recourse function value, and the resulting cut is referred to as a tight cut.
\begin{definition}[Tight Cuts]
A cut generated for a function $f$ at the incumbent solution $\hatx$ is tight if the cut's corresponding hyperplane passes through the point $(\hatx,f(\hatx))$.
\end{definition}
It is known from \cite{zou2019stochastic} that when the first-stage decision variables are binary, the Lagrangian cuts are tight. However, this property does not extend to general mixed integer first-stage stochastic programs. In fact, the primal characterizations \eqref{primal0} of the Lagrangian dual also define the convex envelopes of the local recourse functions, as formally proved in \cite{fullner2024lipschitz}. Furthermore, \cite{chen2022generating} show that Lagrangian cuts are sufficient to describe the convex hulls of epigraphs of the local recourse functions, which we refer to as the ``local convex hulls."

\begin{proposition}\label{prop:conv_env}[Theorem 3.9, \citealt{fullner2024lipschitz}]
    The primal characterization $$\inf_{\bm{x}} \left\{\theta : (\bm{x}, \theta) \in \conv(\epi_{\barx}(\Q_s)),  \bm{x} = \hat{\bm{x}} \right\}=\mathrm{co}(\Q_s(\hatx)),$$
where the convex envelope of $\Q_s$, denoted by $\mathrm{co}(\Q_s): \conv(\barx)\rightarrow\Re$, is defined as $\mathrm{co}(\Q_s)(\bm x) = \sup\{g(\bm x): g\text{ is convex and } g(\bm z)\leq \Q_s(\bm z), \forall \bm z\in\barx\}$.
\end{proposition}

\begin{proposition}\label{thm:epi_description}[Proposition 2, Theorem 3, \citealt{chen2022generating}]
Let us 
define $\dxhat=\left\{\bm \pi: \bm \pi \text{ is optimal to \eqref{lag_prob}} \right\}$ for any feasible first-stage decision $\hatx\in\barx$. Then, the local convex hull $\conv(\epi_{\barx}(\Q_s)) = \{(\bm x, \theta)\in\conv(\barx)\times\Re: \theta \geq \mathcal{L}_s(\bm \pi; \hatx)+ \bm\pi^\top (\bm x-\hat{\bm x}),  \forall \hat{\bm x}\in \barx, \bm \pi \in \dxhat \}$. 
\end{proposition}

The following corollaries provide insights for Lagrangian cuts.
\begin{corollary}\label{cor_1_prop1}
   Given a first-stage solution $\hatx\in\barx$, a dual multiplier $\bm\pi\in\dxhat$ if and only if it satisfies the optimality condition:
\begin{equation}\label{opt_cond_general}
\mathcal{L}_s(\bm\pi;\hatx) = \min_{\bm x}\{\Q_s(\bm x)+\bm\pi^\top(\hatx-\bm x):\bm x\in\barx\} \geq \mathrm{co}(\Q_s)(\hatx).
\end{equation} 
\end{corollary}
The result in \Cref{cor_1_prop1} follows directly from \Cref{prop:conv_env} and the strong duality between \eqref{primal0} and \eqref{lag_prob}.

\begin{restatable}{corollary}{tight}\label{coro:tight_lag}
A Lagrangian cut generated at $\hatx\in\barx$ is tight if and only if there exists $\bm\alpha\in\Re^n$, $\bm\alpha\neq\bm{0}$, such that $\bm\alpha^\top \bm x + \theta \geq \bm\alpha^\top \hatx + \Q_s(\hatx)$ for all $(\bm x,\theta)\in\conv(\epi_{\barx}(\Q_s))$.
\end{restatable}
\begin{proof}
    See Appendix \ref{pf:coro_tight_lag}. \qed
\end{proof}
When $(\hatx, \Q_s(\hatx))$ is an extreme point of $\conv(\epi_{\barx}(\Q_s))$, such a supporting hyperplane always exists, and a Lagrangian cut derived at this point is tight. The following corollary shows that a local convex hull can be characterized using tight Lagrangian cuts generated at its extreme points.  Since the set $\conv(\epi_{\barx}(\Q_s))$ is a polyhedron and $\barx$ is nonempty and bounded, we have that 
\begin{corollary}\label{coro}
Set $\conv(\epi_{\barx}(\Q_s)) = \conv\{(\bm x^k, \Q_s(\bm x^k)): k\in K\}+\cone\{(\bm{0},1)\}$, where $\{(\bm x^k, \Q_s(\bm x^k))\}_{k\in K}$ are the extreme points of $\conv({\epi_{\barx}(\Q_s)})$. In addition, the local convex hull can be represented as $\conv(\epi_{\barx}(\Q_s)) = \{(\bm x,\theta)\in\conv(\barx)\times\Re:\theta\geq \Q_s(\bm x^k) + \bm\pi^\top(\bm x-\bm x^k), \forall k\in K, \bm\pi\in\mathcal{D}_{\bm x^k}^s\}$.
\end{corollary}

The above results show that by generating facet-defining Lagrangian cuts, we can efficiently recover the local convex hulls. Taking the expectation of the Lagrangian cuts across all scenarios, we obtain a cut that is valid for the expected recourse function. This cut is tight for the expected recourse function if and only if the cut for each individual scenario is tight. However, due to the linearity of a Lagrangian cut, tightness is achieved only if the incumbent solution used to construct the Lagrangian cut can be separated from the local convex hull $\conv(\epi_{\barx}(\Q_s)) $ for each $s\in [N]$. Thus, this approach may yield only a lower estimate of the expected recourse function rather than an exact characterization, resulting in only a lower bound of \eqref{master_problem} (see the example below).
\begin{example}\label{eg: conv_env}
Consider a two-stage problem $\min\{-x+\Q(x): x\in\{0,1,2\}\}$, where $\Q(x)=\frac{1}{2}\Q_1(x)+\frac{1}{2}\Q_2(x)$ and local recourse functions are given by $\Q_1(x) = \min\{y: y\geq \frac{1}{2}x+1, y\in\Ze^+\}$ and $\Q_2(x) = \min\{y: y\geq 2x-1, y\in\Ze^+\}$. The values of the local recourse functions are $\Q_1(0) = 1, \Q_1(1) = 2, \Q_1(2) = 2$ (see \Cref{fig:eg1_s1}) and $\Q_2(0) = 0, \Q_2(1) = 1, \Q_2(2) = 3$ (see \Cref{fig:eg1_s2}). Thus, the expected recourse function values are $\mathcal{Q}(0) = \frac{1}{2}, \mathcal{Q}(1) = \frac{3}{2}, \mathcal{Q}(2) = \frac{5}{2}$. At $x=1$, the optimal Lagrangian dual values yield $\mathrm{co}(\Q_1)(1)=\frac{3}{2}$ and $\mathrm{co}(\Q_2)(1)=1$, resulting in an outer approximation of the expected recourse function: $\sum_{s\in[2]}p_s\mathrm{co}(\Q_s)(1)=\frac{5}{4}$. However, as shown in Figure \ref{fig:eg1_exp}, $\mathrm{co}(\mathcal{Q})(1)= \mathcal{Q}(1)=\frac{3}{2}$. In fact, solving the problem solely using Lagrangian cuts provides a lower bound of $\frac{1}{4}$ with the solution $x = 1$. However, the optimal value of this two-stage problem is $\frac{1}{2}$, achieved at points $x = 0, 1, 2$.\qedA
\end{example}
 \begin{figure}[htbp]
  \vspace{-10pt}
\subfigure[Scenario 1]{\label{fig:eg1_s1}
       \begin{tikzpicture}[scale=1.0]
    \begin{axis}[
        axis lines=middle,       
        xlabel={$x$},                 
        xtick={0,1,2},            
        ytick={0,1,2,3},         
        ymin=0, ymax=3.5,        
        xmin=0, xmax=2.5,     
        axis line style={-stealth},
        scatter/classes={
            a={mark=*,black}     
        },
        axis equal image,
    ]
    \addplot[only marks, mark=*, black] coordinates {(0,1) (1,2) (2,2)};

    \addplot[dashed, thick, red] coordinates {(0,1) (2,2)};

    \node at (axis cs:1, 2) [anchor=south east] {$\Q_1(x)$};
    \node at (axis cs:1, 1.6) [anchor=north west, red] {$\mathrm{co}(\Q_1)(x)$};

    \end{axis}
\end{tikzpicture}
}
    \hfill
\subfigure[Scenario 2]{\label{fig:eg1_s2}
       \begin{tikzpicture}[scale=1.0]
    \begin{axis}[
        axis lines=middle,       
        xlabel={$x$},                 
        xtick={0,1,2},            
        ytick={0,1,2,3},         
        ymin=0, ymax=3.5,        
        xmin=0, xmax=2.5,     
        axis line style={-stealth},
        scatter/classes={
            a={mark=*,black}     
        },
        axis equal image,
    ]
    \addplot[only marks, mark=*, black] coordinates {(0,0) (1,1) (2,3)};

    \addplot[dashed, thick, red] coordinates {(0,0) (1,1) (2,3)};

    \node at (axis cs:1, 1) [anchor=south east] {$\Q_2(x)$};
    \node at (axis cs:2.0, 2.0) [anchor=north, red] {$\mathrm{co}(\Q_2)(x)$};

    \end{axis}
\end{tikzpicture}
}
   \hfill
\subfigure[Expected recourse function]{\label{fig:eg1_exp}
       \begin{tikzpicture}[scale=1.0]
    \begin{axis}[
        axis lines=middle,       
        xlabel={$x$},                 
        xtick={0,1,2},            
        ytick={0,1,2,3},         
        ymin=0, ymax=3.5,        
        xmin=0, xmax=2.5,     
        axis line style={-stealth},
        scatter/classes={
            a={mark=*,black}     
        },
        axis equal image,
    ]
    \addplot[only marks, mark=*, black] coordinates {(0,1/2) (1,3/2) (2,5/2)};

    \addplot[dashed, thick, blue] coordinates {(0,1/2) (1,3/2) (2,5/2)};

    \addplot[dashdotted, thick, red] coordinates {(0,1/2) (1,5/4) (2,5/2)};

    \node at (axis cs:0.5, 1.2) [anchor=south] {$\Q(x)$};
    \node at (axis cs:1.3, 2.2) [anchor=south, blue] {$\mathrm{co}(\Q)(x)$};
    \node at (axis cs:0.5, 1) [anchor=north west, red] {\tiny $\frac{1}{2}\mathrm{co}(\Q_1)(x)+\frac{1}{2}\mathrm{co}(\Q_2)(x)$};
    \end{axis}
\end{tikzpicture}
}
    \caption{\centering The illustration of Example \ref{eg: conv_env}}
     \vspace{-10pt}
\end{figure}
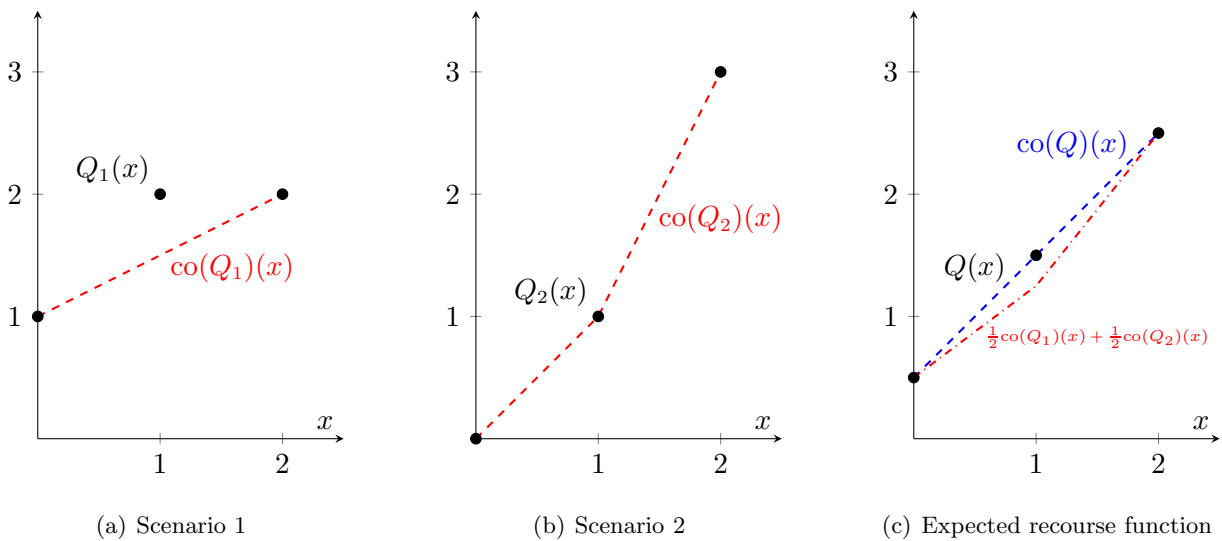

This example indicates that, for a general mixed integer stochastic program, the convex envelope of the expected recourse function may be strictly greater than the expected convex envelope of the local recourse functions. However, to solve \eqref{master_problem} to optimality, we need an approximation that is at least as strong as the convex envelope of the expected recourse function. Unfortunately, Lagrangian cuts are often insufficient for general mixed integer first-stage decisions.

To address this issue, one approach is to derive linear cuts directly for the expected recourse function, such as the scaled cuts proposed in \cite{van2023converging}, which cannot be computed using scenario decomposition methods.
Another approach is to develop non-linear cuts that recover epigraphs of local recourse functions directly instead of their convex hulls, such as reverse norm cuts and augmented Lagrangian cuts introduced in \cite{ahmed2022stochastic}. We will derive stronger cuts that remain tight for general mixed integer first-stage decisions while preserving the effectiveness of Lagrangian cuts in representing local convex hulls.

\subsection{ReLU Lagrangian cuts}
In this section, we introduce the ReLU Lagrangian cuts. Using strong duality, we prove that the cut generated at any feasible first-stage decision is tight, directly recovering the epigraphs of both the local recourse function and its expectation. Unlike ordinary Lagrangian cuts, where nonanticipativity constraints are linear and only linear cuts are produced that are valid for the local convex hull, we study new nonanticipativity constraints using the ReLU function.

For a given feasible first-stage decision $\hatx\in \barx$, its local recourse function value can be obtained by solving
\begin{equation}\label{gen_lag_primal}
\Q_s(\hatx)=\min_{\bm x\in \barx}\left\{\Q_s(\hatx): (x_i-\hat{x}_i)^+ = 0, (x_i-\hat{x}_i)^- = 0,\forall i\in[n]\right\},
\end{equation}
where we define two ReLU functions as $(x_i-\hat{x}_i)^+ = \max\{x_i-\hat{x}_i,0\}$ and $(x_i-\hat{x}_i)^- = \max\{\hat{x}_i-x_i,0\}$.
Taking the dual of this problem with respect to the nonanticipativity constraints, we have
\begin{equation}\label{gen_lag_dual_ori}
\underline{\Q}_s(\bm x) = \sup_{\bm\pi^+,\bm\pi^-\in\Re^n}{\mathcal{L}_s(\bm\pi^+,\bm\pi^-;\hatx)},
\end{equation}
where 
\begin{equation}\label{gen_lag_fun}
\begin{aligned}
\mathcal{L}_s(\bm x,\bm\pi^+,\bm\pi^-;\hatx):= &\inf_{\bm x\in \barx}{\Q_s(\bm x)-\sum_{i\in[n]}\pi^+_i(x_i-\hat{x}_i)^+ - \sum_{i\in[n]}\pi^-_i(x_i-\hat{x}_i)^-}.
\end{aligned}
\end{equation}
The difference between this ReLU Lagrangian dual \eqref{gen_lag_primal} and the ordinary Lagrangian dual \eqref{lag_prob} lies in the ReLU functions, which are nonlinear and can be either linearized or represented by introducing extra binary variables. We can prove that the strong duality holds for the ReLU Lagrangian dual \eqref{gen_lag_primal}, i.e., $\Q_s(\hatx)=\underline{\Q}_s(\hatx)$. As a side product, we can derive optimal dual multipliers in closed form.
\begin{restatable}{theorem}{dual}\label{thm:gen_lag_dual}
Under Assumptions
\ref{assumption:rel_comp_rec}, \ref{assumption:lb} and \ref{assumption:xbounded}, $\underline{\Q}_s(\hatx)=\Q_s(\hatx)$. Moreover, when 
\begin{equation}\label{closed_form}
\rho^*\geq \frac{\Q_s(\hatx) - L}{d},
\end{equation} $(-\bm{1}\rho^*, -\bm{1}\rho^*)$ is optimal to \eqref{gen_lag_dual_ori}, where
 $ 
d = \min\left\{ ||\bm{x}^k - \hatx||_1 : k \in K, \bm{x}^k \neq \hatx \right\},$ and $\{(\bm{x}^k, \bm{y}^k, \omega_k)\}_{k \in K}$ are extreme points of the set $\conv\{(\bm{x}, \bm{y}, \omega) : \bm{T}^s \bm{x} + \bm{W}^s \bm{y} \geq \bm{h}^s, \bm{x} \in \barx, 
\bm{y} \in \Ze^{m_1} \times \Re^{m_2}, \, ||\bm{x} - \hatx||_1 \leq \omega \}
$. 
\end{restatable}
\begin{proof}
See Appendix \ref{pf:thm_dual}.\qed
\end{proof}
The above theorem shows that for any feasible first-stage decision, strong duality must hold and there exists a finite optimal dual solution. This ensures that tight cuts can be generated at any feasible point, which is a fundamental difference between the ReLU Lagrangian and the ordinary Lagrangian cuts. Before formally defining the ReLU Lagrangian cuts, we present the following corollaries, which are useful in subsequent discussions.
% \begin{corollary}\label{coro:opt_rho}
% Under Assumptions
% \ref{assumption:rel_comp_rec}, \ref{assumption:lb} and \ref{assumption:xbounded}, there exists $\rho^*$ such that $\underline{\Q}_s(\hatx)=\mathcal{L}_s(-\rho\bm{1}, -\rho\bm{1};\hatx)$ for all $\rho\geq\rho^*$. 
% \end{corollary}
\begin{restatable}{corollary}{expfeas}\label{coro:expand_feas}
For any bounded set $S$ such that $S\supseteq\barx$, if $({\bm\pi^+}^*, {\bm\pi^-}^*)$ is optimal to $\sup_{\bm\pi^+,\bm\pi^-\in\Re^n}\inf_{\bm x\in S}L_s(\bm x, \bm\pi^+,\bm\pi^-;\hatx)$, then it is also optimal to $\sup_{\bm\pi^+,\bm\pi^-\in\Re^n}\inf_{\bm x\in\barx}L_s(\bm x, \bm\pi^+,\bm\pi^-;\hatx)$,
where
$$L_s(\bm x, \bm\pi^+,\bm\pi^-;\hatx): = \Q_s(\bm x)-\sum_{i\in[n]}\pi^+_i(x_i-\hat{x}_i)^+ - \sum_{i\in[n]}\pi^-_i(x_i-\hat{x}_i)^-.$$
\end{restatable}
\begin{proof}See Appendix \ref{pf:coro_exp_feas}\end{proof}

The following corollary provides an optimality condition for a dual optimal solution.
\begin{restatable}{corollary}{optcond}\label{coro:gen_lag_opt_cond}
For a given feasible first-stage decision $\hatx\in\barx$, $(\hat{\bm\pi}^+,\hat{\bm\pi}^-)$ is optimal to \eqref{gen_lag_dual_ori} if and only if 
\begin{equation}\label{gen_lag_opt_cond}
\Q_s(\hatx) \leq \Q_s(\bm x) - \sum_{i\in[n]}\hat{\pi}^+_i(x_i-\hat{x}_i)^+ - \sum_{i\in[n]}\hat{\pi}^-_i(x_i-\hat{x}_i)^-,
\end{equation}
for all $x\in\barx$, or equivalently, $$\theta \geq \Q_s(\bm x)\geq \Q_s(\hatx)+\sum_{i\in[n]}\hat{\pi}^+_i(x_i-\hat{x}_i)^+ + \sum_{i\in[n]}\hat{\pi}^-_i(x_i-\hat{x}_i)^-,$$ 
for all $(\bm x,\theta)$ in the epigraph of $\Q_s$.
\end{restatable}
\begin{proof}
See Appendix \ref{pf:opt_cond}. \qed
\end{proof}
Now we are ready to formally define the ReLU Lagrangian cuts.
\begin{definition}[ReLU Lagrangian Cuts]
For a given feasible first-stage decision $\hatx\in\barx$, let $(\hat{\bm\pi}^+,\hat{\bm\pi}^-)$ be an optimal solution to the dual problem \eqref{gen_lag_dual_ori}. The following cut is valid for the epigraph of $\Q_s$:
\begin{equation}\label{gen_lag_cut}
\theta \geq \Q_s(\hatx)+\sum_{i\in[n]}\hat{\pi}^+_i(x_i-\hat{x}_i)^+ + \sum_{i\in[n]}\hat{\pi}^-_i(x_i-\hat{x}_i)^-.
\end{equation}
\end{definition}

When the first-stage decision is purely binary, that is, when $\barx=X$, we have $x_i-\hat{x}_i \geq 0$ when $\hat{x}_i=0$ and $x_i-\hat{x}_i \leq 0$ when $\hat{x}_i=1$. In this case, given a binary $\hatx\in X$, the ReLU Lagrangian cut simplifies to the following linear inequality:
\begin{equation}\label{bi_lin_form}
\theta \geq \Q_s(\hatx)+\sum_{i\notin I_{\hatx}}
\hat{\pi}^+_i(x_i-\hat{x}_i) - \sum_{i\in I_{\hatx}}\hat{\pi}^-_i(x_i-\hat{x}_i), 
\end{equation}
where $I_{\hatx} := \{i\in [n]:\hat{x}_i=1\}$.
However, when the first-stage decisions are mixed integers, cut \eqref{gen_lag_cut} is generally nonlinear. To incorporate this nonlinear cut effectively into the master problem when using the cutting plane method, we add the following constraints:
\begin{subequations}\label{lin_gen_lag}
\begin{align}
&\theta \geq \Q_s(\hatx) + \sum_{i\in[n]}\hat\pi_i^+\omega_i^+ + \sum_{i\in[n]}\hat\pi_i^-\omega_i^-,\label{mip_rep_cut}\\
&\omega_i^+ - \omega_i^- = x_i-\hat{x}_i,0\leq \omega_i^+\leq (B_i-\hat{x}_i)z_i,
0\leq \omega_i^-\leq \hat{x}_i(1-z_i),\forall i\in[n], \label{mip_rep_omega}\\
& \bm z\in\{0,1\}^n. \label{mip_rep_z}
\end{align}
\end{subequations}

The system of inequalities in \eqref{lin_gen_lag} is tight in the sense that its continuous relaxation recovers the convex hull of a relaxed local epigraphical set. To be specific, let us define $\widebar{B}=\times_{i\in[n]} [0,B_i]\bigcap\Ze^{n_1}\times\Re^{n_2}$ as a relaxed domain of the first-stage decision and a mixed integer set $S_1 = \left\{(\bm x,\theta)\in\widebar{B}\times\Re: \eqref{bi_lin_form}\right\}$ that consists of the relaxed domain set and a ReLU Lagrangian cut.
The following proposition summarizes this result.
\begin{restatable}{proposition}{lpconv}\label{prop:conv_gen_lshaped}
Given a first-stage decision $\hatx\in\barx$, we have 
\begin{equation}\label{s3}
\conv(S_1) = \left\{(\bm x,\theta)\in \times_{i\in[n]}[0, B_i]\times\Re: 
\exists (\bm\omega^+, \bm\omega^-,\bm z)\in\Re^n\times\Re^n\times[0,1]^n,
\eqref{mip_rep_cut},\eqref{mip_rep_omega}\right\}.
\end{equation}
\end{restatable}
\begin{proof}
See Appendix \ref{pf:lp_conv}. \qed
\end{proof}

%To obtain a dual optimal solution and derive a ReLU Lagrangian cut, \Cref{coro:opt_rho} suggests considering a sufficiently large $\rho$. However, it also indicates that the resulting cut may become arbitrarily weak as $\rho$ approaches infinity. In the subsequent sections, we will demonstrate how to appropriately select such $\rho$ by leveraging existing families of cuts which, while tight, may be weak, and how to strengthen these cuts. Furthermore, \Cref{coro:expand_feas} offers an alternative approach: rather than solving \eqref{gen_lag_dual_ori} directly, we can solve the dual problem with an expanded first-stage feasible region. Meanwhile, to verify whether an inequality is a ReLU Lagrangian cut or not, \Cref{coro:gen_lag_opt_cond} provides a criterion: any inequality valid for the epigraph of a local recourse function in the form \eqref{gen_lag_cut} is a ReLU Lagrangian cut at $\hatx$.
To obtain a dual optimal solution and derive a ReLU Lagrangian cut, the closed-form optimal dual solution \eqref{closed_form} in \Cref{thm:gen_lag_dual} suggests selecting a sufficiently large $\rho$. However, it also indicates that as $\rho$ approaches infinity, the resulting cut may become arbitrarily weak. In the following sections, we will show how to appropriately select $\rho$ by leveraging existing families of cuts that, while tight, may be weak, and how to strengthen these cuts. Additionally, \Cref{coro:expand_feas}  presents an alternative approach: rather than solving \eqref{gen_lag_dual_ori} directly, we may solve the dual problem with an expanded first-stage feasible region. To verify whether an inequality qualifies as a ReLU Lagrangian cut, \Cref{coro:gen_lag_opt_cond} provides a criterion: any inequality valid for the epigraph of a local recourse function in the form \eqref{gen_lag_cut} is a ReLU Lagrangian cut at $\hatx$.

\subsection{Related cut families}
In this subsection, we show that all tight Lagrangian cuts are special ReLU Lagrangian cuts. Therefore, similar to Lagrangian cuts, ReLU Lagrangian cuts are also sufficient to describe the local convex hull. Moreover, we show that certain nonlinear cuts--such as reverse norm cuts and augmented Lagrangian cuts \citep{ahmed2022stochastic}, as well as integer L-shaped cuts \citep{laporte1993integer}--are special cases of ReLU Lagrangian cuts. We extend the concept of L-shaped cuts to derive new cuts for purely integer first-stage decisions, termed ``$\Lambda$-shaped cuts." A comparison of ReLU Lagrangian cuts with existing cut families demonstrates their advantages and necessity in accurately describing both local convex hulls and local epigraphs. 

We define the set of admissible ReLU Lagrangian cut coefficients.
\begin{definition}[Admissible ReLU Lagrangian Cut Coefficients]\label{def_pihat}
For a given $\hatx\in\barx$, let set $\pixhat := \{(\bm\pi^+,\bm\pi^-)\in\Re^{2n}:(\bm\pi^+,\bm\pi^-)\text{ is optimal to \eqref{gen_lag_dual_ori}}\}$ denote all optimal solutions to the ReLU Lagrangian dual problem \eqref{gen_lag_dual_ori}. 
\end{definition}

Our first result shows that
\begin{restatable}{proposition}{lag}\label{prop:lag_relu}
Any tight Lagrangian cut is a ReLU Lagrangian cut.
\end{restatable}
\begin{proof}
See Appendix \ref{pf:lag}. \qed
\end{proof}
As shown in \Cref{coro}, the local convex hull can be recovered using Lagrangian cuts that are tight at its extreme points. By this proposition, these cuts are also ReLU Lagrangian cuts. 
\begin{corollary}
For a given $\hatx\in\barx$, the following result must hold:
$\{(\bm x,\theta)\in\conv(\barx)\times\Re:
\theta \geq \Q_s(\hatx)-\sum_{i\in[n]}\pi^+_i(x_i-\hat{x}_i)^+ - \sum_{i\in[n]}\pi^-_i(x_i-\hat{x}_i)^-,\forall \hatx\in\barx, (\bm\pi^+,\bm\pi^-)\in\pixhat\}\subseteq\conv(\epi(\Q_s))$.
\end{corollary}
This local convex hull description property demonstrates the strength of the ReLU Lagrangian cuts, distinguishing them from two other cuts: reverse norm cuts and the integer L-shaped cuts.
\begin{definition}[Reverse norm cuts, \citealt{ahmed2022stochastic}]
If the local recourse function $\Q_s$ is Lipschitz continuous with Lipschitz constant $\rho$ under the $L_1$-norm, i.e., 
$|\Q_s(\bm x)-\Q_s(\bm x^\prime)|\leq \rho||\bm x-\bm x^\prime||_1$,
for all $\bm x,\bm x^\prime \in \barx$,
then given $\bm \hatx\in\barx$, we can derive a cut:
\begin{equation*}
\theta\geq\Q_s(\hatx)-\rho||\bm x-\hatx||_1.
\end{equation*}
\end{definition}
\begin{proposition}
    Reverse norm cuts are ReLU Lagrangian cuts.
\end{proposition}

Although reverse norm cuts can be applied to general mixed integer first-stage decisions, they are restricted by the requirement that local recourse functions must be Lipschitz continuous--a condition that is often difficult to satisfy when part of second-stage decision variables is discrete. The following integer L-shaped cuts can be derived using only the current recourse function value and a lower bound of the recourse function. It is applicable when the first-stage variables are binary.
\begin{definition}[Integer L-shaped cuts, \citealt{laporte1993integer}]
Let $L$ be a valid lower bound of the local recourse function $\Q_s$ defined in
\eqref{loc_rec}. Given a feasible first-stage decision $\hat{\bm x}\in X$, 
an integer L-shaped cut generated at $\hat{\bm x}$ admits the form
\begin{equation*}
\theta \geq \left(\Q_s(\hat{\bm x})-L\right)\left(\sum_{i\in \ix}x_i - \sum_{i\notin \ix}x_i\right)-\left(\Q_s(\hat{\bm x})-L\right)(|I|-1)+L,
\end{equation*}
where $\ix:=\{i\in[n]: \hat{x}_i=1\}$.
\end{definition}
We can equivalently rewrite this cut as
$\theta \geq \Q_s(\hatx) - (\Q_s(\hatx)-L)||\bm x-\hatx||_1$. Note that for the optimal dual solution in the form of \eqref{closed_form}, the distance $d$ using $L_1$-norm is at least one for integer first-stage decisions. This coincides with the coefficients of L-shaped cuts. Similarly, when first-stage decisions are purely integers, we generalize the idea of L-shaped cuts to derive the $\Lambda$-shaped cuts.
\begin{definition}[$\Lambda$-shaped cuts]
When the first-stage feasible region $\barx\subseteq\Ze^n$, given $\hatx\in \barx$ and a lower bound $L$ of the recourse function, we can derive a cut 
\begin{equation}\label{gen_lshaped}
\theta \geq \Q_s(\hatx) - (\Q_s(\hatx)-L)||\bm x-\hatx||_1.
\end{equation}
\end{definition}
\begin{restatable}{proposition}{lamrelu}\label{prop_L_shape_ReLU}
        L-shaped cuts and $\Lambda$-shaped cuts are ReLU Lagrangian cuts.
\end{restatable}

The cuts discussed above either cannot be applied to general SMIPs or impose specific requirements on local recourse functions. Due to the symmetry of the $\ell_1$-norm, these cuts fail to recover local convex hulls, resulting in weaker cuts that may not effectively enhance the outer approximation of the expected epigraph during the solution procedure. The following example demonstrates the
insufficiency of reverse norm cuts and $\Lambda$-shaped cuts in describing the local convex hull.
\begin{example}
Given set $\barx=\{(0,0)^\top, (0,1)^\top,(2,1)^\top,(0,3)^\top,(2,3)^\top,  (1,4)^\top,(1,2)^\top\}$, consider a local recourse function given by $\Q_s(0,0)=0, \Q_s(0,1)=\Q_s(2,1)=1, \Q_s(0,3)=\Q_s(2,3)=3,  \Q_s(1,4)=4$ and $\Q_s(1,2)=-10$. It is clear that the Lipschitz constant $\rho\geq \frac{3}{10}$. The strongest reverse norm cut we can derive is $\theta \geq \Q_s(\hatx) - \frac{3}{10}||\bm x-\hatx||_1$ for all $\hatx\in \barx$.

Similarly, to derive $\Lambda$-shaped cuts, we note that $L=-10$ is the best lower bound of the local recourse function $\Q_s$. Then, we can derive cuts $\theta \geq \Q_s(\hatx) - (\Q_s(\hatx)+10)||\bm x-\hatx||_1$ for all $\hatx\in\barx$.

It is easy to check that $x_1=1.5, x_2=0.5,\theta= 0$ is valid for all reverse norm cuts and $\Lambda$-shaped cuts. However, it is not in $\conv(\epi (\Q_s))$ since $\mathrm{co}(\Q_s) (1.5, 0.5) = 0.5$.
\qedA
\end{example}

The following family of cuts subsumes all previously discussed types and provides a stronger outer approximation of local convex hulls, which is also a special case of ReLU Lagrangian cut.
\begin{definition}[Augmented Lagrangian cut \citealt{ahmed2022stochastic}]
Given $\hatx\in\barx$, $\bm\pi\in\Re^n$, $\rho\geq 0$, let 
$L_A^s(\bm\pi, \rho; \hatx) = \min_{\bm x}\left\{\Q_s(\bm x)+\bm\pi^\top(\hatx-\bm x)+\rho||\hatx-\bm x||_1: \bm x\in\barx\right\}.$
The following augmented Lagrangian cut is valid for the epigraph of $\Q_s$:
\begin{equation*}
\theta \geq L_A^s(\bm\pi, \rho; \hatx) + \bm\pi^\top(\bm x-\hatx) - \rho ||\bm x-\hatx||_1.
\end{equation*}
\end{definition}
Given a tight augmented Lagrangian cut
\begin{equation}\label{tight_alc}
\theta \geq \Q_s(\hatx) + \bm\pi^\top(\bm x-\hatx) - \rho ||\bm x-\hatx||_1,
\end{equation}
($\bm\pi=\bm{0}$ when it is a reverse norm cut), we can equivalently rewrite it as
\begin{equation}\label{alc_to_ReLU}
\theta \geq \Q_s(\hatx) + \sum_{i\in[n]} (\pi_i-\rho)(x_i-\hat{x}_i)^+ + \sum_{i\in[n]} (-\pi_i-\rho)(x_i-\hat{x}_i)^-.
\end{equation}
Then, according to \Cref{coro:gen_lag_opt_cond}, we have
\begin{restatable}{proposition}{aug}\label{prop:aug_relu}
  Tight augmented Lagrangian cuts are ReLU Lagrangian cuts. 
\end{restatable}

Augmented Lagrangian cuts are sufficient to describe local convex hulls, as they reduce to ordinary Lagrangian cuts when penalty terms are omitted. However, ReLU Lagrangian cuts generally provide stronger approximations of local epigraphs, allowing for fewer cuts in the solution procedure. This phenomenon is illustrated in the following example.
\begin{example}\label{alc_vs_ReLU}
Consider a local recourse function given by 
$\Q_s(0,1)=3, \Q_s(0,2)=2, \Q_s(0,3)=1, 
\Q_s(1,0)=5, \Q_s(1,1)=\frac{15}{2}, \Q_s(1,2)=10, \Q_s(1,3)=\frac{11}{2},\Q_s(1,4)=1, \Q_s(2,1)=5, \Q_s(2,2)=4, \Q_s(2,3)=3$. At $\hatx = (1,2)^\top$, we can derive a ReLU Lagrangian cut
\begin{equation}\label{lifted_facet_defining}
\theta \geq 10-6(x-1)^+-8(x-1)^--\frac{9}{2}(x-2)^+ -\frac{5}{2}(x-2)^-.    
\end{equation}
As shown in Figure \ref{fig:eg3_relu}, this cut is tight at points $(0,2,\Q_s(0,2))^\top$, $(1,0,\Q_s(1,0))^\top$,   $(1,1,\Q_s(1,1))^\top$,   $(1,2,\Q_s(1,2))^\top$,  $(1,3,\Q_s(1,3))^\top$,  $(1,4,\Q_s(1,4))^\top$ and $(2,2,\Q_s(2,2))$ of the epigraph. To describe the epigraph, we need to add one more cut $\theta \geq 3+(x_1-0)^+ -(x_2-1)^+ + (x_2-1)^-$ generated at $\hatx = (0,1)^\top$.

To describe the local epigraph with augmented Lagrangian cuts, one can check by enumeration that at least three cuts are required. At $\bm x = (1,2)^\top$, cut \eqref{lifted_facet_defining} cannot be expressed as an augmented Lagrangian cut since the system of linear equations $\pi_1 + \rho=-6, -\pi_1+\rho = -8, \pi_2 + \rho=-\frac{9}{2}, -\pi_2+\rho = -\frac{5}{2}$ has no solution. In fact, at least two augmented Lagrangian cuts; for example, the following two cuts 
\begin{subequations}
\begin{equation}
\label{ext_agc_1}
    \theta \geq 10+(x_1-1)+\frac{5}{2}(x_2-2) - 7||\bm x-\hatx||_1,
\end{equation}
%and
\begin{equation}
\label{ext_agc_2}
\theta \geq 10+(x_1-1)-\frac{9}{2}(x_2-2) - 7||\bm x-\hatx||_1,
\end{equation}
\end{subequations}
are required 
to separate $(1,0,\Q_s(1,0))^\top (1,1,\Q_s(1,1))^\top, (1,2,\Q_s(1,2))^\top, (1,3,\Q_s(1,3))^\top$ and $(1,4,\Q_s(1,4))^\top$ from the epigraph, as shown in Figure \ref{fig:eg3_aug}. 
\qedA
\end{example}
\begin{figure}[htbp]
  \vspace{-10pt}
  \centering
\subfigure[One ReLU Lagrangian cut]{\label{fig:eg3_relu}
\begin{tikzpicture}
    \begin{axis}[
        view={70}{45}, 
        xlabel={$x_1$},
        ylabel={$x_2$},
        zlabel={$\Q_s(\bm x)$},
        xtick={0,1,2},
        ytick={0,1,2,3,4},
        ztick={0,2,4,6,8,10},
        ymax = 6,
        axis lines=middle, 
        axis line style={-stealth, gray}, 
        ticks=none,
        xlabel style={anchor=north,gray},
        ylabel style={anchor=west,gray},
        zlabel style={anchor=south,gray},
        % axis equal image, 
        height=8cm,
        width=6cm,  
    ]

    \addplot3[only marks, mark=*, black] coordinates {
        (0,1,3) (0,2,2) (0,3,1) 
        (1,0,5) (1,1,7.5) (1,2,10) (1,3,5.5) (1,4,1) % Q_s(1,y)
        (2,1,5) (2,2,4) (2,3,3) 
    };
    \addplot3[thick,black]coordinates{(1,2,10) (1,0,5) (1, -2, 0)};
\addplot3[dashed, thick,black]coordinates{(1,2,10) (0,2,2) (-1/4,2,0)};
\addplot3[thick,black]coordinates{(1,2,10) (1,4,1) (1,38/9,0)};
\addplot3[thick,black]coordinates{(1,2,10) (2,2,4) (8/3,2,0)};
    % \addplot3[dotted, thick, gray] coordinates {(1,0,5) (0,2,2) (1,4,1)(2,2,4) (1,0,5)
    % };
        \addplot3[fill, faceted color=gray, opacity=0.2] coordinates {
   (1,38/9,0) (1,2,10) (-1/4,2,0)
};
    \addplot3[fill, faceted color=gray, opacity=0.2] coordinates {
   (1,38/9,0) (1,2,10) (8/3,2,0)
};
    \addplot3[fill, faceted color=gray, opacity=0.2] coordinates {
   (8/3,2,0) (1,2,10) (1,0,5) (1, -2, 0)
};
        \addplot3[fill, faceted color=gray, opacity=0.2] coordinates {
 (1, -2, 0) (1,2,10) (-1/4,2,0)
};
\node at (axis cs:1,2,10) [anchor=west] {\tiny Cut \eqref{lifted_facet_defining}};
\end{axis}
\end{tikzpicture}
}
\hspace{5cm}
\subfigure[Two augmented Lagrangian cuts]{\label{fig:eg3_aug}
\begin{tikzpicture}
 [scale=1.00]
    \begin{axis}[
        view={70}{40}, 
        xlabel={$x_1$},
        ylabel={$x_2$},
        zlabel={$\Q_s(\bm x)$},
        xtick={0,1,2},
        ytick={0,1,2,3,4},
        ztick={0,2,4,6,8,10},
        ymax=7,
        axis lines=middle, 
        axis line style={-stealth, gray}, 
        ticks=none,
        xlabel style={anchor=north,gray},
        ylabel style={anchor=west,gray},
        zlabel style={anchor=south,gray},
        % axis equal image, 
        height=8cm,
        width=6cm, 
    ]
    
    \addplot3[only marks, mark=*, black] coordinates {
        (0,1,3) (0,2,2) (0,3,1) 
        (1,0,5) (1,1,7.5) (1,2,10) (1,3,5.5) (1,4,1) % Q_s(1,y)
        (2,1,5) (2,2,4) (2,3,3) 
     };   
     \addplot3[dashed,thick,red]coordinates{(1,2,10) (1,4/3,11/3) (1,18/19,0)};
\addplot3[dashed, red, thick]coordinates{(1,2,10) (0,2,2) (-1/4,2,0)};
\addplot3[thick, red]
coordinates{(1,2,10) (1,4,1) (1,38/9,0)};
\addplot3[red, thick]coordinates{(1,2,10) (2,2,4) (8/3,2,0)};
% \addplot3[dotted, red] coordinates {(1,4/3,11/3) (0,2,2) (1,4,1)(2,2,4) (1,4/3,11/3)};

\addplot3[fill=red, faceted color=red, opacity=0.2] coordinates{
(1,38/9,0) (1,2,10) (-1/4, 2, 0)
};
\addplot3[fill=red, faceted color=red, opacity=0.2] coordinates{
(8/3,2,0) (1,2,10) (1,38/9,0)
};
\addplot3[fill=red, faceted color=red, opacity=0.2] coordinates{
(8/3,2,0) (1,2,10) (1,18/19,0)
};
\addplot3[fill=red, faceted color=red, opacity=0.2] coordinates{
(-1/4, 2, 0) (1,2,10) (1,18/19,0)
};
\addplot3[thick,blue, opacity=0]coordinates{(1,2,10) (1,0,5) (1, -2, 0) };
\addplot3[dashed,thick,blue, opacity=0]coordinates{(1,2,10) (0,2,2) (-1/4,2,0) };
\addplot3[thick,blue, opacity=0]coordinates{(1,2,10) (2,2,4) (8/3,2,0)};
\addplot3[dashed, thick,blue, opacity=0]coordinates{(1,2,10) (1,8/3,7/3)  (1,66/23,0)};
\addplot3[dotted, thick, blue, opacity=0] coordinates {(1,0,5) (0,2,2) (1,8/3,7/3)(2,2,4) (1,0,5)}; 
\addplot3[fill=blue, opacity=0] coordinates{
(-1/4,2,0) (1,2,10) (1, -2, 0)
};
\addplot3[fill=blue, opacity=0] coordinates{
(-1/4,2,0) (1,2,10) (1,66/23,0)
};
\addplot3[fill=blue, opacity=0] coordinates{
(8/3,2,0) (1,2,10) (1,66/23,0)
};
\addplot3[fill=blue, opacity=0] coordinates{
(1, -2, 0) (1,2,10) (8/3,2,0)
};
\addplot3[thick,blue]coordinates{(1,2,10) (1,0,5) (1, -2, 0) };
\addplot3[dashed,thick,blue]coordinates{(1,2,10) (0,2,2) (-1/4,2,0) };
\addplot3[thick,blue]coordinates{(1,2,10) (2,2,4) (8/3,2,0)};
\addplot3[dashed, thick,blue]coordinates{(1,2,10) (1,8/3,7/3)  (1,66/23,0)};
% \addplot3[dotted, thick, blue] coordinates {(1,0,5) (0,2,2) (1,8/3,7/3)(2,2,4) (1,0,5)}; 
\addplot3[fill=blue, opacity=0.2] coordinates{
(-1/4,2,0) (1,2,10) (1, -2, 0)
};
\addplot3[fill=blue, opacity=0.2] coordinates{
(-1/4,2,0) (1,2,10) (1,66/23,0)
};
\addplot3[fill=blue, opacity=0.2] coordinates{
(8/3,2,0) (1,2,10) (1,66/23,0)
};
\addplot3[fill=blue, opacity=0.2] coordinates{
(1, -2, 0) (1,2,10) (8/3,2,0)
};
\node at (axis cs:1,1.7,8) [anchor=south east] {\color{blue}\tiny Cut \eqref{ext_agc_1}};
\node at (axis cs:1,2.4,7) [anchor=west] {\color{red}\tiny Cut \eqref{ext_agc_2}};
\end{axis}
\end{tikzpicture}
}
\caption{\centering The illustration of Example \ref{alc_vs_ReLU}.}
  \vspace{-10pt}
\end{figure}
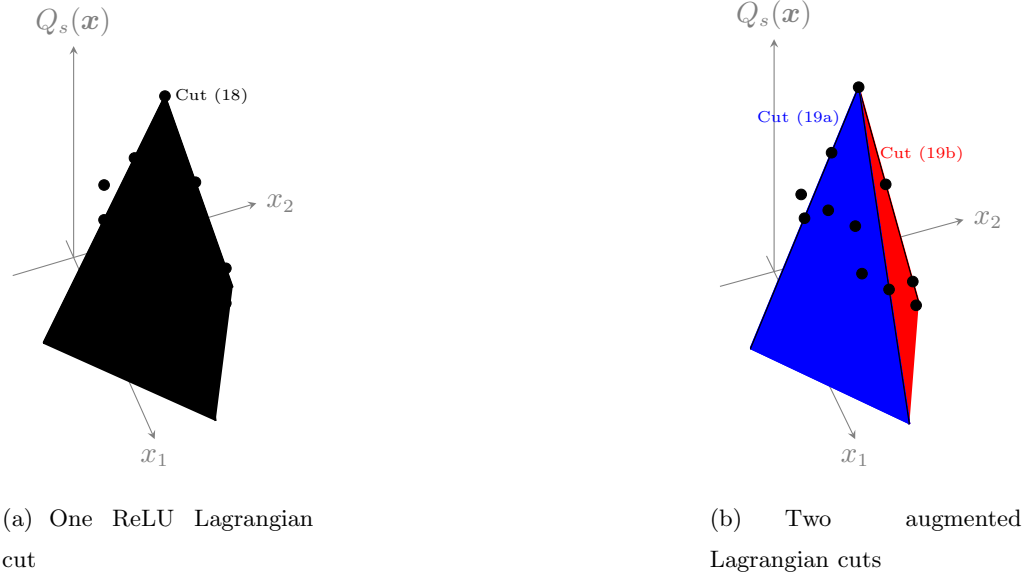

We will further explain in Section \ref{subsec:mixed_integer} that the cut \eqref{lifted_facet_defining} is the strongest possible ReLU Lagrangian cut generated at $\hatx=(1,2)^\top$ by analyzing the structure of the set $\pixhat$. Intuitively, from \eqref{alc_to_ReLU}, a ReLU Lagrangian cut reduces to an augmented Lagrangian cut if and only if $\pi_i^+ = \pi_i - \rho, \pi_i^- = -\pi_i - \rho$ for all $i\in[n]$ and $\rho\geq 0$, which is equivalent to say $\pi_i^+ +\pi_i^- = \pi_j^+ +\pi_j^-\leq 0$ for all $i,j\in [n]$. This requirement restricts our choice of cut coefficients in $\pixhat$.

It is also worth noting that although both augmented Lagrangian cuts and ReLU Lagrangian cuts can describe the local convex hulls, they fail to describe the convex hull of the epigraph of the expected recourse function, as demonstrated in the following example.
\begin{example}\label{example_conv_hull_ReLU}
Let $X=\{0,1,2,3\}$ and define the local recourse functions as $\Q_1(x)=0$ if $x\in\{0,3\}$, $\Q_1(x)=1$ if $x\in\{1,2\}$, and $\Q_2(x)=4$ if $x\in\{0,3\}$, $\Q_2(x)=1$ if $x\in\{1,2\}$.
The expected recourse function is then given by $\mathcal{Q}(x)=2$ if $x\in\{0,3\}$, and $\mathcal{Q}(x)=1$ if $x\in\{1,2\}$.
For scenario 1, the strongest ReLU Lagrangian cut we can derive are $\theta \geq 0$ at $x=0$ and $x=3$, $\theta\geq 1-(x-1)^- -\frac{1}{2}(x-1)^+$ at $x=1$, and $\theta\geq 1-\frac{1}{2}(x-2)^--(x-2)^+$ at $x=2$.
For scenario 2, the strongest ReLU Lagrangian cuts we can derive are $\theta \geq 4-3(x-0)^+$ at $x=0$, $\theta \geq 1+3(x-1)^-$ at $x=1$, $\theta \geq 1+3(x-2)^+$ at $x=2$, and $\theta \geq 4-3(x-3)^-$ at $x=3$. 
Combining the two scenarios, we have $\theta \geq 2-\frac{3}{2}(x-0)^+$ at $x=0$, $\theta\geq 1+(x-1)^--\frac{1}{4}(x-1)^+$ at $x=1$, $\theta\geq 1-\frac{1}{4}(x-2)^-+(x-2)^+$ at $x=2$, and $\theta\geq 2-\frac{3}{2}(x-3)^-$ at $x=3$.
In Figure \ref{fig:eg4}, the black dots represent the local and expected recourse functions, while the dash lines represent ReLU Lagrangian cuts. The gray areas depict the convex hulls of the epigraphs, and the shadowed areas are outer approximations derived by ReLU Lagrangian cuts. We can observe that for the expected recourse function, the point $(\frac{3}{2},\frac{7}{8})^\top$ is contained in the outer approximation shaped by ReLU Lagrangian cuts, while it is not in the convex hull of the epigraph.\qedA
\end{example}
 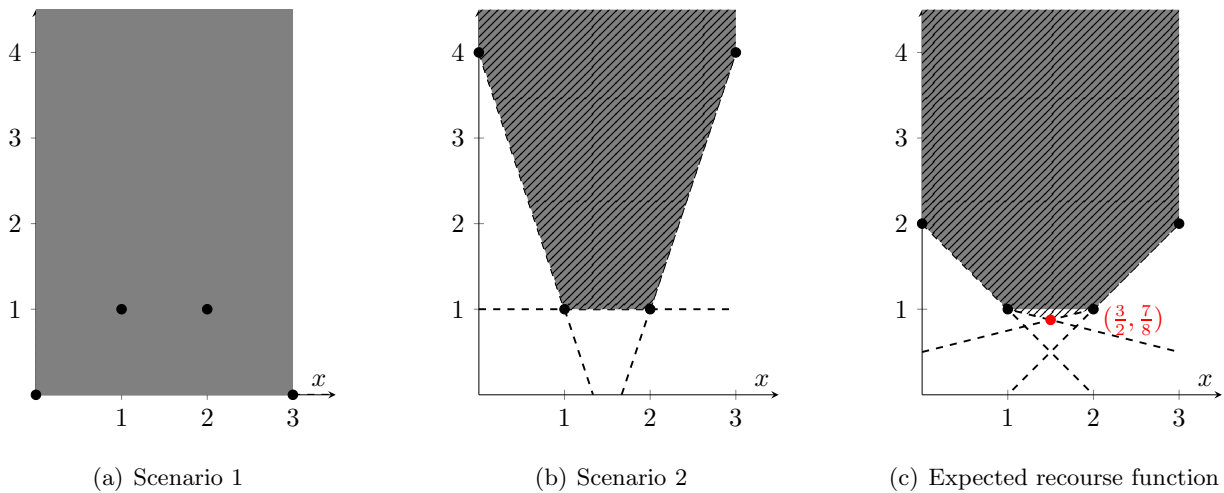
\begin{figure}[htbp]
 \vspace{-10pt}
\subfigure[Scenario 1]{\label{fig:eg4_s1}
       \begin{tikzpicture}[scale=0.9]
    \begin{axis}[
        axis lines=middle,       
        xlabel={$x$},                 
        xtick={0,1,2,3},            
        ytick={0,1,2,3,4},         
        ymin=0, ymax=4.5,        
        xmin=0, xmax=3.5,     
        axis line style={-stealth},
        scatter/classes={
            a={mark=*,black}     
        },
        axis equal image,
    ]
    \addplot[only marks, mark=*, black] coordinates {(0,0) (3,0) (1,1) (2,1)};

    \addplot[dashed, thick] coordinates {(0,0) (3.5,0)};
    % \node at (axis cs:0, 0) [anchor=south west] {$\theta \geq 0$};

    \addplot[dashed,thick] coordinates {(0,0) (1,1) (3,0)};
    % \node at (axis cs:1, 1) [anchor=south] {\tiny $\theta\geq 1-(x-1)^- -\frac{1}{2}(x-1)^+$};
    \addplot[dashed,thick] coordinates {(0,0) (2,1) (3,0)};
     
      \addplot [
        pattern=north east lines,
        draw=none,
    ] coordinates {(0,4.5) (0,0) (1,1) (1.5,3/4) (2,1) (3,0) (3,4.5)} --cycle;

     \addplot [
        fill=gray,
        fill opacity=0.3, 
        draw=none,
    ] coordinates {(0,0) (3,0)  (3,4.5) (0,4.5)} --cycle;
    
    \end{axis}
\end{tikzpicture}
}
    \hfill
\subfigure[Scenario 2]{\label{fig:eg4_s2}
       \begin{tikzpicture}[scale=0.9]
    \begin{axis}[
        axis lines=middle,       
        xlabel={$x$},                 
        xtick={0,1,2,3},            
        ytick={0,1,2,3,4},         
        ymin=0, ymax=4.5,        
        xmin=0, xmax=3.5,     
        axis line style={-stealth},
        scatter/classes={
            a={mark=*,black}     
        },
        axis equal image,
    ]
    \addplot[only marks, mark=*, black] coordinates {(0,4) (3,4) (1,1) (2,1)};

    \addplot[dashed, thick] coordinates {(0,4) (4/3,0)};

    \addplot[dashed,thick] coordinates {(0,4) (1,1) (3,1)};
   
    \addplot[dashed,thick] coordinates {(0,1) (2,1) (3,4)};
    \addplot[dashed,thick] coordinates {(5/3,0) (3,4)};
     
      \addplot [
        fill=gray,
        fill opacity=0.3, 
        draw=none,
    ] coordinates {(0,4) (1,1) (2,1) (3,4) (3,4.5) (0,4.5) (0,4)} --cycle;

     \addplot [
        pattern=north east lines,
        draw=none,
    ] coordinates {(0,4) (1,1) (2,1) (3,4) (3,4.5) (0,4.5) (0,4)} --cycle;
    
    \end{axis}
\end{tikzpicture}
}
   \hfill
\subfigure[Expected recourse function]{\label{fig:eg4_exp}
       \begin{tikzpicture}[scale=0.9]
    \begin{axis}[
        axis lines=middle,       
        xlabel={$x$},                 
        xtick={0,1,2,3},            
        ytick={0,1,2,3,4},         
        ymin=0, ymax=4.5,        
        xmin=0, xmax=3.5,     
        axis line style={-stealth},
        scatter/classes={
            a={mark=*,black}     
        },
        axis equal image,
    ]
    \addplot[only marks, mark=*, black] coordinates {(0,2) (3,2) (1,1) (2,1)};

    \addplot[dashed, thick] coordinates {(0,2) (2,0)};

    \addplot[dashed,thick] coordinates {(3,2) (1,0)};
   
    \addplot[dashed,thick] coordinates {(0,2) (1,1) (3,1/2)};
    \addplot[dashed,thick] coordinates{(0,1/2) (2,1) (3,2)};
     
      \addplot [
        fill=gray,
        fill opacity=0.3, 
        draw=none,
    ] coordinates {(0,2) (1,1)  (2,1) (3,2) (3,4.5) (0,4.5) (0,2)} --cycle;

 \addplot [
        pattern=north east lines,
        draw=none,
    ] coordinates {(0,2) (1,1) (1.5,7/8) (2,1) (3,2) (3,4.5) (0,4.5) (0,2)} --cycle;

     \addplot[only marks, mark=*, red] coordinates {(3/2,7/8)};

     \node at (axis cs:2,7/8) [anchor=west] {\color{red}$(\frac{3}{2},\frac{7}{8})$};
    
    \end{axis}
\end{tikzpicture}
}
    \caption{\centering The illustration of Example \ref{example_conv_hull_ReLU}.}\label{fig:eg4}
     \vspace{-12pt}
\end{figure}

This example also demonstrates that by adding the ReLU Lagrangian cuts, the outer approximation eventually converges to the epigraph of the expected recourse function when enforcing the integrality of first-stage decisions. That is, we have the following equalities:
$$v^* = \min \{\bm{c}^\top \bm{x} + \theta: (\bm{x}, \theta) \in \epi_{\barx}(\mathcal{Q})\} = \min_{\bm{x}} \{\bm{c}^\top \bm{x} + \theta: \bm{x} \in \barx, (\bm{x}, \theta) \in E_R\},$$
where set $E_R = \{(\bm{x}, \theta) \in \mathbb{R}^n \times \mathbb{R}: \theta \geq \mathcal{Q}(\hatx)-\sum_{i \in [N]} \pi^+_i (x_i - \hat{x}_i)^+ - \sum_{i \in [n]} \pi^-_i (x_i - \hat{x}_i)^-, \forall \hatx \in \barx, (\bm{\pi}^+, \bm{\pi}^-) \in \sum_{s \in [N]} p_s\pixhat\}$. This is different from Lagrangian cuts, whose corresponding outer approximation only provides a lower bound for the optimal value. 

We also observe that the integrality constraints of the first-stage problem in Example \ref{example_conv_hull_ReLU} are essential for solving the original SMIP to optimality using ReLU Lagrangian cuts. This is because replacing the first-stage feasible region with its convex hull and adding all generated ReLU Lagrangian cuts is insufficient to recover the convex hull of the epigraph. Specifically, since
$\conv\{(\bm{x}, \theta) \in \conv(\barx) \times \mathbb{R}: (\bm{x}, \theta) \in E_R\} \supseteq \conv(\text{epi}_{\barx}(\mathcal{Q}))$, 
we must have
$$\begin{aligned}
v^* & = \min_{\bm{x}} \{\bm{c}^\top \bm{x} + \theta: \bm{x} \in \barx, (\bm{x}, \theta) \in E_R\} \geq \min_{\bm{x}} \{\bm{c}^\top \bm{x} + \theta: \bm{x} \in \conv(\barx), (\bm{x}, \theta) \in E_R\} \\
&\geq \min_{\bm{x}} \{\bm{c}^\top \bm{x} + \theta: \bm{x} \in \barx^{LP}, (\bm{x}, \theta) \in E_R\}.
\end{aligned}
$$
Therefore, when implementing these cuts, maintaining the integrality constraints in the first stage ensures that we achieve the optimal value $v^*$.

Finally, we note that there is a more straightforward way to obtain a ReLU Lagrangian cut in practice than by directly solving the Lagrangian dual \eqref{gen_lag_dual_ori}. This observation motivates our strategy of initially generating a valid cut, albeit potentially weak, at a low computational cost and subsequently refining it into a stronger cut. The methodology and implementation details will be elaborated upon in the following two sections.

\section{Purely Binary First-stage Decisions}\label{sec_binary}

Existing literature generates Lagrangian cuts by solving a sequence of optimization problems (see, e.g., \citealt{zou2019stochastic,chen2022generating}), which can be computationally demanding. In this section, we propose a simple and effective method for generating ReLU Lagrangian cuts when the first-stage decision variables are purely binary.

\subsection{Dual optimal solution set $\pixhat$}
We first analyze the structure of $\pixhat$, the set of all optimal solutions to the Lagrangian dual \eqref{gen_lag_dual_ori}.
\begin{restatable}{proposition}{poly}
\label{prop:dual_poly}
The set $\pixhat$  is a polyhedron.
\end{restatable}
\begin{proof}
See Appendix \ref{pf:poly}. \qed
\end{proof}
This result holds for general mixed integer first-stage decisions. However, when the first-stage decision variables are purely binary, the set $\pixhat$ is a non-pointed polyhedron. Intuitively, we observe that for any $(\bm \pi^+,\bm\pi^-)\in\pixhat$, a ReLU Lagrangian cut \eqref{gen_lag_cut} is equivalent to a linear cut \eqref{bi_lin_form}. This linear cut is uniquely determined by the entries $\pi_i^+$ for $i\notin \ix$ and $\pi_i^-$ for $i\in \ix$, while the remaining entries in the ReLU Lagrangian cut can take arbitrary values.
\begin{restatable}{lemma}{linsp}\label{lemma:boundary}
When the first-stage feasible decisions are purely binary, the recession cone of $\pixhat$ contains $\linspace\{(\bm e_i, \bm{0})\}_{i\in \ix}+\linspace\{(\bm{0}, \bm e_i)\}_{i\notin \ix}$.
\end{restatable}
\begin{proof}
See Appendix \ref{pf:linsp}. \qed
\end{proof}

According to this lemma, we focus on the restriction of $\pixhat$ to the orthogonal complement of the linear subspace $\linspace\{(\bm e_i, \bm{0})\}_{i\in \ix}+\linspace\{(\bm{0}, \bm e_i)\}_{i\notin \ix}$. Projecting out the entries that are fixed to be zero, we consider the set
\begin{align*}
&\mathrm{Proj}_{([n]\backslash \ix, \ix)}\pixhat\left|_{\linspace\{(\bm e_i, \bm{0})\}_{i\in \ix}+\linspace\{(\bm{0}, \bm e_i)\}_{i\notin \ix}}\right.\\
 :=&\left\{\bm\pi\in\Re^n: \exists (\bm\pi^+, \bm\pi^-)\in\pixhat: \pi_i = \pi_i^+, \forall i\notin \ix, \pi_i = \pi_i^-,\forall i\in \ix\right\} 
 = \diag(\bm{\chi}) \dxhat,
\end{align*}
where we let $\chi_i = 2\hat{x}_i-1$ for each $i\in [n]$ and $\diag(\bm{\chi})$ is the diagonalization of the vector $\bm{\chi}$. Note that the diagonal matrix $\diag(\bm{\chi})$ is symmetric and orthonormal. In \Cref{thm:epi_description}, we define $\dxhat$ as the set of all optimal solutions to the Lagrangian dual \eqref{lag_prob} for ordinary Lagrangian cuts. Thus, ReLU Lagrangian cuts coincide with ordinary Lagrangian cuts for purely binary variables. The orthonormal transformation $\diag(\bm{\chi})$ preserves the polyhedral structure-- for example, the extreme points and the extreme rays-- of a polyhedron. Although Lagrangian cuts are sufficient to characterize the epigraphs of local recourse functions, many alternative Lagrangian cuts exist, since the set $\dxhat$ may not be a singleton. Our goal is to identify the strongest cuts that can reconstruct local convex hulls. This motivates us to seek facet-defining cuts of the local convex hulls.
\begin{definition}
A cut of the form $\theta \geq v + \bm \pi^\top \bm x$ is facet-defining for $\conv(\epi_{\barx}(\Q^s))$ if it is valid for $\conv(\epi_{\barx}(\Q^s))$ and the corresponding hyperplane $\theta = v + \bm \pi^\top \bm x$ defines a facet of $\conv(\epi_{\barx}(\Q^s))$.
\end{definition}
The following theorem demonstrates that the optimal solutions to the outer supremum of \eqref{lag_prob} form a polyhedron, with facet-defining cuts corresponding to its extreme points. 
\begin{restatable}{theorem}{polar}\label{polar}
The following properties hold for the set $\dxhat$ of dual optimal solutions and Lagrangian cuts generated at $\hatx\in \barx$: (i) The set $\dxhat$ is a polyhedron; (ii) A Lagrangian cut $\theta \geq \co(\Q_s)(\hatx) + \bm \pi^\top (\bm x-\hatx)$ is facet-defining if and only if $\bm \pi$ is an extreme point of $\dxhat$; (iii) The recession cone of $\dxhat$ is the normal cone of $\conv(\barx)$ at $\hatx$, denoted by $N_{\conv(\barx)}(\hatx)$; and (iv) The local convex hull can be described using the facet-defining Lagrangian cuts:
$$\conv(\epi_{\barx}(\Q_s)) = \left\{(\bm x,\theta)\in\conv(\barx)\times\Re: \theta \geq \co(\Q_s)(\hatx) + \bm \pi^\top (\bm x-\hatx), \forall \hatx\in \barx, \bm{\pi}\in\ext(\dxhat)\right\}.$$
\end{restatable}
\begin{proof}
See Appendix \ref{pf:polar}. \qed
\end{proof}

This result also suggests that as the cut coefficients shift in the direction of the recession cone of set $\dxhat$, the resulting Lagrangian cuts may serve as valid inequalities only for $\conv(\barx)$. However, these cuts do not contribute to the outer approximations of local epigraphs and should be avoided.

\subsection{Cut strengthening: theoretical foundation}\label{sec_stren_cut}
According to Proposition~\ref{prop_L_shape_ReLU}, we know that any integer L-shaped cut is a special ReLU Lagrangian cut, which is known to be weak \citep{zou2019stochastic}.
In this subsection, we present a method for deriving a stronger ReLU Lagrangian cut by strengthening the integer L-shaped cut coefficients. 

Let us define $\hat{\bm\pi} =  (L - \Q_s(\hatx))\bm \chi$. We aim to find a proper $\bm\eta$ such that the strengthened cut:
\begin{equation}\label{stren_cut}
  \theta \geq \Q_s(\hatx) + 
 \left(\hat{\bm\pi}+\bm{\eta}\right)^\top(\bm x-\hatx),  
\end{equation}
can be close to a non-trivial facet of $\conv(\epi_{X}(\Q_s))$. 
By \Cref{polar}, we know that a strengthened cut \eqref{stren_cut} is a valid ReLU Lagrangian cut if and only if $\hat{\bm \pi} + \bm \eta \in \dxhat$. Therefore, it suffices to find $\bm \eta$ such that $\hat{\bm \pi} + \bm \eta$ is an extreme point of $\dxhat$. To identify an extreme point of a polyhedron, a natural approach is to solve a linear program:
\begin{equation}
\tag{Stren}
\label{stren}
\min \left\{\bm a^\top\bm \eta:
\hat{\bm\pi}+\bm \eta  \in \dxhat\right\}.
\end{equation}

Note that the feasible region of this problem, denoted by $\dxhat-\hat{\bm\pi}:= \{\bm\eta: \Q_s(\hatx)\leq \Q_s(\bm x)+(\hat{\bm\pi}+\bm\eta)^\top(\hatx-\bm x), \forall \bm x\in X\} = \{\bm\eta: (\hatx-\bm x)^\top \bm\eta \geq \Q_s(\hatx)+\hat{\bm\pi}^\top (\hatx-\bm x)-\Q_s(\bm x), \forall \bm x\in X\}$ is unbounded according to \Cref{polar}.  We illustrate this with the following example:
\begin{example}\label{eg:unbounded_stren_prob}
Consider the local recourse function $\Q_s(x_1,x_2) = 
\min\{2y_1+2y_2: 0.2y_1+y_2+x_1+0.5x_2\geq 2.4, \bm y\in\{0,1,2\}^2\}$ with binary first-stage decisions $\bm x\in\{0,1\}^2$. This function takes the values $\Q_s(0,0) = 8, \ \Q_s(1,0)=4, \ \Q_s(0,1)=4$ and $\Q_s(1,1)=2$. Let $\hatx = (1,0)^\top$ and $L=0$ be a lower bound of $\Q_s$. We can derive an L-shaped cut $\theta \geq 4+4(x_1-1)-4x_2$. To enhance this cut, we consider the feasible region of the strengthening problem \eqref{stren}, given by 
 $\dxhat-\hat{\bm\pi} = \left\{\bm \eta:\eta_1\geq -8, \eta_1-\eta_2 \geq -8, \eta_2\leq 2\right\}$ (see \Cref{eg:strategy1:a}). 
 It has two extreme points $(-8,0)^\top$  
 and $(-6, 2)^\top$, corresponding to two facet-defining cuts:
 \begin{subequations}
 \begin{equation}\label{fct-def_-4_-4}
 \theta\geq 4-4(x_1-1)-4x_2,
 \end{equation}
passing through $(0,0,8)^\top$, 
 $(1,0,4)^\top$ and $(0,1,4)^\top$, and
 \begin{equation}\label{fct-def_-2_-2}
 \theta\geq 4-2(x_1-1)-2x_2,
 \end{equation}
  \end{subequations}
passing through $(1,0,4)^\top$, $(0,1,4)^\top$ and $(1,1,2)^\top$. Meanwhile, $\dxhat-\hat{\bm\pi}$ is unbounded with extreme rays $(1,0)^\top$ and $(0,-1)^\top$. The strengthened cut $\theta \geq 4+(4+\eta_1)(x_1-1)+(-4+\eta_2)(x_2-0)$ approaches $x_1\leq 1$ as $\eta_1\rightarrow \infty$ and $x_2\geq 0$ as $\eta_2\rightarrow -\infty$.
\qedA
\end{example}

\begin{figure}[htbp]
 \vspace{-10pt}
\centering
\subfigure[\tiny $\dxhat-\hat{\bm\pi}$ when $L=0$]{
    \begin{tikzpicture}[scale=0.65]
        \begin{axis}[
            axis lines=middle,
            xlabel={$\eta_1$},
            ylabel={$\eta_2$},
            xtick={-8,-6},
            ytick={2,8},
            ymin=-5, ymax=10,
            xmin=-10, xmax=5,
            axis line style={-stealth},
            scatter/classes={
                a={mark=*,black}
            },
            axis equal image,
        ]
        
        \addplot[dashed,very thick] coordinates {(-10,-2)(-8,0) (0,8) (2,10)};
         \addplot[dashed,very thick] coordinates {(-10,2)(5,2)};
         \addplot[dashed,very thick] coordinates {(-8,10)(-8,-5)};

         \addplot [
        pattern=north east lines,
        draw=none,
    ] coordinates {(-8,0) (-6,2) (5,2) (5,-5) (-8,-5)} --cycle;

        \end{axis}
    \end{tikzpicture} \label{eg:strategy1:a}
}
\hfill
\subfigure[\tiny $\dxhat-\hat{\bm\pi}$ with constraints \eqref{reverse_normal_cone}]{
    \begin{tikzpicture}[scale=0.65]
        \begin{axis}[
            axis lines=middle,
            xlabel={$\eta_1$},
            ylabel={$\eta_2$},
            xtick={-8,-6},
            ytick={2,8},
            ymin=-5, ymax=10,
            xmin=-10, xmax=5,
            axis line style={-stealth},
            scatter/classes={
                a={mark=*,black}
            },
            axis equal image,
        ]
        
        \addplot[dashed,very thick] coordinates {(-10,-2)(-8,0) (0,8) (2,10)};
         \addplot[dashed,very thick] coordinates {(-10,2)(5,2)};
         \addplot[dashed,very thick] coordinates {(-8,10)(-8,-5)};
         \addplot[dashed,very thick] coordinates {(0,-5)(0,8)};
         \addplot[dashed,very thick,blue] coordinates {(-5,-5) (5,5)};
         \addplot[dashed,very thick,blue] coordinates {(-10,0)(5,0)};
         \addplot[dashed,very thick,blue] coordinates {(0,-5)(0,8)};

         \addplot [
        pattern=north east lines,
        draw=none,
    ] coordinates {(-8,0) (-6,2) (0,2) (0,0)} --cycle;

        \end{axis}
    \end{tikzpicture}\label{eg:strategy1:b}
    }
    \hfill
\subfigure[\tiny $\dxhat-\hat{\bm\pi}$ when $L=2$]{
    \begin{tikzpicture}[scale=0.65]
        \begin{axis}[
            axis lines=middle,
            xlabel={$\eta_1$},
            ylabel={$\eta_2$},
            xtick={-6,-4},
            ytick={0,-2},
            ymin=-5, ymax=10,
            xmin=-10, xmax=5,
            axis line style={-stealth},
            scatter/classes={
                a={mark=*,black}
            },
            axis equal image,
        ]
        
        \addplot[dashed,very thick] coordinates {(-10,-6)(-4,0) (0,4) (2,6)};
         \addplot[dashed,very thick] coordinates {(-10,0)(5,0)};
         \addplot[dashed,very thick] coordinates {(-6,10)(-6,-5)};

         \addplot [
        pattern=north east lines,
        draw=none,
    ] coordinates {(-4,0) (5,0) (5,-5) (-6,-5) (-6,-2)} --cycle;
        \end{axis}
    \end{tikzpicture}\label{eg:strategy1:c}
  }
\hfill
\subfigure[\tiny $\dxhat-\hat{\bm\pi}$ with constraints \eqref{reverse_normal_cone}]{
    \begin{tikzpicture}[scale=0.65]
        \begin{axis}[
            axis lines=middle,
            xlabel={$\eta_1$},
            ylabel={$\eta_2$},
            xtick={-6,-4},
            ytick={0,-2},
            ymin=-5, ymax=10,
            xmin=-10, xmax=5,
            axis line style={-stealth},
            scatter/classes={
                a={mark=*,black}
            },
            axis equal image,
        ]
        
        \addplot[dashed,very thick] coordinates {(-10,-6)(-4,0) (0,4) (2,6)};
         \addplot[dashed,very thick] coordinates {(-10,0)(5,0)};
         \addplot[dashed,very thick] coordinates {(-6,10)(-6,-5)};
       
        \addplot[dashed,very thick,blue] coordinates {(0,10)(0,-5)};
        \addplot[dashed,very thick,blue] coordinates {(-10,0)(5,0)};
        
         \addplot [
        pattern=north east lines,
        draw=none,
    ] coordinates {(-4,0.2) (0,0.2) (0,-0.2) (-4,-0.2)} --cycle;

        \end{axis}
    \end{tikzpicture}\label{eg:strategy1:d}
   }
\caption{\centering The illustration of Examples \ref{eg:unbounded_stren_prob} and \ref{eg:strategy1}}
 \vspace{-10pt}
\end{figure}
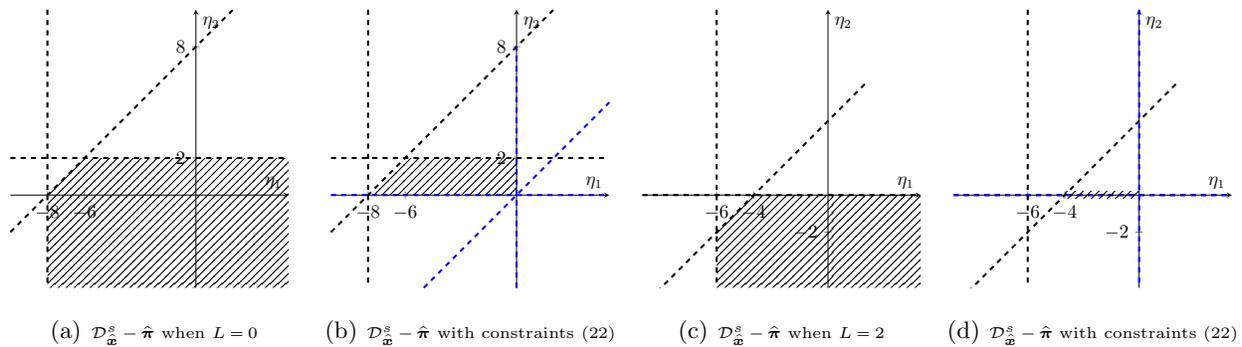

We propose two strategies to prevent the unboundedness of the strengthening problem \eqref{stren}. 
% when $\conv(X)$ is full dimensional, for example, $X=\{0,1\}^n$. This occurs when the first-stage problem is unconstrained. In this case, $\reccone(\dxhat-\hat{\bm\pi})=\reccone(\dxhat) = N_{\conv(X)}(\hatx) = \cone\{\bm e_i\}.$
\begin{itemize}
\item {\bf Strategy 1: Introduce additional constraints to bound the feasible region $F_{\eta}$}

When set $X$ is full-dimensional\footnotemark\footnotetext{If set $X$ is not full-dimensional, we can choose $\bm{\eta}$ to be in the null space of affine hull of set $X$.}, we can restrict $\bm{\eta}$ to lie in the reverse direction of the normal cone of $X$. Specifically, the constraints we add are
\begin{equation}\label{reverse_normal_cone}
(\hatx-\bm x)^\top\bm\eta \leq 0, \quad \forall \bm x\in X.
\end{equation}
In this way, the recession cone of the feasible region becomes $\{\bm d: (\hatx-\bm x)^\top \bm d \geq 0, (\hatx-\bm x)^\top \bm d \leq 0,  \forall \bm x\in X\}=\{\bm d: (\hatx-\bm x)^\top \bm d = 0,\forall \bm x\in X\}=\{\bm{0}\}$, where the last equality follows from the full dimensionality of the set $X$. 
However, this approach may eliminate extreme points that satisfy $(\hatx-\bm x)^\top \bm \eta > 0$ for some $\bm x\in X$. In general, we can add box constraints:
\begin{equation}\label{bound_eta}
\underline{M}_i\leq \eta_i\leq \overline{M}_i, \quad i\in[n].
\end{equation} 
One possibility is to ensure that $\mathrm{diag}(\bm\chi)^{-1}\bm\eta \geq \bm{0}$ and $\mathrm{diag}(\bm\chi)^{-1}\bm\eta \leq (\Q_s(\hatx)-L)\bm{1}$, which implies that the absolute values of the entries of $\bm{\eta}$ are in the range $[0, \Q_s(\hatx)-L]$. When $|\underline{M}_i|$ and $|\overline{M}_i|$ are large enough, we can retain all extreme points of the original set $\dxhat-\hat{\bm\pi}$.
A drawback of this strategy is that additional constraints can cause extra extreme points.
\begin{example}\label{eg:strategy1}
In \Cref{eg:unbounded_stren_prob}, at $\hatx=(1,0)^\top$, 
adding constraints \eqref{reverse_normal_cone} to $\dxhat-\hat{\bm\pi}$, the feasible region of \eqref{stren} becomes $\{\bm\eta: -8\leq \eta_1\leq 0, -8\leq \eta_1-\eta_2\leq 0, 0\leq \eta_2\leq 2\}$.
It has the original extreme points and includes two new extreme points $(0,0)^\top$ and $(0,2)^\top$ (see \Cref{eg:strategy1:b}).
When adding constraints $\eta_1\leq 0, \eta_2\geq 0$ to bound the entries of $\bm\eta$, we get the same result.

If we take $L=2$ as the lower bound of the recourse function $\Q$, the L-shaped cut becomes 
$\theta \geq 4+2(x_1-1)-2x_2$, and 
$\dxhat-\hat{\bm\pi} = \{\bm\eta: \eta_1 \geq -6, \eta_1-\eta_2 \geq -4, \eta_2\leq 0\}$. 
It has extreme points $(-6,-2)^\top$ and $(-4, 0)^\top$ (see \Cref{eg:strategy1:c}).
By adding constraints \eqref{reverse_normal_cone}, the feasible region of \eqref{stren} becomes $\{\bm\eta: -4\leq \eta_1\leq 0, \eta_2=0\}$, which excludes the extreme point $(-6,-2)^\top$ (see \Cref{eg:strategy1:d}). If we bound the entries of $\bm\eta$ by adding upper bounds and lower bounds with sufficiently large absolute values, for example, $|\eta_1|\leq 6,|\eta_2|\leq 6$, the feasible region of \eqref{stren} becomes $F_{\bm\eta}^\prime = \{-6\leq \eta_1\leq 6, -6\leq \eta_2\leq 0, \eta_1-\eta_2\geq 4\}$. By doing so, we still have the original two extreme points while adding three new ones-- $(6, 0)^\top$, $(6, -6)^\top$ and $(-6, -6)^\top$. They correspond to cuts $\theta\geq 4+8(x_1-1)-2x_2$, $\theta\geq 4+8(x_1-1)-8x_2$,
and $\theta\geq 4-4(x_1-1)-8x_2$, which are dominated by cuts \eqref{fct-def_-4_-4}, \eqref{fct-def_-2_-2}.
\qedA
\end{example}

\item  {\bf Strategy 2: Choose appropriate coefficients of the objective function}\label{strategy2}

\noindent
We restrict the objective coefficients to lie in a reverse direction of the tangent cone of $\conv(X)$ at $\hatx$, i.e., $-\bm a\in \mathcal{T}_{\conv(X)}(\hatx):= \cone\{\bm x-\hatx: \bm x\in X\}$. 
\begin{restatable}{proposition}{revtag}\label{prop:-a}
The strengthening problem \eqref{stren} is bounded if and only if $-\bm a\in \mathcal{T}_{\conv(X)}(\hatx)$. 
\end{restatable}
\begin{proof}
See Appendix \ref{pf:-a}. \qed
\end{proof}

A practical choice is to let $\bm{a} = \sum_{i\in I} (\hatx-\bm x^i)$, where $(\bm x^i)$'s are all explored first-stage decisions. 
\end{itemize}
\begin{example}
In \Cref{eg:unbounded_stren_prob}, since $\hat{x}_1=1$ and $\hat{x}_2=0$, we let $a_1=-1$ and $a_2=1$. The strengthening problem \eqref{stren} then becomes
$\min\{\eta_1-\eta_2:\eta_1\geq -8, \eta_1-\eta_2 \geq -8, \eta_2\leq 2\}$.
Both extreme points $(-8,0)^\top$ and $(-6, 2)^\top$ and their convex combinations are optimal.
\qedA
\end{example}

\subsection{Cut strengthening -- practical implementations based on LP relaxation}
In this subsection, we integrate Strategies 1 and 2 within a practical framework to efficiently strengthen the ReLU Lagrangian cuts based on L-shaped ones. Although the feasible region of \eqref{stren} is a polyhedron, it is generally NP-hard to separate from it. As an alternative, we replace the set $X$ with its LP relaxation, $X^{LP}$. Thus, rather than solving the original strengthening problem
\begin{equation*}
\min_{\bm\eta}\left\{\bm a^\top \bm\eta:
\min_{\bm x}\left\{ \Q_s(\bm x)+(\hat{\bm\pi}+\bm\eta)^\top(\hatx-\bm x):\bm x\in X\right\}\geq \Q_s(\hatx)\right\},
\end{equation*}
% $$\min\left\{\bm a^\top \bm\eta: \min_{\bm x} \left\{\Q_s(\bm x)+(\hat{\bm\pi}+\bm\eta)^\top(\hatx-\bm x):\bm x\in X\right\}\geq \Q_s(\hatx)\right\},$$
% which is equivalent to
% $$\min\left\{\bm a^\top \bm\eta: \min_{\bm x} \left\{\Q_s(\bm x)+(\hat{\bm\pi}+\bm\eta)^\top(\hatx-\bm x):\bm x\neq \hatx, \bm x\in X\right\}\geq \Q_s(\hatx)\right\},$$
we consider
% $$\min\left\{\bm a^\top \bm\eta: \min_{\bm x}\left\{\Q_s^{LP}(\bm x)+(\hat{\bm\pi}+\bm\eta)^\top(\hatx-\bm x):\sum_{i\in I}(1-x_i)+\sum_{i\notin I}x_i\geq 1, \bm x\in X^{LP}\right\}\geq \Q_s(\hatx)\right\},$$
\begin{equation}\label{stren_lp}
\min_{\bm\eta}\left\{\bm a^\top \bm\eta:
\min_{\bm x}\left\{\Q_s^{LP}(\bm x)+(\hat{\bm\pi}+\bm\eta)^\top(\hatx-\bm x):\bm x\in X^{LP}\right\}\geq \Q_s(\hatx)\right\},
\end{equation}
where 
$\Q^{LP}_s$ is defined by solving the LP relaxation of the local recourse problem \eqref{loc_rec}.
\begin{restatable}{proposition}{shrink}\label{prop:feas_eta}
    Any feasible solution of \eqref{stren_lp} is also feasible for \eqref{stren}.
\end{restatable}
\begin{proof}
See Appendix \ref{pf:shrink}. \qed
\end{proof}
This proposition shows that by relaxing the inner minimization problem, we can tighten the feasible region of the strengthening problem. When $\Q_s^{LP}(\hatx) < \Q_s(\hatx)$, we have $\min_{\bm x\in X^{LP}}\Q_s^{LP}(\bm x)+(\hat{\bm\pi}+\bm\eta)^\top(\hatx-\bm x)\leq \Q_s^{LP}(\hatx) < \Q_s(\hatx)$ for any $\bm \eta\in\Re^n$. Thus, it is possible that the resulting formulation \eqref{stren_lp} is infeasible. In this case, we need to improve the formulation \eqref{stren_lp}.

\begin{example}
Consider the local recourse function 
$\Q_1(x)=\min\{y_1+y_2:2y_1+y_2 \geq 3x+2, 0\leq y_1\leq 2, 0\leq y_2\leq 3, \bm y\in\Ze^2\}$.
It is easy to check that $\Q_1(1)=\Q^{LP}_1(1)=3$. Therefore, when $\hat{x}=1$, the relaxed strengthening problem \eqref{stren_lp} is feasible.
However, consider another recourse function:
$\Q_2(x)=\min\{y_1+y_2:2y_1+y_2 \geq 3x, 0\leq y_1\leq 2, 0\leq y_2\leq 3, \bm y\in\Ze^2\}$.
In this case
$\Q_2(1)=2$ while $\Q^{LP}_2(1)=\frac{3}{2}$. Thus, the relaxed strengthening problem \eqref{stren_lp} becomes infeasible.
% \begin{figure}
%     \centering
% \begin{minipage}{0.45\textwidth}
%     \centering
%     \begin{tikzpicture}[scale=0.7]
%         \begin{axis}[
%             axis lines=middle,
%             xlabel={$x$},
%             xtick={0,1},
%             ytick={0,1,2,3},
%             ymin=0, ymax=3.5,
%             xmin=0, xmax=1.2,
%             axis line style={-stealth},
%             scatter/classes={
%                 a={mark=*,black}
%             },
%         ]
%         \addplot[only marks, mark=*, black] coordinates {(0,1) (1,3)};     
%         \addplot[dashed, thick, blue] coordinates {(0,1) (2/3,2) (1,3)};     
%        \node at (axis cs:0, 1) [anchor=north west] {$\Q_1(x)$};
%        \node at (axis cs:2/3,2) [anchor=north west,blue] {$\Q^{LP}_1(x)$};
%      \end{axis}
%     \end{tikzpicture}
% \end{minipage}
% \hfill
% \begin{minipage}{0.45\textwidth}
%     \centering
%     \begin{tikzpicture}[scale=0.7]
%         \begin{axis}[
%             axis lines=middle,
%             xlabel={$x$},
%             xtick={0,1},
%             ytick={0,1,2,3},
%             ymin=0, ymax=3.5,
%             xmin=0, xmax=1.2,
%             axis line style={-stealth},
%             scatter/classes={
%                 a={mark=*,black}
%             },
%         ]
%         \addplot[only marks, mark=*, black] coordinates {(0,0) (1,2)};      
%         \addplot[dashed, thick, blue] coordinates {(0,0) (1,3/2)};      
%        \node at (axis cs:0, 0) [anchor=north west] {$\Q_2(x)$};
%        \node at (axis cs:1/2,3/4) [anchor=north west,blue] {$\Q^{LP}_2(x)$};
%      \end{axis}
%     \end{tikzpicture}
% \end{minipage}
% \end{figure}
\qedA
\end{example}
To address the feasibility issue, we incorporate the following no-good cut into the inner minimization problem of the formulation \eqref{stren_lp}:
$$\bm\chi^\top(\hatx-\bm x)\geq 1,$$
where we recall that $\chi_i = 2\hat{x}_i-1$ for each $i\in [n]$.
Next, we show that the strengthening problem \eqref{stren_lp} is always feasible after incorporating the no-good cut.
\begin{restatable}
{proposition}{nogood}\label{prop:nogood}
The set
\begin{equation}\label{F_eta_lp}
F_{s}^{LP}:=\left\{\bm\eta:\min_{\bm x}\left\{\Q_s^{LP}(\bm x)+(\hat{\bm\pi}+\bm\eta)^\top(\hatx-\bm x):\bm x\in\widetilde{X}\right\}\geq \Q_s(\hatx)\right\}
\end{equation}
is always nonempty, where $\widetilde{X}:=\{\bm x\in X^{LP}: \bm\chi^\top(\hatx-\bm x)\geq 1\}$.
\end{restatable}
\begin{proof}
See Appendix \ref{pf:nogood}. \qed
\end{proof}

The relaxed strengthening problem is then given by
\begin{align*}
\min_{\bm \eta}\left\{\bm a^\top \bm\eta:
\Q_s(\hatx)\leq \min_{\bm x,\bm y}\left\{\bm q^\top \bm y + (\hat{\bm\pi}+\bm\eta)^\top(\hatx-\bm x):
\bm A\bm x \geq \bm b,
\bm T^s\bm x + \bm W^s\bm y \geq  \bm h^s,\bm\chi^\top(\hatx-\bm x)\geq 1,\bm{0}\leq \bm x\leq \bm{1}
%&\qquad \bm y \in \mathbb{R}^{m_1+m_2}.
\right\}\right\}.
\end{align*}
Taking the dual of the minimization problem on the right-hand side of the constraint and invoking LP strong duality, we have
\begin{align*}
&\min_{\bm\eta}\bigg\{\bm a^\top \bm\eta:
\Q_s(\hatx)\leq (\hat{\bm \pi}+\bm\eta)^\top \hatx + \max_{\bm\tau,\bm\sigma,\rho,\bm\omega}\bigg\{\bm b^\top\bm\tau+(\bm h^s)^\top\bm \sigma+(1-\sum_{i\in[n]}\hat{x}_i)\rho+\bm{1}^\top\bm\omega:\\
&\bm A^\top\bm\tau+(\bm T^s)^\top\bm\sigma-\bm \chi\rho+\bm\omega\leq -(\hat{\bm\pi}+\bm\eta),
(\bm W^s)^\top \bm\sigma \leq \bm q,
\bm\tau\geq \bm{0}, \bm \sigma \geq \bm{0},\rho \geq 0, \bm\omega \leq \bm{0}\bigg\}\bigg\}.
\end{align*}
Replacing the maximization operator with the existential quantifier, we obtain an equivalent strengthening linear program:
\begin{equation} \tag{Stren LP}\label{stren_lp_ngc}
\begin{aligned}
&\min_{\bm\eta,\bm\tau,\bm\sigma,\rho,\bm\omega}\bigg\{ \bm a^\top \bm\eta:
\hatx^\top\bm\eta+\bm b^\top\bm\tau+(\bm h^s)^\top\bm \sigma+(1-\sum_{i\in[n]}\hat{x}_i)\rho+\bm{1}^\top\bm\omega\geq \Q_s(\hatx)+\hatx^\top\hat{\bm\pi},\\
& \bm A^\top\bm\tau+(\bm T^s)^\top\bm\sigma+\bm \chi\rho+\bm\omega+\bm\eta\leq -\hat{\bm\pi},
(\bm W^s)^\top \bm\sigma \leq \bm q,
\bm\tau\geq \bm{0}, \bm \sigma \geq \bm{0},\rho \geq 0,\bm\omega \leq 0\bigg\}.
\end{aligned}
\end{equation}

This LP can also be unbounded. Similar to the analysis of \eqref{stren}, the feasible region of $\bm\eta$ is unbounded, with its recession cone being the normal cone of $\conv(\widetilde{X})$ at point $\hatx$. As stated in \Cref{prop:feas_eta}, the feasible region of the LP-based formulation is a subset of that of the original strengthening problem. Therefore, the two strategies introduced in \Cref{sec_stren_cut} can effectively address the unboundedness issue here. Specifically, in Strategy 1, the additional constraints that ensure the boundedness of the feasible region in \eqref{stren} are also sufficient to guarantee that the feasible region of \eqref{stren_lp_ngc} is bounded. That is, when restricting $\bm\eta$ in the reverse direction of the recession cone, we add the following constraint to \eqref{stren_lp_ngc}:
\begin{align*}
    \min_{\bm x} \left\{ \bm x^\top \bm\eta:\bm A\bm x\geq \bm b, \bm\chi^\top(\hatx-\bm x)\geq 1,
    \bm{0}\leq \bm x\leq \bm{1}\right\} \geq \hatx^\top\bm\eta,
\end{align*}
which is different from \eqref{reverse_normal_cone} since we replace set $X^{LP}$ by a smaller set $\widetilde{X}$. 
Similarly, we can take the dual of the minimization problem on the left-hand side to make it compatible with the LP. Then, it is equivalent to adding the following constraints to \eqref{stren_lp_ngc}:
$$
\begin{aligned}
 &\bm b^\top\bm\tau_1+\left(1-\sum_{i\in[n]}\hat{x}_i\right)\rho_1+\bm{1}^\top \bm\omega_1 -\hatx^\top\bm\eta   \geq 0,\bm A^\top \bm\tau_1 + \bm\chi\rho_1+\bm\omega_1 -\bm\eta\leq\bm{0}, \bm\tau_1 \geq \bm{0}, \rho_1\geq 0,  \bm\omega\leq \bm{0}.
\end{aligned}
$$
In Strategy 2, for any objective coefficient $\bm a$ such that $-\bm a\in\mathcal{T}_{\conv(X)}(\hatx)$, we have $-\bm a\in\mathcal{T}_{\conv(\widetilde{X})}(\hatx)$ as $X\backslash\{\hatx\}\subseteq\widetilde{X}$. Thus, \eqref{stren_lp_ngc} is bounded according to \Cref{prop:-a}.

The next example illustrates the importance of the no-good cuts. % and how to implement the strengthening problem using the following example:
\begin{example}\label{eg:infeas_stren_prob}
Recall in \Cref{eg:unbounded_stren_prob}, the local recourse function is
$\Q_s(x_1,x_2) = \min_{\bm y}\{ 2y_1+2y_2:
0.2y_1+y_2+x_1+0.5x_2\geq 2.4,
\bm y\in\{0,1,2\}^2\}$ with binary first-stage decisions 
$\bm x\in \{0,1\}^2$.
At $\hatx=(1,1)$, we can generate an L-shaped cut
$\theta \geq 2+2(x_1-1)+2(x_2-1)$.
To strengthen this cut, we solve the following problem using Strategy 2 to avoid unboundedness:
\begin{equation}\label{eg1_opt_con}
\min_{\bm{\eta}} \left\{\eta_1+\eta_2:
  \min_{\bm x\in[0,1]^2}\left\{\Q_s^{LP}(\bm x)+(2+\eta_1)(1-x_1)+(2+
  \eta_2)(1-x_2)\right\}\geq 2\right\},
\end{equation}
which is equivalent to
\begin{align*}
\min_{\bm{\eta}}\{&\eta_1+\eta_2: 4+\eta_1+\eta_2+2.4\alpha+2\beta_1+2\beta_2+\gamma_1+\gamma_2\geq 2,
0.2\alpha+\beta_1\leq 2,
\alpha+\beta_2\leq 2,\\
&\alpha+\gamma_1\leq -(2+\eta_1),
0.5\alpha+\gamma_2\leq -(2+\eta_2),
\alpha\geq 0,
\bm{\beta}\leq \bm{0},
\bm{\gamma}\leq \bm{0}\}.
\end{align*}
This problem is infeasible. However, if we add a constraint $(1-x_1)+(1-x_2) \geq 1$ to the minimization problem in \eqref{eg1_opt_con},
the problem becomes feasible with optimal solution $(-3.8, -2.8)$. In this way, we obtain a stronger ReLU Lagrangian cut $\theta \geq 2-1.8(x_1-1)-0.8(x_2-1)$.
\qedA
\end{example}

To further enhance the relaxed strengthening problem \eqref{stren_lp}, we can incorporate additional valid inequalities of $\conv\{(\bm x, \theta) \in \epi_X(\Q_s):\bm x \neq \hatx\}$ into the inner minimization problem, similar to the no-good cuts. For example, we can add the following objective cuts.
\begin{definition}[Objective cuts]
For each scenario $s\in[N]$, let $L_s$ be a lower bound of $\Q_s$. Then, we can introduce the following objective cut $(\bm q^s)^\top y \geq L_s$ to the inner minimization problem of \eqref{stren_lp_ngc}, %:
% \begin{equation*}
%     ,
% \end{equation*}
where we assume that $L_s \geq L$. In fact, this lower bound $L_s$ can be obtained by solving the scenario problem
$L_s := \min_{\bm x\in X} \Q_s(\bm x).$
\end{definition}

\subsection{Implementation details}\label{sec:implementation_two_stage}
This subsection details the implementations of the ReLU Lagrangian cuts for stochastic integer programs with purely binary first-stage decisions within the framework of the basic cutting plane method (see \Cref{alg:binary}).
This procedure is also suitable for the branch-and-cut algorithm.

\begin{algorithm}[htbp]
\caption{ReLU Lagrangian cuts for binary first-stage decisions}\label{alg:binary}
\begin{algorithmic}[1]
\State \textbf{Input:} 
Master problem: $v^*=\min\{\bm c^\top \bm x + \theta: \bm A\bm x\geq \bm b, \bm x \in X\}$
and local recourse problems \eqref{loc_rec}
\State \textbf{Output:}
Optimal solution $x^*$
\State \textbf{Initialize:}
$lb\gets -\infty,\ ub\gets+\infty, \ i\gets 0$
\State Analyze the master problem and its LP relaxation to select an appropriate strategy for avoiding unboundedness
\While{stopping criterion not met}{
    \State Solve the master problem, obtain an optimal solution $\hatx$, and set $lb\gets v^*$
    \For{$s\in [N]$}
        \State Solve the local recourse problem and obtain $\Q_s(\hatx)$, and generate the L-shaped coefficients $\bm \pi_s^i$
        \State Update the strengthening problem \eqref{stren_lp_ngc} and solve it
        \If{ \eqref{stren_lp_ngc} is optimal}
        \State Let $\hat{\bm{\eta}}_s^i$ be an optimal solution, and  $\bm\pi_s^i \gets \bm \pi_s^i+\hat{\bm\eta}_s^i$
        \EndIf
    \EndFor
    \State Add the Lagrangian cut
        $\theta \geq p_s(\Q_s(\hatx)+ (\bm \pi_s^i)^\top \bm x)$
        to the master problem
    \State $ub\gets \min\{ub, \bm c^\top\hatx+\sum_{s\in[N]}p_s\Q_s(\hatx)\}$ and $i\gets i+1$
    \EndWhile
}
\end{algorithmic}
\end{algorithm}

Given that the feasible region is unbounded for any binary $\hat{\bm x}$, it is important to ensure that the strengthening problem \eqref{stren_lp_ngc} is bounded prior to the cut-strengthening procedure. We seek an approach, potentially combining both strategies introduced in \Cref{sec_stren_cut}, that preserves as many extreme points of the original feasible region as possible, as they correspond to cuts that are closer to being facet-defining. Since both strategies rely on an incumbent solution to determine the additional constraints and the objective function, we also aim to minimize the necessary updates to the strengthening problem in each iteration. First, we try to select appropriate objective coefficients following Strategy 2 without adjusting the feasible region. However, in practice, tracking the tangent cone at the incumbent solution can be challenging. It is desirable to add constraints and choose the objective coefficients to guarantee that the strengthening problem is bounded while its optimal solution effectively improves the cut. One approach is to set the objective coefficients $\bm a=\bm\chi$ and enforce $\mathrm{diag}(\bm\chi)^{-1}\bm\eta \geq \bm{0}$. In this way, we have $\bm\eta^\top(\bm x-\hatx) = (\mathrm{diag}(\bm\chi)^{-1}\bm\eta)^\top\mathrm{diag}(\bm\chi)(\bm x-\hatx)\geq 0$. If the strengthening problem is optimal, the resulting cut \eqref{stren_cut}, though not necessarily facet-defining, will dominate the original L-shaped cut. If the problem remains unbounded, we add constraints $\mathrm{diag}(\bm\chi)^{-1}\bm\eta \leq (\Q_s(\hatx)-L)\bm{1}$ to it and solve the problem again, as described in Strategy 1.

At each iteration, we need to update the optimality condition in the strengthening problem based on the
current first-stage decision $\hatx$ and its local recourse function value $\Q_s(\hatx)$. In addition, depending on which strategy is used to avoid unboundedness, we either choose the objective coefficients according to $\hatx$ or use random coefficients and set upper and lower bounds for the entries of $\bm\eta$ according to the L-shaped cuts' coefficients. Although we address the infeasibility of the strengthening problem caused by the integrality gap by adding a no-good cut to the inner minimization problem, the additional constraints introduced in Strategy 1 to prevent unboundedness may still render the problem infeasible. In this case, the L-shaped cuts serve as alternative ReLU Lagrangian cuts.

We can also combine them with other cuts that help accelerate the algorithm. For example, in \cite{zou2019stochastic}, the Lagrangian cuts can be implemented together with the Benders cuts and the strengthened Benders cuts. These cuts may not be tight but can improve the outer approximation of the expected epigraph.

\section{General Mixed Integer First-stage Decisions}\label{sec_general_mixed}
In this section, we study ReLU Lagrangian cuts for SMIPs with mixed integer first-stage decisions. Following the strategies outlined in Section \ref{sec_binary}, we begin with an easily obtainable valid cut. For stochastic programs with purely integer first-stage decisions, we convert them into models with binary first-stage variables and then apply the methods from Section \ref{sec_binary}. For programs with mixed integer first-stage decisions, we use the reverse norm cuts to derive the initial cuts. In both cases, similar to \Cref{sec_binary}, the cuts can be strengthened based on the dual optimality condition \eqref{gen_lag_opt_cond}.

\subsection{General integer first-stage decisions: binarization vs. non-binarization}
According to \cite{zou2019stochastic}, Lagrangian cuts can be applied to general integer first-stage decisions by binarizing integer variables. In this subsection, we show that while the ReLU Lagrangian cuts can effectively solve the problem without reformulating the original problem, binarization provides additional benefits. Specifically, it not only reduces the number of variables required to linearize the cuts but also achieves stronger outer approximations of the local epigraphs. Throughout this subsection, we assume that $\barx\subseteq \Ze^n$.

For a given integer first-stage decisions $\hatx\in\barx$, we can construct a $\Lambda$-shaped cut \eqref{gen_lshaped} as an initial ReLU Lagrangian cut. Alternatively, we can first binarize integer variables $\bm x$. Let $N_i = \lfloor \log_2 B_i\rfloor$ and represent $x_i$ as 
$x_i = \sum_{j\in [0,N_i]} 2^j\delta^i_{j}$ with $\bm{\delta}^i\in\{0,1\}^{N_i+1} $ for each $i\in [n].$ In the binarized space, we can then derive an L-shaped cut as an initial cut:
\begin{equation}\label{bi_lshaped}
\theta \geq \Q_s(\hatx) - (\Q_s(\hatx)-L)\sum_{i\in[n]}||\bm\delta^i-\hat{\bm\delta}^i||_1.
\end{equation}

We compare cuts \eqref{gen_lshaped} and \eqref{bi_lshaped} from two perspectives. First, Compare the additional variables introduced. In \eqref{gen_lshaped}, for each $i\in[n]$ and $\hat{x}_i\in[1,B_i-1]$, as $|x_i-\hat{x}_i|=(x_i-\hat{x}_i)^++(x_i-\hat{x}_i)^-$, we need a binary variable to linearize the terms $(x_i-\hat{x}_i)^+$ and $(x_i-\hat{x}_i)^-$ as shown in the constraint system \eqref{lin_gen_lag}. This results in up to $\sum_{i\in[n]}(B_i-1)$ additional binary variables. In contrast, it takes $\sum_{i\in[n]}(1+\lfloor \log_2 B_i\rfloor)$ binary variables to represent the first-stage decisions.

Second, we show that the L-shaped cuts derived in the binarized space perform better by proving the inclusion of the convex hulls of the subsystems described by two types of initial cuts.
\begin{restatable}{proposition}{lam}\label{prop:lam_l}
Given a first-stage feasible solution $\hatx\in\barx$, let us define $E_\Lambda^s =\left\{(\bm x, \theta)\in \mathcal{B}\times\Re: \eqref{gen_lshaped}\right\}$ and 
$$E_L^s = \left\{(\bm x, \theta)\in \mathcal{B}\times\Re: 
\exists \bm\delta\in \{0,1\}^{N_1+1}\times\cdots\times\{0,1\}^{N_n+1},x_i = \sum_{j\in [0,N_i]}2^j \delta^i_j,  \forall i\in[n],\eqref{bi_lshaped}
\right\},$$
where $\mathcal{B}:= \times_{i\in[n]}[0,B_i]$.
If $B_i\geq 3$ for all $i\in[n]$, then $\conv(E^s_\Lambda) \supseteq \conv(E^s_L)$.
\end{restatable}
\begin{proof}
See Appendix \ref{pf:lam_l}. \qed
\end{proof}
When $B_i \leq 2$ for some $i\in[n]$, the inclusion of the two sets may not hold, as shown in the following example. 
\begin{example}
Consider the local recourse function $\Q_s(x) = \min\{y:y\geq x, y\in\Ze\}$ for $x\in\{0,1,2\}$. Let $L=0$. When $\hat{x}=1$, we can derive a $\Lambda$-shaped cut $\theta \geq 1 + (0-1)|x-\hat{x}|$. For this cut, we have 
$$\conv(E_\Lambda^s) = \{(x,\theta): \theta \geq 1 -(\omega^++\omega^-), \omega^+-\omega^-=x-\hat{x}, 0\leq \omega^+\leq (2-1)z, 0\leq \omega^-\leq 1-z, z\in[0,1] \}.$$
It is easy to check that $\theta \geq 1-\max\{\hat{x},B-\hat{x}\}=0$ when $x=2$. Meanwhile, binarizing $x$ as $x=\delta_0 + 2\delta_1$, we can derive a cut $\theta \geq 1 + (0-1)(|\delta_0-1|+|\delta_1-0|)$. When $x=2$, we have $\theta \geq -1$. There exists a point $(2, -1)\in\conv(E_L^s)$ that is not in $\conv(E_\Lambda^s)$.\qedA
\end{example}
This example indicates that binarization generally provides stronger initial cuts before the cut-strengthening procedure. It is more effective as fewer additional binary variables are required. In addition, the strengthening strategy introduced in the previous section can be applied directly after binarization.

We close this subsection by providing the convex hull of the set $E_L^s$. 
\begin{restatable}{proposition}{convbox}
\label{prop:conv_box}
Suppose that integer $B_i = 2^{j_{1i}}+\cdots+2^{j_{\ell i}}+2^{N_i}$, and set $J_i = \{j_{1i},\ldots,j_{\ell i},N_i\}$ for each $i\in [n]$. Then, 
\begin{align*}
\conv(E_L^s) = \left\{
(\bm x,\theta)\in \times_{i\in[n]}[0, B_i]\times\Re:
\begin{aligned}
&\exists\bm\delta \in [0,1]^{N_1+1}\times\cdots\times[0,1]^{N_n+1},\eqref{bi_lshaped},\\
&x_i=\sum_{j\in [0,N_i]}2^j \delta^i_j \leq B_i,  \forall i\in[n],\\
&\delta^i_r+\sum_{\tau\in J_{ir}}\delta^i_{\tau}\leq |J_{ir}|,\forall r\in\{0,1,\ldots, N_i\}\backslash J_i,\forall i\in[n] 
\end{aligned}\right\},
\end{align*}
where we let $J_{ir}=\{\ell\in J_i:\ell > r\}$.
\end{restatable}
\begin{proof}
See Appendix \ref{pf:conv_box}. \qed
\end{proof}

\subsection{Mixed integer first-stage decisions}\label{subsec:mixed_integer}
For mixed integer first-stage decisions, according to \cite{zou2019stochastic}, binarization may still remain an option and the number of binary variables required can be bounded in terms of the desired approximation accuracy. However, this approach may introduce an excessive number of additional binary variables and potentially fail to achieve an exact optimal solution due to binarization. Given these issues, we propose solving the problem directly using ReLU Lagrangian cuts.

Similar to the L-shaped and $\Lambda$-shaped cuts used for purely binary and integer variables, we first develop an initial cut. According to \Cref{thm:gen_lag_dual}, for any feasible first-stage decisions $\hatx$, there exists a $\rho^*>0$ such that $(-\rho^*\bm{1}, -\rho^*\bm{1})$ is optimal for the dual problem. This allows us to derive a cut of the form $\theta\geq \Q_s(\hatx) - \rho^* ||\bm x-\hatx||_1$. If the recourse function is Lipschitz continuous, we use the Lipschitz constant as $\rho^*$ and derive a reverse norm cut. Otherwise, $\rho^*$ may need to be adjusted based on the given $\hatx$, which can be determined using binary search.

When the second-stage feasible region has a special structure, such as a knapsack-constrained set, we can determine a valid $\rho^*$ by following the procedure described in the proof of \Cref{thm:gen_lag_dual}. First, we relax the first-stage feasible region to a larger set $S$ with fewer constraints, which facilitates the identification of extreme points $\{(\bm x^k, \bm y^k, \omega_k)\}_{k \in K}$ of $\conv(F_s(S))$ with set $F_s(S) = \{(\bm x, \bm y, \omega) \in \Re^n \times \Ze^{m_1} \times \Re^{m_2} \times \Re: \bm T^s \bm x + \bm W^s \bm y \geq \bm h^s, \bm x \in S\}$. We then compute the distance $d \leq \min\{||\bm x^k - \hatx||_1 : k \in K, \bm x^k \neq \hatx\}$ and let
\begin{equation}\label{closed_form_rho}
\rho^* = \frac{\Q_s(\hatx)-L}{d}.
\end{equation}
Letting $\bm\pi^+=\bm\pi^-=-\rho^*\bm{1}$, we find an optimal solution to $\sup_{\bm\pi^+,\bm\pi^-\in\Re^n}\inf_{\bm x\in S}L_s(\bm x, \bm\pi^+,\bm\pi^-;\hatx)$. According to \Cref{coro:expand_feas}, it is also optimal for the original dual problem \eqref{gen_lag_dual_ori}. In this way, we can derive an initial ReLU Lagrangian cut:
\begin{equation}\label{mi_init_cut}
\theta \geq \Q_s(\hatx) - \rho^* \left(\sum_{i\in[n]}(x_i-\hat{x}_i)^+ + \sum_{i\in[n]}(x_i-\hat{x}_i)^-\right)=\Q_s(\hatx) - \rho^* ||\bm x-\hatx||_1.
\end{equation}
The following example illustrates how to find such $\rho^*$ using this approach.
\begin{example}\label{eg:mi}
Let $\Q_s(\bm x)=\min\{y:y \geq x_1+x_2, y\in \Ze\}$, where the domain $\barx=\{\bm x\in\Ze\times\Re: 0\leq x_1\leq 2, 0\leq x_2\leq 2\}$, $\hatx = (1,1)^\top$ and $L=0$. We relax set $\barx$ to be set $S=\{\bm x\in\Re^2:||\bm x-\hatx||_1 \leq 2\}$. Then, we have $F_s(S)=\{(\bm x, y, \omega)\in \Re^2\times\Ze\times \Re:  ||\bm x-\hatx||_1 \leq 2, ||\bm x-\hatx||_1 \leq \omega, y\geq x_1+x_2\}=\{(\bm x, y, \omega)\in \Re^2\times\Ze\times \Re: x_1+x_2-2\leq\omega, x_1-x_2\leq \omega, x_2-x_1\leq \omega, -x_1-x_2+2\leq\omega, 0\leq\omega\leq 2, y\geq x_1+x_2\}$. This is an integral polyhedron, i.e., $\conv(F_s(S)) = \{(\bm x, y, \omega): x_1+x_2-2\leq\omega, x_1-x_2\leq \omega, x_2-x_1\leq \omega, -x_1-x_2+2\leq\omega, 0\leq\omega\leq 2, y\geq x_1+x_2\}$ with extreme points $(-1, 1, 0, 2)^\top, (1, -1, 0, 2)^\top, (1,3,4, 2)^\top, (3,1,4,2)^\top$ and $(1,1,2,0)^\top$. Then $d=\min\{||\bm x^k-\hatx||_1:k\in K, \bm x^k\neq \hatx\}=2$. We let $\rho = \frac{2(2-0)}{2}=2$. The initial ReLU Lagrangian cut we can derive is $\theta \geq 2-2||\bm x-\hatx||_1$.
\qedA
\end{example}

To strengthen this initial ReLU cut, we lift the first-stage feasible region $\barx$ to a higher-dimensional space by considering the set $\Omega_{\hatx}:=\{(\bm\omega^+, \bm\omega^-): \eqref{mip_rep_omega},\eqref{mip_rep_z}, \bm x\in \barx\}$. Then, we redefine the recourse function $\Q_s(\bm\omega^+, \bm\omega^-) := \Q_s(\hatx+\bm\omega^+-\bm\omega^-)$. With this transformation, the ReLU Lagrangian cuts generated at $\hatx$ become Lagrangian cuts generated at $(\bm{0},\bm{0})$ for the lifted epigraphical set $\epi_{\Omega_{\hatx}}(\Q_s)$. Therefore, the properties of the Lagrangian cuts discussed in \Cref{sec_ReLU_cuts} can be directly applied to the transformed problem.

Based on \Cref{polar}, we can define a ``facet-defining" cut generated at $\hatx$ whose coefficients correspond to an extreme point of the set $\pixhat$ (recall that set $\pixhat$ is defined in \Cref{def_pihat}). It also implies that in the lifted space, the number of affinely independent points at which a ReLU Lagrangian cut is tight is bounded by the dimension of the lifted feasible region $\Omega_{\hatx}$. This further explains why certain ReLU Lagrangian cuts can be stronger than augmented Lagrangian cuts, where the latter, due to the symmetry of the $\ell_1$-norm, may fail to be ``facet-defining" in the lifted space. 
\begin{example}
In Example \ref{alc_vs_ReLU}, let $\hatx=(1,2)^\top$. We have $(1,2)^\top= \hatx + (0,0)^\top- (0,0)^\top$, $(1,0)^\top=\hatx+(0,0)^\top-(0,2)^\top$, $(0,1)^\top = \hatx + (0,0)^\top-(1,1)^\top$, $(0,2)^\top=\hatx + (0,0)^\top-(1,0)^\top$, $(0,3)^\top = \hatx + (0,1)^\top-(1,0)^\top$, $(1,4)^\top=\hatx +(0,2)^\top-(0,0)^\top$, $(2,3)^\top=\hatx+(1,1)^\top-(0,0)^\top$, $(2,2)^\top=\hatx+(1,0)^\top-(0,0)^\top$, $(2,1)^\top=\hatx+(1,0)^\top-(0,1)^\top$. The first-stage feasible region can transformed to the set $\Omega_{\hatx}=\{(0,0,0,0)^\top, (0,0,0,2)^\top, (0,0,1,1)^\top, (0,0,1,0)^\top, (0,1,1,0)^\top, (0,2,0,0)^\top, (1,1,0,0)^\top, (1,0,0,0)^\top$, $(1,0,0,1)^\top\}$ in the lifted space. The cut \eqref{lifted_facet_defining} is facet-defining as it passes through five points $(0,0,0,0,10)^\top, (0,0,0,2,5)^\top, (0,0,1,0,2)^\top, (0,2,0,0,1)^\top,(1,0,0,0,4)^\top$ that are affinely independent in the convex hull of the epigraph that is 5-dimensional. 
Meanwhile, $\pixhat=\{(\pi_1^+,\pi_2^+, \pi_1^-,\pi_2^-): \pi_1^-+\pi_2^- \leq -7, \pi_1^- \leq -8, \pi_2^+ + \pi_1^-\leq -9, 2\pi_2^- \leq -5, \pi_2^- \leq -\frac{5}{2}, \pi_2^+ \leq -\frac{9}{2}, 2\pi_2^+ \leq -9, -\pi_2^- \leq -5, \pi_1^+ \leq -6,\pi_1^+ +\pi_2^+ \leq -7\}=\{(\pi_1^+,\pi_2^+, \pi_1^-,\pi_2^-):\pi_1^+ \leq -6, \pi_2^+ \leq -\frac{9}{2}, \pi_1^- \leq -8, \pi_2^- \leq \frac{5}{2}\}$. The vector of the coefficients of \eqref{lifted_facet_defining} is the only extreme point of $\pixhat$.

For augmented Lagrangian cuts, the cut coefficients are selected from a more restricted set $\pixhat \cap \{(\pi_1^+, \pi_2^+, \pi_1^-, \pi_2^-): \pi_1^+ + \pi_1^- = \pi_2^+ + \pi_2^- \leq 0\}$. It can be shown that the extreme points of set $\pixhat$ are eliminated by the extra constraints, and the restricted set contains two new extreme points that correspond to the coefficients of cuts \eqref{ext_agc_1} and \eqref{ext_agc_2}. These new extreme points satisfy three linearly independent constraints in $\pixhat$ with equality, and the corresponding cuts are tight at only four affinely independent points in the convex hull of the epigraph.
\qedA
\end{example}

This example suggests that the cut coefficients $(\bm\pi^+, \bm\pi^-)$ of a strongest augmented Lagrangian cuts have $n-1$ linearly independent equality constraints from the set $\{(\bm\pi^+, \bm\pi^-):\pi_i^+ + \pi_i^- = \pi_j^+ + \pi_j^-\leq 0, \forall i,j\in[n]\}$ and require $n+1$ linearly independent inequalities from $\pixhat$ that can be active for the cut coefficients $(\bm\pi^+, \bm\pi^-)$. Therefore, the number of linearly independent points in the epigraph at which this augmented Lagrangian cut is tight is at most $n+1$, which is clearly not facet-defining. Hence, more augmented Lagrangian cuts may be needed to fully recover the epigraph, compared to ReLU Lagrangian cuts.

Similar to the approach being applied to purely binary first-stage decisions, we can enhance the initial cuts \eqref{mi_init_cut} using the following strengthening problem in the lifted space:
\begin{align*}
\max_{\bm\eta^+,\bm\eta^-} \bigg\{(\bm a^+)^\top\bm\eta^+ + (\bm a^-)^\top\bm\eta^-:
 \Q_s(\hatx) \leq \min_{\bm x\in\barx} \Q_s(\bm x)+\sum_{i\in[n]}(\rho^* - \eta^+_i)(x_i-\hat{x}_i)^+ +\sum_{i\in[n]}(\rho^* - \eta^-_i)(x_i-\hat{x}_i)^-\bigg\},
\end{align*}
which is equivalent to
\begin{align*}
\max_{\bm\eta^+,\bm\eta^-} \bigg\{ &(\bm a^+)^\top\bm\eta^+ + (\bm a^-)^\top\bm\eta^-:
 \Q_s(\hatx) \leq \min_{\bm x, \bm y, \bm\omega^+,\bm\omega^-,\bm z} \bigg\{(\bm q^s)^\top \bm y+\sum_{i\in[n]}(\rho^* - \eta^+_i)\omega_i^+ +\sum_{i\in[n]}(\rho^* - \eta^-_i)\omega_i^-:\\
&\bm A\bm x\geq \bm b,
\bm T^s \bm x+\bm W^s\bm y \geq \bm h^s, \omega_i^+ -\omega_i^- = x_i-\hat{x}_i, \forall i\in[n],
0\leq \omega_i^+ \leq (B_i-\hat{x}_i) z_i, \forall i\in[n],\\
&0\leq \omega_i^- \leq \hat{x}_i(1- z_i), \forall i\in[n],
\bm x\in \Ze^{n_1}\times\Re^{n_2},
 \bm y\in \Ze^{m_1}\times\Re^{m_2},
 \bm z\in\{0,1\}^n\bigg\}\bigg\}.
\end{align*}
Relaxing the integrality constraints,
taking the dual of the inner minimization problem and invoking LP strong duality yields
\begin{equation}\label{stren_mixed_integer}
\begin{aligned}
\max_{\bm\eta^+,\bm\eta^-} 
\bigg\{&(\bm a^+)^\top\bm\eta^+ + (\bm a^-)^\top\bm\eta^-:\bm b^\top \bm\tau + (\bm h^s)^\top \bm\sigma -\hatx^\top \bm\gamma + \hatx^\top \bm\psi+\bm{1}^\top \bm\kappa \geq \Q_s(\hatx),\\
&
\bm A^\top\bm\tau + (\bm T^s)^\top\bm\sigma -\bm\gamma = \bm{0},
(\bm W^s)^\top\bm\sigma = \bm q^s,
\bm\gamma +\bm\phi + \bm\eta^+ \leq \rho^*\bm{1},\bm\gamma +\bm\psi + \bm\eta^- \leq \rho^*\bm{1}, \\
&
(\hat{x}_i-B_i)\phi_i + \hat{x}_i\psi_i + \kappa_i \leq 0,  \forall i\in[n],
\bm \tau\geq \bm{0},\bm \sigma \geq \bm{0},
\bm \phi\leq \bm{0}, \bm \psi \leq \bm{0}, \bm\kappa \leq \bm{0}\bigg\}.
\end{aligned}
\end{equation}
If this problem has an optimal solution with $\bm\eta^+ = \hat{\bm\eta}^+$ and $\bm\eta^- = \hat{\bm\eta}^-$, we can derive a stronger cut
\begin{equation*}
\theta \geq \Q_s(\hatx) + \sum_{i\in[n]}(\hat{\eta}_i^+-\rho^*)(x_i-\hat{x}_i)^+ + \sum_{i\in[n]}(\hat{\eta}_i^--\rho^*)(x_i-\hat{x}_i)^-.
\end{equation*}

The strategies described in \Cref{sec_stren_cut} can be used to address the unboundedness. Other techniques, such as objective cuts, can also be used to improve the LP-based strengthening problem. Below is an illustrative example.
\begin{example}
We use the strengthening problem to improve the ReLU Lagrangian cuts derived in Example \ref{eg:mi}. We solve the following relaxed strengthening problem:
\begin{align*}
\min_{\bm\eta^+,\bm\eta^-} \bigg\{&\sum_{i\in[2]}\eta_i^+ + \sum_{i\in[2]}\eta_i^-:
2 \leq \min_{\bm x,y} \{y + \sum_{i\in[2]}(2-\eta_i^+)(x_i-\hat{x}_i)^++ \sum_{i\in[2]}(2-\eta_i^-)(x_i-\hat{x}_i)^-:
y\geq x_1+x_2,\\
&0\leq x_1\leq 2, 0\leq x_2\leq 2\}\bigg\},
\end{align*}
which is equivalent to
\begin{align*}
\min_{\bm\eta^+,\bm\eta^-} \bigg\{&\sum_{i\in[2]}\eta_i^+ + \sum_{i\in[2]}\eta_i^-:
2\tau_1+2\tau_2-\gamma_1-\gamma_2+\psi_1+\psi_2+\kappa_1+\kappa_2 \geq 2,
\tau_i -\sigma -\gamma_i \leq 0, \forall i\in[2],\\
&\sigma = 1,\gamma_i +\phi_i + \eta_i^- \leq 2, \forall i\in[2],-\gamma_i +\psi_i + \eta_i^-\leq 2, \forall i\in[2],
-\phi_i+\psi_i +\kappa_i \leq 0, \forall i\in[2], \\
&\sigma \geq 0,
\bm\tau\leq \bm{0}, \bm\phi\leq \bm{0}, \bm\psi\leq \bm{0}, \bm\kappa \leq \bm{0}\bigg\}.
\end{align*}
The optimal solution is $\hat{\bm\eta}^+ = (3, 3)^\top$ and $\hat{\bm\eta}^- = (1, 1)^\top$, which corresponds to the cut $\theta \geq 2 + \sum_{i\in[n]} (x_i-\hat{x}_i)^+ -  \sum_{i\in[n]} (x_i-\hat{x}_i)^-$. This, in fact, is a facet-defining cut in the lifted space.
\qedA
\end{example}

\begin{algorithm}[htbp]
\caption{ReLU Lagrangian cuts for two-stage SMIP}
\begin{algorithmic}[1]\label{alg:mip}
\State \textbf{Input:} 
Master problem: $v^*=\min_{\bm{x},\theta}\{\bm c^\top \bm x + \theta: \bm A\bm x\geq \bm b, \bm x \in \barx\}$
and local recourse problems \eqref{loc_rec}
\State \textbf{Output:}
Optimal solution $x^*$
\State \textbf{Initialize:}
$lb\gets -\infty, ub\gets+\infty, iter \gets 0$
\If{$\barx \subseteq \{0,1\}^n$}
\State Solve it using Algorithm \ref{alg:binary}
\ElsIf{$\barx \subseteq \mathbb{Z}^n$}
\State Binarize the first-stage variables and solve it using Algorithm \ref{alg:binary}
\Else
\While{stopping criterion not met}
\State Solve the master problem, obtain an optimal solution $\hat{\bm{x}}$, and set $lb \gets v^*$
\For{$s \in [N]$}
\State Solve the local recourse problem to get $\Q_s(\hat{\bm{x}})$, compute $\rho^*$ using binary search or the closed form \eqref{closed_form_rho}, and update the strengthening problem \eqref{stren_mixed_integer} and solve it
\If{\eqref{stren_lp_ngc} is optimal}
\State Let $(\hat{\bm{\eta}}^+_s, \hat{\bm{\eta}}^-_s)$ be an optimal solution. $(\bm{\pi}_s^+)^{iter} \gets  \hat{\bm \eta}_s^+ -\rho^*\bm{1}$, $(\bm{\pi}_s^-)^{iter} \gets  \hat{\bm \eta}_s^- -\rho^*\bm{1}$
 \Else 
 \State $(\bm{\pi}_s^+)^{iter} \gets  -\rho^*\bm{1}$, $(\bm{\pi}_s^-)^{iter} \gets  -\rho^*\bm{1}$
\EndIf
\EndFor
\State $q_{iter} \gets \sum_{s\in [N]}\Q_s(\hatx)$, $(\bm\pi^+)^{iter} \gets \sum_{s\in [N]}(\bm{\pi}_s^+)^{iter}$, $(\bm\pi^-)^{iter} \gets \sum_{s\in [N]}(\bm{\pi}_s^-)^{iter}$
\State Add the following constraints to the master problem
\begin{equation*}
%\left\{
\begin{aligned}
&\theta \geq q_{iter} + \sum_{i\in[n]}(\pi^+_i)^{iter}\omega^+_i(\hat{x}_i) + \sum_{i\in[n]}(\pi^-_i)^{iter}\omega^-_i(\hat{x}_i) ,\omega^+_i(\hat{x}_i) - \omega^-_i(\hat{x}_i) = x_i-\hat{x}_i,\\
&0\leq \omega^+_i(\hat{x}_i)\leq (B_i-\hat{x}_i)z_i(\hat{x}_i), 0\leq \omega^-_i(\hat{x}_i)\leq \hat{x}_i(1-z_i(\hat{x}_i)), z_i(\hat{x}_i)\in\{0,1\}, \forall i\in[n].
\end{aligned}%\right.
\end{equation*}
\State $ub \gets \min\left\{ub, \bm{c}^\top \hat{\bm{x}} + q_{iter})\right\}$, $iter \gets iter + 1$
\EndWhile
\EndIf
\end{algorithmic}
% \vspace{-10pt}
\end{algorithm}

\section{Numerical Experiments}\label{sec_numerical}
In this section, we compare the performance of ReLU Lagrangian cuts and other existing cut families in the literature through numerical studies on two-stage and multistage models. All numerical experiments are conducted in Python using Gurobi version 11.0.1 on virtualized Intel Xeon Cascade Lake CPUs running at 2.9 GHz, with 61 GB RAM, under a Linux operating system.

\subsection{Two-stage models}
We test the performance of four cut combinations for solving two-stage SMIPs with purely binary first-stage decisions and mixed integer second-stage decisions. The four combinations are: (i) integer L-shaped cuts (L); (ii) Benders cuts combined with integer L-shaped cuts (B); (iii) strengthened Benders cuts \citep{zou2019stochastic} combined with integer L-shaped cuts (SB); and (iv) ReLU Lagrangian cuts (R), generated following the procedure described in \Cref{sec_binary}.

When implementing the cuts, we begin by adding Benders cuts to solve the LP relaxation to optimality at the root node. Next, we use Gurobi's lazy callback function to solve the problem through a combination of cuts within the branch-and-cut framework.

\noindent\textbf{Experiment 1.} In this experiment, we consider the stochastic server location problem (SSLP) \citep{ntaimo2005million} formulated as:
\begin{align*}
    \min_{\bm x} \left\{\sum_{j\in J} c_jx_j + \frac{1}{N}\sum_{s\in [N]}\Q_s(\bm x):\sum_{j\in J}x_j \leq v,
    x_j \in \{0,1\}, \forall j\in J\right\}
\end{align*}
where for each scenario $s\in [N]$, the local recourse function is defined as
\begin{align*}
\Q_s(\bm x):=\min_{\bm y} \ &\sum_{j\in J}q_{0j}y_{0j} - \sum_{i\in I}\sum_{j\in J}q_{ij}y_{ij}\\
    \textrm{s.t.} \ &\sum_{i\in I}d_{ij}y_{ij}-y_{0j} \leq ux_j, %\quad \forall j\in J,%\\
    \sum_{j\in J} y_{ij} = h_i^s, %\quad \forall i\in I,\\
    y_{ij}\in\{0,1\}, y_{0j} \geq 0, \quad \forall i\in I, j\in J.
\end{align*}
The model parameters are generated as follows: For the first stage, the upper bound $v$ on the total number of servers is set to $\lceil |J|/3\rceil$, the cost $c_i$ of locating a server at the location $j$ follows a discrete uniform distribution $U[40, 80]$. For the second stage, the demand $d_{ij}$ of client $i\in I$ from the server located at $j\in J$ and the revenue $q_{ij}$ from client $i$ served by the server at location $j$, both are drawn from a discrete uniform distribution $U[0, 25]$. The server capacity $u$ is 50. The overflow rate $q_{0j}$ for server $j\in J$ is 1000. Client availability $\tilde{h}_i$ is stochastic and follows a Bernoulli distribution with a success rate 0.7.

We perform numerical studies on SSLP with 20 and 30 locations and clients, considering 10, 50, 100, and 200 scenarios. The results are displayed in Table \ref{tbl:sslp}. The algorithm terminates when the gap, defined as the difference between the best upper and lower bounds divided by absolute value of the best lower bound,  is less than $0.01\%$ within a time limit of 1 hour. 

% For scenario grouping, local recourse problems are partitioned into groups of 10. When generating ReLU Lagrangian cuts for the scenario grouping, we start with L-shaped cuts based on the average of the 10 scenarios and strengthen them using the grouped local problem.

Columns 1 and 2 of Table \ref{tbl:sslp} represent the total number of locations and clients, respectively. Column 3 shows the number of scenarios, $N$. Columns 4 and 5 are the best lower and upper bounds obtained from Gurobi. Column 6 shows the gap. Columns 7 and 8 display the solution time and the number of nodes explored during the branch-and-cut procedure, respectively. 

Our results show that all four methods solve all the instances within the time limit. The ReLU Lagrangian cuts achieve the shortest solution time. Additionally, this method significantly reduces the number of nodes explored. For the other methods, the number of nodes explored is similar. This can be because methods L, B, and SB rely heavily on the integer L-shaped cuts, which are naive ReLU Lagrangian cuts. This implies that properly strengthening ReLU Lagrangian cuts can be beneficial in solving two-stage stochastic integer programs. 

% In this experiment, we observe that the scenario grouping has no obvious improvement compared to the methods L, B, and SB.

\begin{table}[htbp]
 \vspace{-10pt}
\caption{Numerical results of the SSLP instances}\label{tbl:sslp}
\centering\setlength{\tabcolsep}{3.0pt}
\footnotesize
\begin{minipage}{0.48\textwidth}
\flushright
\begin{tabular}{lllllrrrr}
\hline
$|J|$& $|I|$& $N$               & cut & lb    & ub    & \begin{tabular}[c]{@{}c@{}}Gap \\ (\%)\end{tabular} & \begin{tabular}[c]{@{}c@{}}time \\ (s)\end{tabular} & node \\
\hline
\multirow{4}{*}{20} & \multirow{4}{*}{20} & \multirow{4}{*}{10}&L&-25.4&-25.4&0.0&12.6&26\\
&&&B&-25.4&-25.4&0.0&16.1&52\\
&&& \textbf{R} & \textbf{-25.4} & \textbf{-25.4} & \textbf{0.0} & \textbf{2.8}   & \textbf{7}       \\
&&&SB&-25.4&-25.4&0.0&16.4&52\\  
 \hline
\multirow{4}{*}{30} & \multirow{4}{*}{30} & \multirow{4}{*}{10}&L&-34.4&-34.4&0.0&40.8&67\\
&&& B&-34.4& -34.4& 0.0          & 31.3           & 63               \\
&&& \textbf{R} & \textbf{-34.4} & \textbf{-34.4} & \textbf{0.0} & \textbf{13.3}   & \textbf{14}      \\
&&& SB         & -34.4          & -34.4          & 0.0          & 31.7           & 63               \\
 \hline
\multirow{4}{*}{20} & \multirow{4}{*}{20} & \multirow{4}{*}{50}  & L & -26.4          & -26.4          & 0.0          & 122.5          & 80               \\
&&& B          & -26.4          & -26.4          & 0.0          & 125.2           & 125              \\
& &  & \textbf{R} & \textbf{-26.4} & \textbf{-26.4} & \textbf{0.0} & \textbf{73.2}  & \textbf{65}      \\
& & & SB         & -26.4          & -26.4          & 0.0          & 158.8          & 125              \\
  \hline
\multirow{4}{*}{30} & \multirow{4}{*}{30} & \multirow{4}{*}{50}  & L          & -35.6          & -35.6          & 0.0          & 325.1          & 124              \\
&&& B          & -35.6          & -35.6          & 0.0          & 416.5          & 122              \\
&&& \textbf{R} & \textbf{-35.6} & \textbf{-35.6} & \textbf{0.0} & \textbf{217.6}  & \textbf{37}      \\
&&& SB         & -35.6          & -35.6          & 0.0          & 433.3          & 122              \\
 \hline
\end{tabular}
\end{minipage}
\begin{minipage}{0.48\textwidth}
\flushleft
\begin{tabular}{lllllrrrr}
\hline
$|J|$& $|I|$& $N$               & cut & lb    & ub    & \begin{tabular}[c]{@{}c@{}}Gap \\ (\%)\end{tabular} & \begin{tabular}[c]{@{}c@{}}time \\ (s)\end{tabular} & node \\
 \hline
\multirow{4}{*}{20} & \multirow{4}{*}{20} & \multirow{4}{*}{100} & L& -26.2          & -26.2          & 0.0          & 221.3           & 76               \\
&&&B          & -26.2          & -26.2          & 0.0          & 243.8           & 75               \\
&&&\textbf{R} & \textbf{-26.2} & \textbf{-26.2} & \textbf{0.0} & \textbf{143.5}  & \textbf{48}      \\
&&& SB         & -26.2          & -26.2          & 0.0          & 260.2           & 75               \\
 \hline
\multirow{4}{*}{30} & \multirow{4}{*}{30} & \multirow{4}{*}{100} & L& -36.0          & -36.0          & 0.0          & 1280.7          & 116              \\
&&&B          & -36.0          & -36.0          & 0.0          & 839.9          & 93               \\
&&&\textbf{R} & \textbf{-36.0} & \textbf{-36.0} & \textbf{0.0} & \textbf{328.6} & \textbf{27}      \\
&&&SB         & -36.0          & -36.0          & 0.0          & 805.3          & 93               \\
 \hline
 \multirow{4}{*}{20} & \multirow{4}{*}{20} & \multirow{4}{*}{200} & L & -26.0          & -26.0          & 0.0          & 308.8          & 79               \\
&&&B          & -26.0          & -26.0          & 0.0          & 499.3          & 93               \\
&&&\textbf{R} & \textbf{-26.0} & \textbf{-26.0} & \textbf{0.0} & \textbf{239.3}  & \textbf{63}      \\
&&&SB         & -26.0          & -26.0          & 0.0          & 497.7          & 93               \\
  \hline
\multirow{4}{*}{30} & \multirow{4}{*}{30} & \multirow{4}{*}{200} & L & -36.1          & -36.1          & 0.0          & 1717.6         & 156              \\
 &&&B          & -36.1          & -36.1          & 0.0          & 2086.4        & 126              \\
&&&\textbf{R} & \textbf{-36.1} & \textbf{-36.1} & \textbf{0.0} & \textbf{591.1} & \textbf{24}      \\
&&& SB         & -36.1          & -36.1          & 0.0          & 2252.1         & 126              \\
  \hline
\end{tabular}
\end{minipage}
 \vspace{-10pt}
\end{table}

\noindent \textbf{Experiment 2.} To further evaluate the ReLU Lagrangian cuts, we consider the following stochastic multiple-resource-constrained scheduling problem (SMRCSP) \citep{keller2009scheduling}:
\begin{align*}
\min_{\bm x,\bm z} \ &\sum_{j \in J} \sum_{t\in [T-p_j+1]} c_{jt} x_{j t}+\sum_{k\in[K]} \sum_{t\in [T_0]} b_k z_{t k}+\frac{1}{N}\sum_{s\in [N]}  \Q_s(\bm x) \\
\text{s.t.} \ &\sum_{t\in[T-p_j+1]}x_{jt} = 1, \quad \forall j\in J,\\
&\sum_{j\in J}\sum_{\tau \in S(j,t)}r_{jk}x_{js} - M_{tk}z_{tk} \leq R_k, \quad \forall t\in[T_0], k\in [K],\\
&x_{j t} \in\{0,1\}, \quad \forall j\in J, t\in[T-p_j+1], \quad z_{tk}\in \{0,1\}, \quad \forall t\in[T_0], k\in[K],
\end{align*}
where for each scenario $s\in [N]$, the local recourse function is
\begin{align*}
    \Q_s(\bm x):= \min_{\bm y, \bm u} \ &\sum_{j\in J_B}\sum_{t=T_0+1}^{T+T_0-p_j+1}c_{jt}y_{jt} + \sum_{k\in[K]}\sum_{t=T_0+1}^{T+T_0}b_ku_{tk}\\
    \text{s.t.} \ &\sum_{t=T_0+1}^{T+T_0-p_j+1} y_{jt} = a_j^s, \quad  \forall j\in J_B,\\
    &\sum_{j\in J_B}\sum_{\tau\in S_B(j,t)}r_{jk}y_{j\tau} -M_{tk} u_{tk} \leq R_k -\sum_{j\in J}\sum_{\tau\in S(j,t)}r_{jk}x_{j\tau}, \quad  \forall t\in [T_0+1,T+T_0],k\in[K],\\
    &y_{jt}\in\{0,1\}, \quad \forall j\in J_B,t\in[T_0+1,T+T_0-p_j+1],\\
    &u_{tk}\in \{0,1\}, \quad \forall t\in[T_0+1, T+T_0],k\in[K],
\end{align*}
and $S(j,t)=[\max\{1,t-p_j+1\}, \min\{t,T-p_j+1\}]$, $S_B(j,t)=[\max\{T_0+1,t-p_j+1\}, \min\{t,T+T_0-p_j+1\}]$. In SMRCS, we impose a fixed cost for expansions when the required resources exceed the amount available in a given time period; i.e., $z_{tk}$ and $u_{tk}$ are binary variables. Compared with SSLP with continuous penalties, the integrality gap of SMRCSP is often large  due to big M coefficients.

The parameters are generated as follows. The time period $T_0$, at which we learn the accepted job bids, is set to $\lceil 0.25T_0\rceil$, the processing time $p_j$ of job $j$ is generated from a discrete uniform distribution $U[1,T]$, and the cost $c_{jt}$ of starting job $j$ in period $t$ is set to completion time $t+p_j-1$. For each period, the cost $b_{k}$ of temporary expansion follows a discrete uniform distribution $U[10, 20]$, the amount of resource from class $k$ consumed by job $j$ (i.e., $r_{jk}$) follows a discrete uniform distribution $U[1,5]$. The resource capacity $R_k = \frac{\widebar{p}\widebar{r}_k|J| + 0.75\widebar{p}_B\widebar{r}_{Bk}|J_B|}{(T+T_0)\rho}$, where $\widebar{p} = \frac{1}{|J|}\sum_{j\in J}p_j$, $\widebar{p}_B = \frac{1}{|J_B|}\sum_{j\in J_B}p_j$, $\widebar{r}_k = \frac{1}{|J|}\sum_{j\in J}r_{jk}$, $\widebar{r}_{Bk} = \frac{1}{|J_B|}\sum_{j\in J_B}r_{jk}$, and $\rho$ follows a continuous uniform distribution $\mathcal{U}[0.5, 1.2]$. The big Ms are set to $(|J|+|J_B|)T$. The indicator $a_j$ of whether the bid on job $j$ is accepted or not is stochastic and follows a Bernoulli distribution with a success rate 0.75. In our numerical study, we set the number of known jobs $|J|$ to 5 and 7, and the number of jobs available for bidding $|J_b|$ to 5 and 10. The number of time intervals $T$ is set to 10, and we consider $N = 10$ and $N = 100$ scenarios for the second stage. The results are shown in Table \ref{tbl:smrcsp}. The stopping criterion and the meanings of the other columns the same as those presented in Table \ref{tbl:sslp}.

\begin{table}[htbp]
 \vspace{-10pt}
\caption{Numerical results of the SMRCSP instances}\label{tbl:smrcsp}
\centering\setlength{\tabcolsep}{3.0pt}
\scriptsize
\begin{minipage}{0.48\textwidth}
\flushright
\begin{tabular}{llllllrrrr}
\hline
$|J|$& $|J_b|$                & $T$& $N$              & cut & lb    & ub    & \begin{tabular}[c]{@{}c@{}}Gap \\ (\%)\end{tabular} & \begin{tabular}[c]{@{}c@{}}time \\ (s)\end{tabular} &  node \\
\hline
\multirow{4}{*}{5} & \multirow{4}{*}{5}  & \multirow{4}{*}{10} & \multirow{4}{*}{10}  & L           & 243.1          & 243.1          & 0.0          & 5.7             & 1570             \\
 &   &   &    & B           & 243.1          & 243.1          & 0.0          & 7.0             & 1794             \\
 &   &   &    & \textbf{R}  & \textbf{243.1} & \textbf{243.1} & \textbf{0.0} & \textbf{3.7}    & \textbf{428}     \\
 &   &   &    & SB          & 243.1          & 243.1          & 0.0          & 8.4             & 422              \\
       
 \hline
\multirow{4}{*}{5} & \multirow{4}{*}{5}  & \multirow{4}{*}{10} & \multirow{4}{*}{100} & L           & 256.1          & 256.1          & 0.0          & 47.7            & 1574             \\
 &   &   &    & B           & 256.1          & 256.1          & 0.0          & 56.3            & 1717             \\
 &   &   &    & \textbf{R}  & \textbf{256.1} & \textbf{256.1} & \textbf{0.0} & \textbf{30.7}   & \textbf{392}     \\
 &   &   &    & SB          & 256.1          & 256.1          & 0.0          & 70.3            & 390              \\
  \hline
 \multirow{4}{*}{5} & \multirow{4}{*}{10} & \multirow{4}{*}{10} & \multirow{4}{*}{10}  & L           & 135.6          & 135.6          & 0.0          & 36.8            & 1215             \\
 &   &   &    & B           & 135.6          & 135.6          & 0.0          & 38.7            & 1349             \\
 &   &   &    & \textbf{R}  & \textbf{135.6} & \textbf{135.6} & \textbf{0.0} & \textbf{0.4}    & \textbf{1}       \\
 &   &   &    & SB          & 135.6          & 135.6          & 0.0          & 2.2             & 23               \\
 \hline
\multirow{4}{*}{5} & \multirow{4}{*}{10} & \multirow{4}{*}{10} & \multirow{4}{*}{100} & L           & 139.9          & 139.9          & 0.0          & 321.3           & 1248             \\
 &   &   &    & B           & 139.9          & 139.9          & 0.0          & 332.1           & 1337             \\
 &   &   &    & \textbf{R}  & \textbf{139.9} & \textbf{139.9} & \textbf{0.0} & \textbf{3.7}    & \textbf{1}       \\
 &   &   &    & SB          & 139.9          & 139.9          & 0.0          & 16.6            & 11               \\
 \hline
\end{tabular}
 \end{minipage}
 %\hfill
 \begin{minipage}{0.5 \textwidth}
   \flushleft
     \begin{tabular}{{llllllrrrr}}
     \hline
$|J|$& $|J_b|$                & $T$& $N$              & cut & lb    & ub    & \begin{tabular}[c]{@{}c@{}}Gap \\ (\%)\end{tabular} & \begin{tabular}[c]{@{}c@{}}time \\ (s)\end{tabular} &  node \\
\hline
\multirow{4}{*}{7} & \multirow{4}{*}{5}  & \multirow{4}{*}{10} & \multirow{4}{*}{10}  & L           & 288.1          & 288.1          & 0.0          & 859.0           & 181689           \\
 &   &   &    & B           & 288.1          & 288.1          & 0.0          & 1064.5          & 186444           \\
 &   &   &    & \textbf{R}  & \textbf{288.1} & \textbf{288.1} & \textbf{0.0} & \textbf{542.1}  & \textbf{51487}   \\
 &   &   &    & SB          & 288.1          & 288.1          & 0.0          & 1281.0          & 45846            \\
 \hline
\multirow{4}{*}{7} & \multirow{4}{*}{5}  & \multirow{4}{*}{10} & \multirow{4}{*}{100} & L           & 130.2          & 285.0          & 118.9        & 3600.0         & 96382            \\
 &   &   &    & B           & 125.3          & 285.0          & 127.4        & 3600.0          & 78561            \\
 &   &   &    & \textbf{R}  & \textbf{285.0} & \textbf{285.0} & \textbf{0.0} & \textbf{3099.8} & \textbf{35840}   \\
 &   &   &    & SB          & 277.6          & 285.0          & 2.7          & 3600.0          & 11761            \\
 \hline
 \multirow{4}{*}{7} & \multirow{4}{*}{10} & \multirow{4}{*}{10} & \multirow{4}{*}{10}  & L           & 115.4          & 143.1          & 24.0         & 3600.0         & 30361            \\
 &   &   &    & B           & 117.2          & 143.1          & 22.1         & 3600.0          & 30260            \\
 &   &   &    & R           & 143.1          & 143.1          & 0.0          & 48.1            & 737              \\
 &   &   &    & \textbf{SB} & \textbf{143.1} & \textbf{143.1} & \textbf{0.0} & \textbf{40.5}   & \textbf{276}     \\
 \hline
\multirow{4}{*}{7} & \multirow{4}{*}{10} & \multirow{4}{*}{10} & \multirow{4}{*}{100} & L           & 112.3          & 157.2          & 40.0         & 3600.0         & 1936             \\
 &   &   &    & B           & 112.7          & 157.2          & 39.4         & 3600.0          & 2336             \\
 &   &   &    & \textbf{R}  & \textbf{157.2} & \textbf{157.2} & \textbf{0.0} & \textbf{240.6}  & \textbf{293}     \\
 &   &   &    & SB          & 157.2          & 157.2          & 0.0          & 644.1           & 224  \\
 \hline
\end{tabular}
\end{minipage}
 \vspace{-10pt}
\end{table}

In Table \ref{tbl:smrcsp}, we see that only ReLU Lagrangian cuts can solve all instances within the time limit. In general, they can close the optimality gap more quickly and require fewer branch-and-bound nodes to reach an optimal solution compared to integer L-shaped cuts and Benders cuts. Although strengthened Benders cuts outperform standard Benders cuts and require fewer nodes to explore than ReLU Lagrangian cuts in most instances, they still underperform the ReLU Lagrangian cuts. This may be due to the need to solve an MIP to obtain each strengthened Benders cut.

\noindent\textbf{Experiment 3.} Finally, we consider the following dynamic capacity acquisition and allocation problem (DCAP) adapted from \cite{ahmed2003dynamic} with mixed integer first-stage decisions: 
\begin{align*}
\min_{\bm x, \bm u}\left\{\sum_{t\in T} \sum_{i\in I} (\alpha_{it}x_{it} + \beta_{it}u_{it})+\frac{1}{N} \sum_{s\in [N]}  \Q_s(\bm x): x_{it}\leq b_{it}u_{it}, x_{it} \in [0,b_{it}], u_{it}\in\{0,1\}, \forall i\in I, t\in T\right\},
\end{align*}
where the local recourse functions
\begin{align*}
\Q_s(\bm x) := \min_{\bm y} &\sum_{t\in T} \sum_{i\in I}\sum_{j\in J} c_{ijt}^s y_{ijt} + \sum_{t\in T}\sum_{i\in I}p_{it}y_{it}^0\\
\text{s.t.} \ &\sum_{j\in J} d_{jt}^s y_{ijt}  - y^0_{it} \leq \sum_{\tau \in [t]} x_{i\tau}, \quad \forall i\in I, t\in T,\\
&\sum_{i\in I}y_{ijt} = 1, %\quad \forall j\in J, t\in T,\\
%&
y_{it}^0 \geq 0, y_{ijt} \in \{0,1\}, \quad \forall i\in I, j\in J, t\in T.
\end{align*}

In this experiment, the parameters are generated as follows. In the first stage, the capacity expansion cost of acquiring resource $i$ in period $t$ consists of two parts, where the variable cost $\alpha_{it}$ and the fixed cost $\beta_{it}$ are drawn from discrete uniform distributions $U[20, 40]$ and  $U[50, 70]$, respectively. The maximum number of units of resource $i$ that can be required at time $t$, denoted as $b_{it}$, is set to 50. In the second stage, the cost $\tilde{c}_{ijt}$ of processing task $j$ using resource $i$ in period $t$, as well as the processing requirement $\tilde{d}_{jk}$ for task $j$ in period $t$, are stochastic. They follow discrete uniform distributions $U[40, 80]$ and $U[1, 10]$, respectively. To ensure relatively complete recourse, we add a penalty term for overflow, where the penalty rate $p_{it}$ is set to 1000.

We compare the performance of Augmented Lagrangian cuts (AL) given by \eqref{tight_alc} with $\bm\pi=\bm{0}$ and ReLU Lagrangian cuts (R) obtained by improving these augmented Lagrangian cuts through the cut-strengthening procedure described in \Cref{sec_general_mixed}. When solving the instances, we first add Benders cuts to solve the LP relaxation to optimality. Then, we maintain the integrality constraints in the first stage while taking the LP relaxation of the second stage and solve this relaxed problem to optimality using Benders cuts. At each iteration, we add a strengthened Benders cut with a nonlinear cut (AL or R) until the gap falls below $0.1\%$\footnotemark\footnotetext{Since this is an iterative solution procedure, we choose a slightly worse stopping criterion for the gap. If we choose the gap to be $0.01\%$, in most cases, the solution time will exceed an hour, but ReLU Lagrangian cuts still perform better than augmented Lagrangian cuts.} or the time limit of one hour is reached. 

The results are shown in Table \ref{tbl:dcap}. Columns 1, 2, and 3 represent the total number of resource types, tasks, and periods, respectively. Column 4 shows the number of scenarios. The iteration column lists the number of iterations used to add nonlinear cuts, excluding those for solving the relaxed problem with Benders cuts. As shown by the results, our cut-strengthening procedure leads to significant improvements over augmented Lagrangian cuts. For the instances in the first two rows, it significantly reduces the number of cuts required to solve the problem and shortens the solution time. For the larger instances, although both methods reach the time limit, ReLU Lagrangian cuts still achieve much smaller gaps.

\begin{table}[htbp]
 \vspace{-10pt}
\caption{Numerical results of the DCAP instances}\label{tbl:dcap}
\centering\setlength{\tabcolsep}{3.0pt}
\scriptsize
\begin{minipage}{0.48\textwidth}
\flushright
\begin{tabular}{lllllrrrrrr}
\hline
$|I|$             & $|J|$& $|T|$&  $N$                & cut & lb     & ub     & \begin{tabular}[c]{@{}c@{}}Gap \\ (\%)\end{tabular} & \begin{tabular}[c]{@{}c@{}}iter- \\ ation\end{tabular} & \begin{tabular}[c]{@{}c@{}}time \\ (s)\end{tabular} \\
\hline
\multirow{2}{*}{2} & \multirow{2}{*}{2} & \multirow{2}{*}{4} & \multirow{2}{*}{10}  & AL         & 1014.2          & 1015.2          & $<$0.1          & 998          & 3478.8          \\
 &&&  & \textbf{R} & \textbf{1014.9} & \textbf{1015.4} & \textbf{$<$0.1} & \textbf{41}  & \textbf{6.7}    \\
 \hline
\multirow{2}{*}{2} & \multirow{2}{*}{3} & \multirow{2}{*}{4} & \multirow{2}{*}{10}  & AL         & 2134.0          & 2280.7          & 6.9          & 521          & 3600.0          \\
 &&&  & \textbf{R} & \textbf{2257.5} & \textbf{2259.3} & \textbf{$<$0.1} & \textbf{89}  & \textbf{35.1}   \\
 \hline
\multirow{2}{*}{3} & \multirow{2}{*}{4} & \multirow{2}{*}{5} & \multirow{2}{*}{10}  & AL         & 2189.1          & 2366.1          & 8.1         & 324          & 3600.0          \\
 &&&  & \textbf{R} & \textbf{2218.6} & \textbf{2247.9} & \textbf{1.3} & \textbf{361} & \textbf{3600.0} \\
 \hline
\multirow{2}{*}{4} & \multirow{2}{*}{5} & \multirow{2}{*}{6} & \multirow{2}{*}{10}  & AL         & 2994.7          & 3192.0          & 6.6          & 509          & 3600.0          \\
 &&&  & \textbf{R} & \textbf{3009.7} & \textbf{3192.0} & \textbf{6.1} & \textbf{393} & \textbf{3600.0} \\
\hline
\end{tabular}
\end{minipage}
\begin{minipage}{0.48\textwidth}
\flushleft
\begin{tabular}{lllllrrrrrr}
\hline
$|I|$             & $|J|$& $|T|$&  $N$                & cut & lb     & ub     & \begin{tabular}[c]{@{}c@{}}Gap \\ (\%)\end{tabular} & \begin{tabular}[c]{@{}c@{}}iter- \\ ation\end{tabular} & \begin{tabular}[c]{@{}c@{}}time \\ (s)\end{tabular}\\
\hline
\multirow{2}{*}{2} & \multirow{2}{*}{2} & \multirow{2}{*}{4} & \multirow{2}{*}{100} & AL         & 1047.0          & 1048.0          & $<$0.1          & 629          & 1000.3          \\
 &&&  & \textbf{R} & \textbf{1047.4} & \textbf{1048.4} & \textbf{$<$0.1} & \textbf{22}  & \textbf{24.6}   \\
 \hline
\multirow{2}{*}{2} & \multirow{2}{*}{3} & \multirow{2}{*}{4} & \multirow{2}{*}{100} & AL         & 1976.1          & 2092.5          & 5.9          & 438          & 3600.0          \\
 &&&  & \textbf{R} & \textbf{2067.3} & \textbf{2069.3} & \textbf{$<$0.1} & \textbf{381} & \textbf{3289.8} \\
 \hline
\multirow{2}{*}{3} & \multirow{2}{*}{4} & \multirow{2}{*}{5} & \multirow{2}{*}{100} & AL         & 2307.6          & 2469.6          & 7.0          & 260          & 3600.0          \\
 &&&  & \textbf{R} & \textbf{2340.6} & \textbf{2388.6} & \textbf{2.1} & \textbf{291} & \textbf{3600.0} \\
 \hline
\multirow{2}{*}{4} & \multirow{2}{*}{5} & \multirow{2}{*}{6} & \multirow{2}{*}{100} & AL         & 3005.0          & 3170.8          & 5.5          & 307          & 3600.0          \\
 &&&  & \textbf{R} & \textbf{3020.3} & \textbf{3113.6} & \textbf{3.1} & \textbf{255} & \textbf{3600.0}\\
 \hline
\end{tabular}
\end{minipage}
 \vspace{-10pt}
\end{table}

Based on the numerical experiments above, we conclude the advantages of ReLU Lagrangian cuts in solving two-stage SMIPs as follows: (i) For binary first-stage decisions, they can replace the combination of Benders and integer L-shaped cuts. By improving the integer L-shaped cuts through LP-based strengthening problems, ReLU Lagrangian cuts inherit the strength of Benders cuts and the tightness of integer L-shaped cuts, making the master problem easier to solve compared to adding both Benders and integer L-shaped cuts simultaneously. (ii) Similar to strengthened Benders cuts, they provide a stronger outer approximation than the epigraph of the continuous relaxation of local recourse functions. The efficiency of the ReLU Lagrangian cuts is possibly due to incorporating valid inequalities of local convex hulls, including no-good and objective cuts, to the LP-based strengthening problems. (iii) For mixed integer first-stage decisions, they improve augmented Lagrangian cuts at a low computational cost and achieve better performance. And (iv) by implementing an appropriate strategy to prevent unboundedness in the strengthening problems, the resulting cuts accelerate the master problem and mitigate potential numerical issues arising from large coefficients in the integer L-shaped cuts or augmented Lagrangian cuts.

\subsection{A mutistage model}\label{subsec:numerical_mulstage}
The proposed ReLU Lagrangian cuts can be integrated into different algorithmic frameworks \citep{zou2019stochastic, ahmed2022stochastic, van2019level, yu2022multistage} to solve multistage SMIPs, which have been applied to many areas such as unit commitment \citep{zou2018unit}, facility location \citep{yu2021value},  disaster relief logistics planning \citep{castro2022markov}, and so on. In this subsection, we evaluate the performance of ReLU Lagrangian cuts within the stochastic dual dynamic integer programming (SDDiP) framework on a multistage airline revenue management problem \citep{zou2019stochastic}. 
% As ReLU Lagrangian cuts can be integrated into the SDDiP framework, we evaluate their performance on a multistage model from \cite{zou2019stochastic}-- the airline revenue management problem. 

In this subsection, we compare different combinations of non-tight and tight cuts. Non-tight cuts include Benders cuts (B), strengthened Benders cuts (SB), and improved Benders cuts (IB). An improved Benders cut is generated by applying our cut-strengthening procedure to a Benders cut. Tight cuts are selected from either integer L-shaped cuts (L) or ReLU Lagrangian cuts (R).

\noindent\textbf{Experiment 4.} Consider the following multistage airline revenue management problem (MARM):
\begin{align*}
\max \ &\sum_{t\in [T]}\left[(\bm f^b)^{\top} \bm b_t-(\bm f^c)^{\top} \bm c_t)\right] \\
\text { s.t. } \ & \bm B_t=\bm B_{t-1}+\bm b_t, \bm C_t=\bm C_{t-1}+\bm c_t, %\quad \forall t\in [T], \\& 
\bm C_t=\lfloor\bm \Gamma \bm B_t+0.5\rfloor, \bm b_t \leq \bm d_t,\bm A(\bm B_t-\bm C_t) \leq \bm R,\quad \forall t\in [T],\\
&  %\quad \forall t\in [T],\\&
 % \quad \forall t\in [T], \\& 
\bm B_0=\widebar{\bm B}_0, \bm C_0=\widebar{\bm C}_0, %\\& 
\bm B_t\in \Ze_{+}^m, \bm C_t \in \Ze_{+}^m, \bm b_t\in \Ze_{+}^m, \bm c_t \in \Ze_{+}^m, \quad \forall t\in[T].
\end{align*}
The parameters are primarily based on the descriptions in \cite{zou2019stochastic} and \cite{moller2008airline} with slightly different origin-destination (OD) pairs and fair classes. Specifically, we suppose that there are 4 OD pairs, each having 2 classes, and each class offers 2 fare prices. Therefore, the total number of classes $n$ is 8, and the total number of ticket types $m$ is 12. The vector $\bm f^b$ representing the ticket prices is $(500, 340, 200, 100, 500, 340, 200, 100, 800, 540, 320, 160, 800, 540, 320, 160)^\top$, and the vector $\bm f^c$ representing the refund amount for cancellations is set as $0.8 \bm f^b$. The matrix $\bm\Gamma$ is a diagonal matrix, where the diagonal entries represent the cancellation rates. The cancellation rate for the first four types of tickets is 0.15, for the next eight types, it is 0.1, and for the last four types, it is 0.05. The matrix $\bm{A}$ is an $n\times m$ indicator matrix, where $A_{ij} = 1$ if ticket type $i$ belongs to class $j$, and 0 otherwise. The seating capacity for the first class is 24, and for the economy class, it is 120. Thus, the vector $\bm{R}$ representing seat capacities is set to $(24, 120, \ldots, 24 ,120)^\top$. Both the initial bookings vector  $\widebar{\bm{B}}_0$  and the initial cancellations vector  $\widebar{\bm{C}}_0$ are zero vectors.

The demand vector $\bm d_t$ at stage $t$ is stochastic. For each ticket type $i$, the total number of bookings follows a gamma distribution $G(p_i, g_i)$. These bookings are distributed over 60 days according to a beta distribution $B(\alpha_i, \beta_i)$, and the bookings for the 60 days are then assigned to the stages $T$. We repeat the sampling procedure 50 times, and for each stage $t\in [T]$, the demand vector $\bm d_t$ is drawn from these samples independently.  Note that we suppose that $\bm p = (3.0, 3.0, 70.0/3, 52.0, 2.0, 2.0, 35.0/3, 26.0, 2.0, 2.0, 17.5, 39.0, 3.0, 3.0, 35.0, 78.0)^\top$, $\bm g = (1.5, 1.5, 1.2, 1.0,  \ldots, 1.5, 1.5, 1.2, 1.0)^\top$, $\bm \alpha=(12.0, 8.0, 6.0, 2.0, \ldots,12.0, 8.0, 6.0, 2.0)^\top $, and $\bm \beta=(1.5, 2.0, 2.0, 4.0, \ldots, 1.5, 2.0, 2.0, 4.0)$.

\begin{table}[htbp]
 \vspace{-10pt}
\caption{Numerical results of the MARM instances.}\label{tbl:mulstage}
\centering\setlength{\tabcolsep}{3.0pt}
\footnotesize
\begin{minipage}{0.48\textwidth}
\flushright
\begin{tabular}{lllrrrr}
\hline
stage             & \begin{tabular}[c]{@{}c@{}}scen- \\ ario\end{tabular}             & cut  & (stat) lb & (best) ub & \begin{tabular}[c]{@{}c@{}}gap \\ (\%)\end{tabular} & \begin{tabular}[c]{@{}c@{}}iter- \\ ation\end{tabular}\\
\hline
\multirow{6}{*}{4}  & \multirow{6}{*}{3} & B+L           & 64925.9          & 66524.6          & 2.46          & 249    \\
&& B+R           & 64887.9          & 66407.6          & 2.34          & 298    \\
&& SB+L          & 65318.6          & 66530.5          & 1.86          & 174    \\
&& \textbf{SB+R} & \textbf{65319.8} & \textbf{66494.1} & \textbf{1.80} & \textbf{184}    \\
&& IB+L         & 65041.2          & 66591.0          & 2.38          & 197    \\
&& IB+ R         & 64942.7          & 66577.4          & 2.52          & 215    \\
\hline
\multirow{6}{*}{6}  & \multirow{6}{*}{3} & B+L           & 60744.3          & 62603.0          & 3.06          & 188    \\
&& B+R          & 60745.6          & 62564.9          & 3.00          & 225    \\
&& SB+L          & 60914.9          & 62670.7          & 2.88          & 127    \\
&& \textbf{SB+R} & \textbf{61083.7} & \textbf{62648.7} & \textbf{2.56} & \textbf{140}    \\
&& IB+L          & 60685.0          & 62629.7          & 3.20          & 161    \\
&& IB+R          & 60731.7          & 62652.4          & 3.16          & 173    \\
\hline
\multirow{6}{*}{8}  & \multirow{6}{*}{3} & B+L           & 63582.7          & 65131.5          & 2.44          & 155    \\
&& B+R           & 63593.9          & 65120.8          & 2.40          & 176    \\
&& SB+L          & 63949.0          & 65165.7          & 1.90          & 96     \\
&& \textbf{SB+R} & \textbf{63943.6} & \textbf{65156.2} & \textbf{1.90} & \textbf{108}    \\
&& IB+L          & 63667.2          & 65154.9          & 2.34          & 132    \\
&& IB+R          & 63859.5          & 65140.2          & 2.01          & 144    \\
\hline
\multirow{6}{*}{10} & \multirow{6}{*}{3} & B+L           & 63615.3          & 65037.2          & 2.24          & 136    \\
&& \textbf{B+R}           & \textbf{63607.9}          & \textbf{65027.4}          & \textbf{2.23}          & \textbf{149}    \\
&& SB+L & 63652.5 & 65070.0 & 2.23 & 84     \\
&& SB+R          & 63643.1          & 65065.9          & 2.24          & 92     \\
&& IB+L          & 63618.5          & 65052.0          & 2.25          & 117    \\
&& IB+R          & 63617.5          & 65043.1          & 2.24          & 130    \\
\hline
\end{tabular}
\end{minipage}
 \begin{minipage}{0.48\textwidth}
\flushleft
\begin{tabular}{lllrrrr}
\hline
stage             & \begin{tabular}[c]{@{}c@{}}scen- \\ ario\end{tabular}             & cut  & (stat) lb & (best) ub & \begin{tabular}[c]{@{}c@{}}gap \\ (\%)\end{tabular} & \begin{tabular}[c]{@{}c@{}}iter- \\ ation\end{tabular}\\
\hline
\multirow{6}{*}{4}  & \multirow{6}{*}{5} & B+L           & 65276.6          & 66737.9          & 2.24          & 211    \\
&& B+R           & 65172.6          & 66705.1          & 2.35          & 253    \\
&& SB+L          & 65278.0          & 66809.0          & 2.35          & 135    \\
&& \textbf{SB+R} & \textbf{65516.9} & \textbf{66798.3} & \textbf{1.96} & \textbf{144}    \\
&& IB+L          & 65125.8          & 66796.8          & 2.57          & 170    \\
&& IB+R          & 65193.0          & 66782.5          & 2.44          & 183    \\
\hline
\multirow{6}{*}{6}  & \multirow{6}{*}{5} & B+L           & 64148.0          & 65722.4          & 2.45          & 144    \\
&& B+R           & 64162.7          & 65721.6          & 2.43          & 157    \\
&& SB+L          & 64216.3          & 65770.9          & 2.42          & 88     \\
&& \textbf{SB+R} & \textbf{64224.7} & \textbf{65763.3} & \textbf{2.40} & \textbf{97}     \\
&& IB+L          & 64171.8          & 65753.3          & 2.46          & 120    \\
&& IB+R          & 64208.4          & 65746.0          & 2.39          & 130    \\
\hline
\multirow{6}{*}{8}  & \multirow{6}{*}{5} & B+L           & 66238.6          & 68181.7          & 2.93          & 125    \\
&& B+R           & 66224.2          & 68173.5          & 2.94          & 140    \\
&& SB+L          & 66165.9          & 68090.1          & 2.91          & 73     \\
&& \textbf{SB+R} & \textbf{66251.7} & \textbf{68086.7} & \textbf{2.77} & \textbf{82}     \\
&& IB+L          & 66182.9          & 68191.5          & 3.03          & 107    \\
&& IB+R          & 66199.0          & 68166.6          & 2.97          & 120    \\
\hline
\multirow{6}{*}{10} & \multirow{6}{*}{5} & \textbf{B+L}  & \textbf{66954.7} & \textbf{68544.1} & \textbf{2.37} & \textbf{111}    \\
&& B+R           & 66928.5          & 68518.2          & 2.38          & 124    \\
&& SB+L          & 66860.0          & 68576.4          & 2.57          & 64     \\
&& SB+R          & 66788.5          & 68555.4          & 2.65          & 69     \\
&& IB+L          & 66919.6          & 68531.1          & 2.41          & 94     \\
&& IB+R          & 66729.2          & 68527.9          & 2.70          & 101  \\
\hline
\end{tabular}
\end{minipage}
 \vspace{-10pt}
\end{table}

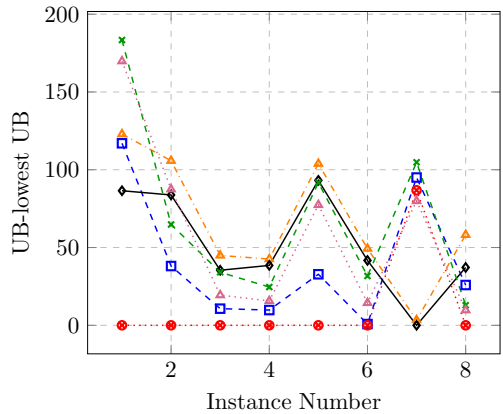
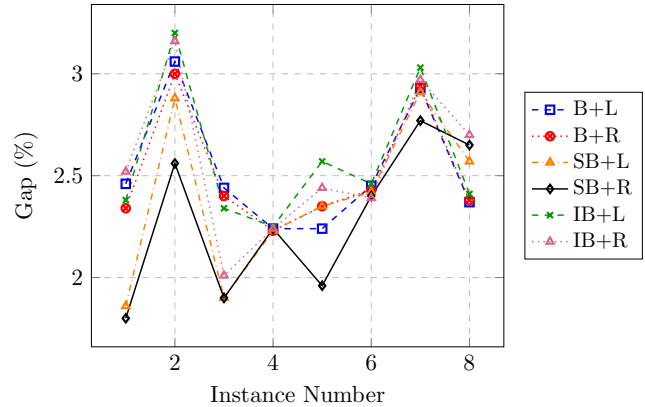
\begin{figure}[htbp]
\centering
\subfigure[The comparison of upper bounds]{\label{fig:marm_ub}
\begin{tikzpicture}[scale=0.8]
\begin{axis}[
    xlabel={Instance Number},
    ylabel={UB-lowest UB},
    legend style={at={(1.05,0.5)}, anchor=west, font=\small, legend cell align={left}},
    grid=both,
    ymajorgrids=true,
    xmajorgrids=true,
    grid style=dashed
]

% B+L data
\addplot[color=blue, mark=square, mark options={solid, line width=1pt}, dashed, line width=0.7pt] 
    coordinates {(1, 117) (2, 38.1) (3, 10.7) (4, 9.8) (5, 32.8) (6, 0.8) (7, 95) (8, 25.9)};
% \addlegendentry{B+L}

% B+R data
\addplot[color=red, mark=otimes, mark options={solid, line width=1pt}, dotted, line width=0.7pt] 
    coordinates {(1, 0) (2, 0) (3, 0) (4, 0) (5, 0) (6, 0) (7, 86.8) (8, 0)};
% \addlegendentry{B+R}

% SB+L data
\addplot[color=orange, mark=triangle, mark options={solid, line width=1pt}, dash dot, line width=0.7pt] 
    coordinates {(1, 122.9) (2, 105.8) (3, 44.9) (4, 42.6) (5, 103.9) (6, 49.3) (7, 3.4) (8, 58.2)};
% \addlegendentry{SB+L}

% SB+R data
\addplot[color=black, mark=diamond, mark options={solid, line width=1pt}, solid, line width=0.7pt] 
    coordinates {(1, 86.5) (2, 83.8) (3, 35.4) (4, 38.5) (5, 93.2) (6, 41.7) (7, 0) (8, 37.2)};
% \addlegendentry{SB+R}

% IB+L data
\addplot[color=green!60!black, mark=x, mark options={solid, line width=1pt}, dashed, line width=0.7pt] 
    coordinates {(1, 183.4) (2, 64.8) (3, 34.1) (4, 24.6) (5, 91.7) (6, 31.7) (7, 104.8) (8, 12.9)};
% \addlegendentry{IB+L}

% IB+R data
\addplot[color=purple!60!white, mark=triangle, mark options={solid, line width=1pt}, dotted, line width=0.7pt] 
    coordinates {(1, 169.8) (2, 87.5) (3, 19.4) (4, 15.7) (5, 77.4) (6, 14.4) (7, 79.9) (8, 9.7)};
% \addlegendentry{IB+R}

\end{axis}
\end{tikzpicture}}
\hfill
\subfigure[The comparison of gaps]{\label{fig:marm_gap}
\begin{tikzpicture}[scale=0.8]
\begin{axis}[
    xlabel={Instance Number},
    ylabel={Gap (\%)},
    legend style={at={(1.05,0.5)}, anchor=west, font=\small, legend cell align={left}},
    grid=both,
    ymajorgrids=true,
    xmajorgrids=true,
    grid style=dashed
]

% B+L data
\addplot[color=blue, mark=square, mark options={solid, line width=1pt}, dashed, line width=0.7pt] 
    coordinates {(1, 2.46) (2, 3.06) (3, 2.44) (4, 2.24) (5, 2.24) (6, 2.45) (7, 2.93) (8, 2.37)};
\addlegendentry{B+L}

% B+R data
\addplot[color=red, mark=otimes, mark options={solid, line width=1pt}, dotted, line width=0.7pt] 
    coordinates {(1, 2.34) (2, 3.00) (3, 2.40) (4, 2.23) (5, 2.35) (6, 2.43) (7, 2.94) (8, 2.38)};
\addlegendentry{B+R}

% SB+L data
\addplot[color=orange, mark=triangle, mark options={solid, line width=1pt}, dash dot, line width=0.7pt] 
    coordinates {(1, 1.86) (2, 2.88) (3, 1.90) (4, 2.23) (5, 2.35) (6, 2.42) (7, 2.91) (8, 2.57)};
\addlegendentry{SB+L}

% SB+R data
\addplot[color=black, mark=diamond, mark options={solid, line width=1pt}, solid, line width=0.7pt] 
    coordinates {(1, 1.80) (2, 2.56) (3, 1.90) (4, 2.24) (5, 1.96) (6, 2.40) (7, 2.77) (8, 2.65)};
\addlegendentry{SB+R}

% IB+L data
\addplot[color=green!60!black, mark=x, mark options={solid, line width=1pt}, dashed, line width=0.7pt] 
    coordinates {(1, 2.38) (2, 3.20) (3, 2.34) (4, 2.25) (5, 2.57) (6, 2.46) (7, 3.03) (8, 2.41)};
\addlegendentry{IB+L}

% IB+R data
\addplot[color=purple!60!white, mark=triangle, mark options={solid, line width=1pt}, dotted, line width=0.7pt] 
    coordinates {(1, 2.52) (2, 3.16) (3, 2.01) (4, 2.24) (5, 2.44) (6, 2.39) (7, 2.97) (8, 2.70)};
\addlegendentry{IB+R}

\end{axis}
\end{tikzpicture}}
\caption{\centering Comparisons of cut combinations for the MARM instances}\label{fig:arm}
\end{figure}

When implementing the SDDiP algorithm, during the forward step, we sample five paths and use the results of all five paths to generate cuts in the backward step. The statistical lower bounds are obtained by computing a 95\% confidence value using 1,500 sampled paths after the algorithm terminates. The results are shown in Table \ref{tbl:mulstage} and illustrated in Figure \ref{fig:arm}. Any cut combination that involves the ReLU Lagrangian cuts generally outperforms its counterpart with integer L-shaped cuts. As shown in the iteration column of Table \ref{tbl:mulstage}, the efficient cut-generating procedure for ReLU Lagrangian cuts enables more iterations to be completed within the same time limit. This may be because by using Strategy 1 to avoid unboundedness, many cut coefficients are improved to smaller absolute values or even zero, making the master problem easier to solve compared to the original L-shaped cuts. It is seen that B+R often achieves the lowest upper bounds in \Cref{fig:marm_ub} and B+R achieves the smallest gap for the majority of instances in \Cref{fig:marm_gap}. Although the improved Benders cuts benefit from the LP-based cut-strengthening procedure, their performance remains less efficient than that of the strengthened Benders cuts. Moreover, the advantage of ReLU Lagrangian cuts becomes less significant as the number of stages increases.

This numerical example shows that for purely binary state variables, ReLU Lagrangian cuts, which are equivalent to Lagrangian cuts, can be alternatives to integer L-shaped cuts, which is quite different from the existing literature \citep{zou2019stochastic}. While previous studies have focused on improving Lagrangian cuts to achieve certain properties by solving Lagrangian duals with integrality constraints, the resulting Lagrangian cuts are still not as effective as strengthened Benders cuts. As a result, it may be more efficient to use the combination of strengthened Benders and integer L-shaped cuts. This may be due to two reasons: (i) Obtaining a Lagrangian cut requires solving multiple MIPs, whereas strengthening a Benders cut only solves one; and (ii) Lagrangian cuts may compromise their strength to ensure the tightness. As this effect accumulates across stages, it results in slower recovery of the epigraphs of cost-to-go functions compared to strengthened Benders cuts. Nevertheless, the combination of strengthened Benders cuts and ReLU Lagrangian cuts often outperforms that of strengthened Benders and integer L-shaped cuts, especially in providing better upper bounds.

\section{Conclusion}\label{sec_con}

This paper introduced a new family of nonlinear cuts, termed ``ReLU Lagrangian cuts," for solving stochastic integer programs. These cuts improved traditional methods by addressing nonanticipativity constraints through ReLU functions, enabling both local and expected recourse epigraphs to be tightly and efficiently recovered. The tightness of these cuts was established through strong duality. We also proved that ReLU Lagrangian cuts are a generalization of existing cut families, including integer L-shaped cuts, ordinary Lagrangian cuts, reversed norm cuts, and augmented Lagrangian cuts. Therefore, existing cuts can serve as a foundation for initiating the ReLU Lagrangian cuts. We also proposed efficient cut generation schemes that enhance cut coefficients while eliminating the need to solve multiple mixed integer programs at each iteration. Our numerical studies demonstrated the superior performance of ReLU Lagrangian cuts, particularly in reducing the number of iterations required for the cutting-plane method to converge, compared to existing approaches.

\bibliographystyle{informs2014}
\bibliography{ref.bib}

\newpage

\titleformat{\section}{\large\bfseries}{\appendixname~\thesection .}{0.5em}{}
\begin{appendices}
\section{Proofs}\label{proofs}
\subsection{Proof of \Cref{coro:tight_lag}} \label{pf:coro_tight_lag}
\tight*
\begin{proof}
It is equivalent to show that any cut coefficient $\bm\pi\in\dxhat$ satisfies $\mathcal{L}_s(\bm\pi;\hatx) = \Q_s(\hatx)$.
Given $\bm\pi\in\dxhat$,
   from the primal characterization we have $\mathcal{L}_s(\bm\pi;\hatx)=\min_{\bm{x}} \{\theta : (\bm{x}, \theta) \in \conv(\epi_{\barx}(\Q_s)), \bm{x} = \hat{\bm{x}} \}\geq \min_x\{\theta:\theta\geq \bm\alpha^\top(\hatx-\bm x)+\Q_s(\hatx), \bm x=\hatx  \}=\Q_s(\hatx)$. Combined with \eqref{primal0} completes the proof.
   \qed
\end{proof}

\subsection{Proof of \Cref{thm:gen_lag_dual}}
\label{pf:thm_dual}
\dual*
\begin{proof}
By weak duality, we have $\underline{\Q}_s(\hatx) \leq \Q_s(\hatx)$.  
To show that $\underline{\Q}_s(\hatx) \geq \Q_s(\hatx)$, we note
\begin{align*}
\underline{\Q}_s(\hatx) = &\sup_{\bm\pi^+,\bm\pi^-\in\Re^n}{\mathcal{L}_s(\bm\pi^+,\bm\pi^-;\hatx)} 
\geq \sup_{\rho>0}{\mathcal{L}_s(-\rho\bm{1},-\rho\bm{1};\hatx)},
\end{align*}
where 
%\begin{equation}
\begin{align}
\mathcal{L}_s(-\rho\bm{1},-\rho\bm{1};\hatx)= &\inf_{\bm x}\{\Q_s(\bm x)+\rho||\bm x-\hatx||_1:\bm x\in \barx\}
= \inf_{\bm x, \bm y}\{(\bm q^s)^\top \bm y + \rho||\bm x-\hatx||_1: (\bm x,\bm y)\in \widebar{Y}_s\}\notag\\
=&\inf_{\bm x, \bm y, \omega}\{(\bm q^s)^\top \bm y + \rho \omega: (\bm x,\bm y)\in \widebar{Y}_s, ||\bm x-\hatx||_1\leq \omega\},\label{rev_norm_fun}
\end{align}
and $\widebar{Y}_s:=\{(\bm x,\bm y):\Ze^{n_1}\times\Re^{n_2}\times\Ze^{m_1}\times\Re^{m_2}: \bm A\bm x\geq \bm b, \bm T^s \bm x+ \bm W^s\bm y \geq \bm h^s\}$.  
Since formulation \eqref{rev_norm_fun} is an MIP with rational data and 
$\inf_{\bm x, \bm y, \omega}\{(\bm q^s)^\top \bm y + \rho \omega: (\bm x,\bm y)\in \widebar{Y}_s, ||\bm x-\hatx||_1\leq \omega\} \geq \inf_{\bm x, \bm y}\{(\bm q^s)^\top \bm y: (\bm x,\bm y)\in \widebar{Y}_s\}\geq L$ by Assumption \ref{assumption:lb}, the infimum of formulation \eqref{rev_norm_fun} is attained at some $(\bm x_\rho, \bm y_\rho, \omega_\rho)$. From the weak duality, we have
\begin{equation}\label{weak_dual}
\mathcal{L}_s(-\rho\bm{1},-\rho\bm{1};\hatx)=
(\bm q^s)^\top \bm y_\rho+\rho\omega_\rho=
\Q_s(\bm x_\rho)+\rho\omega_\rho
\leq \Q_s(\hatx),
\end{equation}
for all $\rho>0$ and any optimal solution $(\bm x_\rho, \bm y_\rho, \omega_\rho)$ of problem \eqref{rev_norm_fun}. 
This implies that $\omega_\rho\leq \frac{\Q_s(\hatx)-\Q_s(\bm x_\rho)}{\rho}\leq \frac{\Q_s(\hatx)-L}{\rho}$. Letting $\rho\rightarrow\infty$, we have $\omega_\rho\rightarrow 0$. Hence, we must have $\lim_{\rho\rightarrow \infty} \bm x_\rho = \hatx$, since $||\bm x_\rho -\hatx||_1 \leq \omega_\rho$ for any $\rho>0$. Notice that $\Q_s(\bm x)$ is the value function of a rational MIP and is, therefore, lower semicontinuous with respect to $\bm x$ \citep{meyer1975integer}. Thus,  $\liminf_{\rho\rightarrow\infty}\Q_s(\bm x_\rho) \geq \Q_s(\hatx)$. For any $\rho>0$, we have $\mathcal{L}_s(-\rho\bm{1},-\rho\bm{1};\hatx)=\Q_s(\bm x_\rho)+\rho\omega_\rho \geq \Q_s(\bm x_\rho)$. Taking the $\liminf$ on both sides, we obtain $\liminf_{\rho\rightarrow\infty}\mathcal{L}_s(-\rho\bm{1},-\rho\bm{1};\hatx)\geq \Q_s(\hatx)$. Meanwhile, taking the $\limsup$ on both sides of \eqref{weak_dual} and using the fact that $\lim_{\rho\rightarrow\infty}\omega_\rho = \hatx$, we have $\limsup_{\rho\rightarrow\infty}\mathcal{L}_s(-\rho\bm{1},-\rho\bm{1};\hatx)\leq \Q_s(\hatx)$. Thus, $\lim_{\rho\rightarrow\infty}\mathcal{L}_s(-\rho\bm{1},-\rho\bm{1};\hatx)=\Q_s(\hatx)$. Together with the weak duality, this shows that $\underline{\Q}_s(\hatx)=\Q_s(\hatx)$.

Next, we prove the supremum can be attained at some finite $\rho^*$. 
Define a set
\begin{equation}\label{set_F}
\widebar{F} := \{(\bm x,\bm y, \omega):(\bm x,\bm y)\in \widebar{Y}_s, ||\bm x-\hatx||_1\leq \omega\}.
\end{equation}
When $\rho > 0$, we have
\begin{align*}
\mathcal{L}_s(-\rho\bm{1},-\rho\bm{1};\hatx) =& \inf_{\bm x, \bm y, \omega}\left\{(\bm q^s)^\top \bm y +\rho\omega: (\bm x,\bm y, \omega)\in \widebar{F}\right\}
= \inf_{\bm x, \bm y, \omega}\left\{(\bm q^s)^\top \bm y +\rho\omega: (\bm x,\bm y, \omega)\in \conv{(\widebar{F})}\right\}\\
=&\min_{k\in K}\{(\bm q^s)^\top \bm y^k +\rho\omega_k\}
\geq  \min_{k\in K}\{\Q_s(\bm x^k) +\rho\omega_k\},
\end{align*}
where $\{(\bm x^k, \bm y^k, \omega_k)\}_{k\in K} \subseteq \widebar{F}$ are all extreme points of $\conv{(\widebar{F})}$. The third equality holds since $\rho>0$ ensures that $\mathcal{L}_s(-\rho\bm{1},-\rho\bm{1};\hatx)\geq L$, and the infimum is attained at some extreme point of $\conv{(\widebar{F})}$. The inequality follows because, for any feasible solution $(\bm x, \bm y)\in \widebar{Y}_s$, we have $(\bm q^s)^\top \bm y \geq \Q_s(\bm x)$ by the definition of the local recourse function. If $\bm x^k=\hatx$ for all $k\in K$, then we have $\mathcal{L}_s(-\rho\bm{1},-\rho\bm{1};\hatx)\geq \Q_s(\hatx)$ for any positive $\rho^*$. If not, let $d=\min\{||\bm x^k-\hatx||_1: k\in K, \bm x^k\neq \hatx\}$ and let $\rho^* = \frac{\Q_s(\hatx)-L}{d}$. Then, for any optimal $\omega_{\rho^*}$, we have $\omega_{\rho^*} \leq \frac{\Q_s(\hatx)-L}{\rho^*}\leq d$, according to \eqref{weak_dual}. Therefore, we have
\begin{align*}
\mathcal{L}_s(-\rho^*\bm{1},-\rho^*\bm{1};\hatx)=&\min_{k\in K}\left\{(\bm q^s)^\top \bm y^k +\rho^*\omega_k: \omega_k\leq d\right\}
=\min_{k\in K}\{(\bm q^s)^\top \bm y^k +\rho^*\omega_k: \omega_k \in \{0,d\}\}
\geq \Q_s(\hatx),
\end{align*}
where the first equality follows since restricting $\omega_k\leq d$ preserves all optimal solutions, and the last inequality follows since when $\omega_k=0$, we have $(\bm q^s)^\top \bm y^k +\rho^*\omega_k=\Q_s(\hatx)$, and when $\omega_k = d$, $(\bm q^s)^\top \bm y^k +\rho^*\omega_k \geq L + \frac{\Q_s(\hatx)-L}{d}d = \Q_s(\hatx)$.
Therefore, the optimal value of \eqref{gen_lag_dual_ori} is $\Q_s(\hatx)$ with an optimal solution ${\bm\pi^+}^*={\bm\pi^-}^*=-\rho^*\bm{1}$.
\qed
\end{proof}

\subsection{Proof of \Cref{coro:expand_feas}}
\label{pf:coro_exp_feas}
\expfeas*
\begin{proof}
Since $({\bm\pi^+}^*, {\bm\pi^-}^*)$ is optimal to $\sup_{\bm\pi^+,\bm\pi^-\in\Re^n}\inf_{\bm x\in S}L_s(\bm x, \bm\pi^+,\bm\pi^-;\hatx)$, we have $\inf_{\bm x\in S}L_s(\bm x, \bm\pi^+,\bm\pi^-;\hatx)=\Q_s(\hatx)$. Taking the infimum over a subset $\barx \subseteq S$, we have $\inf_{\bm x\in \barx}L_s(\bm x, \bm\pi^+,\bm\pi^-;\hatx) \geq \inf_{\bm x\in S}L_s(\bm x, \bm\pi^+,\bm\pi^-;\hatx)=\Q_s(\hatx)$. At the same time, we have $L_s(\hatx, \bm\pi^+,\bm\pi^-;\hatx)=\Q_s(\hatx)$ and $\hatx\in\barx$. Thus, we can conclude that $\inf_{\bm x\in \barx}L_s(\bm x, \bm\pi^+,\bm\pi^-;\hatx)=\Q_s(\hatx)$. This implies that $({\bm\pi^+}^*, {\bm\pi^-}^*)$ is also optimal when restricting $\bm x\in\barx$. 
\qed
\end{proof}

\subsection{Proof of \Cref{coro:gen_lag_opt_cond}}
\label{pf:opt_cond}
\optcond*
\begin{proof}
Given $\hatx\in\barx$ and $(\hat{\bm\pi}^+,\hat{\bm\pi}^-)\in\Re^{2n}$, if the condition \eqref{gen_lag_opt_cond} is satisfied, we have $\mathcal{L}_s(\hat{\bm\pi}^+,\hat{\bm\pi}^-;\hatx)= \inf_{\bm x}\left\{\Q_s(\bm x)-\sum_{i\in[n]}\hat{\pi}^+_i(x_i-\hat{x}_i)^+ - \sum_{i\in[n]}\hat{\pi}^-_i(x_i-\hat{x}_i)^-:\bm x\in \barx\right\}\geq \Q_s(\hatx)$. Meanwhile, if $(\hat{\bm\pi}^+,\hat{\bm\pi}^-)$ is optimal to \eqref{gen_lag_dual_ori}, we have $\mathcal{L}_s(\hat{\bm\pi}^+,\hat{\bm\pi}^-;\hatx)=\Q_s(\hatx)$ according to the strong duality shown in \Cref{thm:gen_lag_dual}. Therefore,  we have \begin{align*}
\Q_s(\hatx) &= \mathcal{L}_s(\hat{\bm\pi}^+,\hat{\bm\pi}^-;\hatx) =\mathcal{L}_s(\bm\pi^+,\bm\pi^-;\hatx)= \inf_{\bm x}\left\{\Q_s(\bm x)-\sum_{i\in[n]}\hat{\pi}^+_i(x_i-\hat{x}_i)^+ - \sum_{i\in[n]}\hat{\pi}^-_i(x_i-\hat{x}_i)^-:\bm x\in \barx\right\} \\
&\leq \Q_s(\bm x) - \sum_{i\in[n]}\hat{\pi}^+_i(x_i-\hat{x}_i)^+ - \sum_{i\in[n]}\hat{\pi}^-_i(x_i-\hat{x}_i)^-
\end{align*}
for all feasible $\bm x\in\barx$.
\qed
\end{proof}

\subsection{Proof of \Cref{prop:conv_gen_lshaped}}
\label{pf:lp_conv}
\lpconv*
\begin{proof}
Let us denote the right-hand side of \eqref{s3} by $S_2$, which is a continuous relaxation of $S_1$. Hence, we must have $S_1\subseteq S_2$.

To prove $S_1\supseteq S_2$, we first observe that $\conv(S_1)= \conv\{(\bm x,\theta): 
\exists (\bm\omega^+, \bm\omega^-,\bm z)\in\Re^n\times\Re^n\times\{0,1\}^n, 
\eqref{mip_rep_cut},\eqref{mip_rep_omega},0\leq x_i\leq B_i,\forall i\in[n]\}$.
For any $(\bm x,\theta)\in S_2$, there exists a solution $(\bm\omega^+,\bm\omega^-,\bm z)$ such that \eqref{s3} holds. Without loss of generality, we assume
$z_0:= 1 \geq z_1\geq \ldots\geq z_n$. To show $(\bm x,\theta)\in\conv(S_1)$, we construct $n+1$ points $(\bm x^k, \theta_k)\}_{k\in [0,n]}$ as follows. For $k\in [0,n]$, let $\lambda_k = z_k-z_{k+1}$, and define $\bm z^k\in\{0,1\}^n$ as 
$$z_i^k = \left\{\begin{aligned}
&1, &i\leq k,\\
&0, &i> k,
\end{aligned}\right.$$
for all $i\in[n]$.
Define $(\bm\omega^+)^k,(\bm\omega^-)^k\in\Re^n$ as follows:
If $z_k = 1$, let 
$$(\omega^+)^k_i = \left\{
\begin{aligned}
&\omega^+_i, &i\leq k,\\
&0, &i> k,
\end{aligned}
\right. \text{ and } (\bm\omega^-)^k=\bm{0}.$$
If $0< z_k <1$, let 
$$(\omega^+)^k_i = \left\{
\begin{aligned}
&\frac{\omega^+_i}{z_i}, &i\leq k,\\
&0, &i> k,
\end{aligned}
\right. \text{ and } 
(\omega^-)^k_i = \left\{
\begin{aligned}
&0, &i\leq k,\\
&\frac{\omega^-_i}{1-z_i}, &i> k.
\end{aligned}
\right.
$$
If $z_k = 0$, let 
$$(\bm\omega^+)^k = \bm{0} \text{ and } (\omega^-)^k_i = \left\{
\begin{aligned}
&0, &i\leq k,\\
&\omega^-_i, &i> k.
\end{aligned}
\right.$$
With $\bm x^k$ defined as $\bm x^k = \hatx + (\bm\omega^+)^k - (\bm\omega^-)^k$, the above formulations ensure that constraints \eqref{mip_rep_omega} are satisfied. To satisfy condition \eqref{mip_rep_cut}, we let $\theta^k = \theta - \sum_{i\in[n]}(\pi_i^+\omega^+_i+\pi_i^-\omega^-_i)+\sum_{i\in[n]}(\pi_i^+(\omega^+)^k_i+\pi_i^-(\omega^-)^k_i)$. Thus, we have $\theta^k - \sum_{i\in[n]}(\pi_i^+(\omega^+)^k_i+\pi_i^-(\omega^-)^k_i) = \theta - \sum_{i\in[n]}(\pi_i^+\omega^+_i+\pi_i^-\omega^-_i) \geq \Q_s(\hatx)$. It is easy to verify that $(\bm x,\theta, \bm\omega^+,\bm\omega^-,\bm z) = \sum_{k\in [0,n]}\lambda_k (\bm x^k, \theta_k, (\bm\omega^+)^k,(\bm\omega^-)^k, \bm z^k)$. Thus, $(\bm x, \theta)\in \conv(S_1)$. Consequently, we have $\conv(S_1)=S_2$.\qed
\end{proof}

\subsection{Proof of \Cref{prop:lag_relu}}
\label{pf:lag}
\lag*
\begin{proof}
Given $\hatx\in\barx$ and $\bm\pi\in\Re^n$ such that $\mathcal{L}(\bm\pi;\hatx)=\Q_s(\hatx)$, we can derive a tight Lagrangian cut $\theta\geq\Q_s(\hatx)+\bm\pi^\top(\bm x-\hatx)$. It is equivalent to $\theta \geq \Q_s(\hatx) + \sum_{i\in[n]} \pi_i(x_i-\hat{x}_i)^+ - \sum_{i\in[n]} \pi_i(x_i-\hat{x}_i)^-$, which is a ReLU Lagrangian cut according to \Cref{coro:gen_lag_opt_cond}.
\qed
\end{proof}

\subsection{Proof of \Cref{prop:dual_poly}}
\label{pf:poly}
\poly*
\begin{proof}
According to \Cref{def_pihat}, from the optimality condition \eqref{gen_lag_opt_cond}, we have
\begin{align*}
\pixhat = 
&\left\{(\bm\pi^+, \bm\pi^-)\in\Re^{2n}: \min_{\bm x}\left\{\Q_s(\bm x)-\sum_{i\in[n]}\pi^+_i(x_i-\hat{x}_i)^+-\sum_{i\in[n]}\pi_i^-(x_i-\hat{x}_i)^-: \bm x\in\barx\right\}\geq \Q_s(\hatx)\right\}\\
=&\left\{(\bm\pi^+, \bm\pi^-)\in\Re^{2n}: \min_{\bm x,\bm\omega^+,\bm\omega^-,\bm z}\left\{\Q_s(\bm x)-(\bm\pi^+)^\top \bm\omega^+-(\bm\pi^-)^\top \bm\omega^-: \bm x\in\barx, \eqref{mip_rep_omega},\eqref{mip_rep_z}\right\}\geq \Q_s(\hatx)\right\}\\
=&\left\{(\bm\pi^+, \bm\pi^-)\in\Re^{2n}: \min_{\bm x,\bm\omega^+,\bm\omega^-,\bm z,\bm y}\left\{(\bm q^s)^\top \bm y-(\bm\pi^+)^\top \bm\omega^+-(\bm\pi^-)^\top \bm\omega^-: (\bm x, \bm y)\in\widebar{Y}_s, \eqref{mip_rep_omega},\eqref{mip_rep_z}\right\}\geq \Q_s(\hatx)\right\}\\
=&\left\{(\bm\pi^+, \bm\pi^-)\in\Re^{2n}: ({\bm\omega^+}^k)^\top \bm\pi^+ + ({\bm\omega^-}^k)^\top \bm\pi^- \leq (\bm q^s)^\top \bm y^k-\Q_s(\hatx),\forall k\in K \right\},
\end{align*}
where $\{(\bm x^k,{\bm\omega^+}^k,{\bm\omega^-}^k,\bm z^k,\bm y^k)\}_{k\in K}$ are all extreme points of the polyhedron $\conv\{(\bm x,\bm\omega^+,\bm\omega^-,\bm z,\bm y): (\bm x, \bm y)\in\widebar{Y}_s, \eqref{mip_rep_omega}, \eqref{mip_rep_z}\}$. This implies that $\pixhat$ is defined by finitely many linear inequalities and, therefore, is a polyhedron.
\qed
\end{proof}

\subsection{Proof of \Cref{lemma:boundary}}
\label{pf:linsp}
\linsp*
\begin{proof}
Given $\hatx\in X$, for any binary $\bm x$, we have $\omega^+_i = (x_i-\hat{x}_i)^+=0$ for all $i\in \ix$ and $\omega^-_i = (x_i-\hat{x}_i)^-=0$ for all $i\notin \ix$. Thus, $({\bm\omega^+}^k)^\top \bm e_i + ({\bm\omega^-}^k)^\top \bm{0} = 0$ for all $i\in \ix$ and $({\bm\omega^+}^k)^\top \bm{0} + ({\bm\omega^-}^k)^\top \bm e_i = 0$ for all $i\notin \ix$.
\qed
\end{proof}

\subsection{Proof of \Cref{polar}}
\label{pf:polar}
\polar*
\begin{proof}
For a given $\hatx\in\barx$,
from the optimality condition \eqref{opt_cond_general}, we have
\begin{align*}
\dxhat 
=\left\{\bm \pi\in\Re^n: \min_{\bm  x,\theta}\left\{\theta+\bm\pi^\top (\hatx - \bm x):(\bm x,\theta)\in \conv(\epi_{\barx}(\Q_s))\right\}\geq \co(\Q_s)(\hatx) \right\},
\end{align*}
where $\conv({\epi_{\barx}(\Q_s)}) = \conv\{(\bm x^k, \Q(\bm x^k)): k\in K\}+\cone\{(\bm{0},1)\}$, and $K$ is finite since $\conv({\epi_{\barx}(\Q_s)})$ is a polyhedron.
We further obtain that
\begin{equation}\label{pixhat}
\dxhat= \left\{\bm\pi\in\Re^n:\Q_s(\bm x^k)+\bm \pi^\top(\hatx-\bm x^k) \geq \co(\Q_s)(\hatx), \ \forall k\in K, \ 1-\bm \pi^\top \bm{0} \geq 0 \right\}.
\end{equation}
which involves finitely many linear inequalities and, therefore, is a polyhedron.

Suppose that set $\conv{(\barx)}$ is full-dimensional. Then $\conv{(\epi_{\barx}(\Q_s))}$ has dimension $n+1$.
A cut $\theta\geq \co(\Q_s)(\hatx)+\bm\pi^\top (\bm x-\hatx)$ is facet-defining if and only if it is tight at $n+1$ extreme points of $\conv{(\epi_{\barx}(\Q_s))}$ that are affinely independent. That is, $\bm \pi$ satisfies 
 $n$ linearly independent inequalities with equality in \eqref{pixhat},  meaning that it is an extreme point of $\dxhat$. 
 When $\barx$ is not full dimensional, we can reduce it to a lower dimensional space where the same result holds.

From the optimality condition \eqref{opt_cond_general}, we have
\begin{align*}
\dxhat %=
=&\left\{\bm \pi\in\Re^n: \bm\pi^\top (\hatx-\bm x)\geq \mathrm{co}(\Q_s)(\hatx)-\Q_s(\bm x), \forall \bm x\in \barx \right\}.
\end{align*}
Thus, $\reccone(\dxhat) = \left\{\bm \pi\in\Re^n: \bm\pi^\top (\hatx-\bm x)\geq 0, \forall \bm x\in \barx \right\} = \left\{\bm \pi\in\Re^n: \bm\pi^\top (\hatx-\bm x)\geq 0, \forall \bm x\in \conv(\barx) \right\}=N_{\conv(\barx)}(\hatx)$.
\qed
\end{proof}

\subsection{Proof of \Cref{prop:-a}}
\label{pf:-a}
\revtag*
\begin{proof}
Note that LP \eqref{stren} is always feasible with $\bm{\eta}=\bm{0}$ being a feasible solution. %Suppose that it is unbounded for some $-\bm a\in \mathcal{T}_{\conv(X)}(\hatx)$. Then 
Hence, the strengthening problem \eqref{stren} is bounded if and only if its dual problem
\begin{align*}
\max \left\{\sum_{k\in K}(\Q_s(\hatx)-\Q_s(\bm x^k)-\hat{\bm \pi}^\top(\hatx-\bm x^k))y_k:
\sum_{k\in K} (\hatx -\bm x^k)y_k = \bm a,\bm y \geq \bm{0}\right\},
\end{align*}
where $\{\bm x^k\}_{k\in K}=X$, is feasible. Therefore, by definition, the strengthening problem \eqref{stren} is bounded if and only if $-\bm a\in \mathcal{T}_{\conv(X)}(\hatx)$. 
\qed
\end{proof}

\subsection{Proof of \Cref{prop:feas_eta}}
\label{pf:shrink}
\shrink*
\begin{proof}
    Since $\Q^{LP}_s(\bm x)$ is obtained by solving the LP relaxation of the local recourse problem, we have $\Q^{LP}_s (\bm x) \leq \Q_s (\bm x)$ for all $\bm x\in X$. Thus, for any $\bm\eta$ feasible for \eqref{stren_lp}, we have $\min_{\bm x}\left\{ \Q_s(\bm x)+(\hat{\bm\pi}+\bm\eta)^\top(\hatx-\bm x):\bm x\in X\right\}\geq \min_{\bm x}\left\{ \Q^{LP}_s(\bm x)+(\hat{\bm\pi}+\bm\eta)^\top(\hatx-\bm x):\bm x\in X\right\} \geq \min_{\bm x}\left\{ \Q^{LP}_s(\bm x)+(\hat{\bm\pi}+\bm\eta)^\top(\hatx-\bm x):\bm x\in X^{LP}\right\}\geq \Q_s(\hatx)$. Therefore, $\bm\eta$ is also feasible for \eqref{stren}.
    \qed
\end{proof}

\subsection{Proof of \Cref{prop:nogood}}
\label{pf:nogood}
\nogood*
\begin{proof}
According to \eqref{F_eta_lp}, we have
\begin{align*}
F_{s}^{LP}=&\left\{\bm\eta:\Q_s^{LP}(\bm x)+(\hat{\bm\pi}+\bm\eta)^\top(\hatx-\bm x)\geq \Q_s(\hatx), \forall \bm x\in\widetilde{X}\right\}\\
=&\left\{
\bm\eta:
\begin{pmatrix}
\bm x^\top & \theta
\end{pmatrix}
\begin{pmatrix}
 -(\hat{\bm \pi}+\bm{\eta}) \\ 1
\end{pmatrix}
\geq 
\begin{pmatrix}
\hatx^\top & \Q_s(\hatx)
\end{pmatrix}
\begin{pmatrix}
-(\hat{\bm\pi}+\bm\eta) \\
1
\end{pmatrix},\forall (\bm x,\theta)\in \epi_{\widetilde{X}}\Q_s^{LP}
\right\}.
\end{align*}
Note that
$$\epi_{\widetilde{X}}(\Q_s^{LP})=\left\{(\bm x, \theta)\in X^{LP}\times \Re : \exists \bm y\in \Re^{m_1+m_2}:
\theta \geq (\bm q^s)^\top \y,  \bm T^s\bm x+\bm W^s\bm y \geq \bm h^s,
\bm\chi^\top(\hatx-\bm x)\geq 1
\right\}
$$
is a polyhedron and, therefore, closed and convex.
Since $(\hatx^\top, \Q_s(\hatx))\notin \epi_{\widetilde{X}}\Q_s^{LP}$,
by the separation theorem, there exists $\bm\alpha,\beta,\gamma$ such that 
$\bm\alpha^\top\hatx+\beta\Q_s(\hatx)<\gamma$ and
$\bm\alpha^\top\bm x+\beta\theta\geq \gamma$ for all $(\bm x,\theta)\in \epi_{\widetilde{X}}(\Q_s^{LP})$. If $\beta\neq 0$, since $(\bm{0},1)$ is an extreme ray of $\epi_{\widetilde{X}}(\Q_s^{LP})$, we have $\beta>0$.
Then, $-\bm\alpha/\beta-\hat{\bm\pi}\in F_{s}^{LP}$.
If $\beta=0$, let 
$$\widebar{\beta} = \frac{\gamma-\bm\alpha^\top\hatx}{2\max\{|L|,\Q_s(\hatx),1\}}>0.$$ Then we have a separation $\bm\alpha^\top \bm x+\widebar{\beta}\theta \geq \bm\alpha^\top\bm x+\widebar{\beta}L\geq \gamma - \frac{\gamma-\bm\alpha^\top\hatx}{2} = \frac{\gamma+\bm\alpha^\top\hatx}{2}$ and 
$\bm\alpha^\top\hatx+\widebar{\beta}\Q_s(\hatx)\leq \bm\alpha^\top\hatx+\frac{\gamma-\bm\alpha^\top\hatx}{2}=\frac{\gamma+\bm\alpha^\top\hatx}{2}$. Hence, $-\bm\alpha/\widebar{\beta}-\hat{\bm\pi}\in F_{s}^{LP}$.
\qed
\end{proof}

\subsection{Proof of \Cref{prop:lam_l}}
\label{pf:lam_l}
\lam*
\begin{proof}
According to \Cref{prop:conv_gen_lshaped}, we have
\begin{align*}
\conv(E_\Lambda^s)=\left\{(\bm x,\theta)\in \times_{i\in[n]}[0, B_i]\times\Re:
\begin{aligned}
&\exists (\bm\omega^+, \bm\omega^-,\bm z)\in\Re^n\times\Re^n\times[0,1]^n,\\
&\theta \geq \Q_s(\hatx) + (L-\Q_s(\hatx))\sum_{i\in[n]}(\omega_i^+ + \omega_i^-),\label{mip_rep_cut}\\
&\omega_i^+ - \omega_i^- = x_i-\hat{x}_i,0\leq \omega_i^+\leq (B_i-\hat{x}_i)z_i,\\
&0\leq \omega_i^-\leq \hat{x}_i(1-z_i),\forall i\in[n] 
\end{aligned}\right\}.
\end{align*}
it suffices to show that $\conv(E_\Lambda^s)\supseteq E_L^s$. We observe that $(\bm x,\Q_s(\hatx)+(L-\Q_s(\hatx))\sum_{i\in[n]}\max\{B_i-\hat{x}_i, \hat{x}_i \}) \in \conv(E_\Lambda^s)$ for all $\bm x\in \mathcal{B}$ by letting $z_i=1$ if $B_i-\hat{x}_i>\hat{x}_i$ and 0, otherwise, 
% . This follows because the system:
% $$%\left\{
% \begin{aligned}
% \omega_i^+  - \omega_i^- = x_i - \hat{x}_i,
% \omega_i^+  + \omega_i^- = \max\{B_i-\hat{x}_i, \hat{x_i}\},
% 0\leq \omega_i^+ \leq (B_i-\hat{x}_i)z_i,
% 0\leq \omega_i^- \leq \hat{x}_i(1-z_i),
% 0\leq z_i\leq 1
% \end{aligned}
% %\right.
% $$
for all $i\in[n]$. It remains to show that 
\begin{claim}\label{claim_delta}
$||\bm \delta^i-\hat{\bm\delta}^i|| \leq \max\{B_i-\hat{x}_i, \hat{x_i}\}$ for all $\bm x\in \mathcal{B}$ and $i\in[n]$. 
\end{claim}
According to Claim \ref{claim_delta}, for any $(\bm x,\theta)\in E_L^s$, we have $\theta \geq \Q_s(\hatx) + (L-\Q_s(\hatx))\sum_{i\in[n]}||\bm\delta^i-\hat{\bm\delta}^i||_1 \geq \Q_s(\hatx)+(L-\Q_s(\hatx))\sum_{i\in[n]}\max\{B_i-\hat{x}_i, \hat{x}_i \}$. Hence, $(\bm x,\theta)\in \conv(E_\Lambda^s)$ and  $E_L^s \subseteq \conv(E_\Lambda^s)$.

To prove Claim \ref{claim_delta}, notice that since $B_i$ is an integer and $N_i=\lfloor \log_2B_i\rfloor$, we have $\max\{B_i-\hat{x}_i, \hat{x}_i \} \geq \lceil B_i/2\rceil \geq 2^{N_i-1}$. On the other hand, we observe that $||\bm\delta^i-\hat{\bm\delta}^i||_1 \leq N_i+1$. There are three cases:
\begin{enumerate}[{Case }1:]
\item   $B_i \geq 5$ or $B_i=3$. In this case, we have $\lceil B_i/2\rceil \geq N_i+1$. 
\item   $B_i = 4$, if $\hat{x}_i = 0,1,3,4$, we have $||\bm\delta^i-\hat{\bm\delta}^i||_1 \leq 3\leq\max\{B_i-\hat{x}_i, \hat{x}_i \}$. 
\item  $B_i = 4$ and $\hat{x}_i = 2$. In this case, we have $\max\{B_i-\hat{x}_i, \hat{x}_i \}=2$, $\hat{\bm\delta}^i=(0,1,0)^\top$, and $||\bm\delta^i-\hat{\bm\delta}^i||_1\leq 3$. Note that $||\bm\delta^i-\hat{\bm\delta}^i||_1=3$ if and only if $\bm\delta^i=(1,0,1)^\top$, i.e., $x_i=5$, contradicting $x_i\leq B_i=4$. Hence, we have $||\bm\delta^i-\hat{\bm\delta}^i||_1 \leq 2$.
\end{enumerate}
Therefore, we must have $\max\{B_i-\hat{x}_i, \hat{x}_i \}\geq ||\bm\delta^i-\hat{\bm\delta}^i||_1$ if $B_i \geq 3$ for any $i\in [n]$.
%This completes the proof.
\qed
\end{proof}
\subsection{Proof of \Cref{prop:conv_box}}
\label{pf:conv_box}
\convbox*
\begin{proof}
Let $D_i=\left\{\bm\delta^i\in\{0,1\}^{N_i+1}:\sum_{j\in [0,N_i]}2^j\delta^i_j\leq B_i\right\}$ for each $i\in [n]$. Then, according to the decomposition structure of set $E_L^s$, we have
\begin{align*}
\conv(E_L^s) = \left\{
(\bm x,\theta)\in \times_{i\in[n]}[0, B_i]\times\Re:
\begin{aligned}
&\exists\bm\delta \in [0,1]^{N_1+1}\times\cdots\times[0,1]^{N_n+1},\eqref{bi_lshaped},\\
&x_i=\sum_{j\in [0,N_i]}2^j \delta^i_j,  \forall i\in[n],\\
&\bm\delta^i \in \conv(D_i) ,\forall i\in[n] .
\end{aligned}\right\},
\end{align*}

On the other hand, according to \cite{laurent1992characterization, gupte2013solving}, the convex hulls of the $D_i$'s can be described using their knapsack structure.
Suppose that $B_i = 2^{j_{1i}}+\cdots+2^{j_{\ell i}}+2^{N_i}$, and let $J_i = \{j_{1i},\ldots,j_{\ell i},N_i\}$ for each $i\in[n]$. Then, we have
\begin{equation*}
\conv(D_i) = \left\{\bm\delta^i\in[0,1]^{N_i+1}:\delta^i_r+\sum_{\tau\in J_{ir}}\delta^i_{\tau}\leq |J_{ir}|,\forall r\in\{0,1,\ldots, N_i\}\backslash J_i\right\},
\end{equation*}
where $J_{ir}=\{\ell\in J_i:\ell > r\}$.  
%Moreover, $\conv(D) = \{(\bm \delta, \theta)\in \conv(D_1\times\cdots\times D_n)\times\Re:  \eqref{bi_lshaped}\} = \{(\bm \delta, \theta)\in \conv(D_1)\times\cdots\times \conv(D_n)\times\Re:  \eqref{bi_lshaped}\}$. 
This completes the proof.
\qed
\end{proof}
\end{appendices}

\end{document}